\documentclass[12pt]{article}

\usepackage{amssymb,amsmath,amsthm}
\usepackage{mathtools}
\usepackage{graphicx,epstopdf}
\usepackage[usenames]{color}
\usepackage{url,hyperref}
\usepackage{algorithm,algorithmic}
\usepackage{verbatim}
\usepackage[utf8]{inputenc}
\usepackage{siunitx}

\usepackage[margin=2.5cm]{geometry}

\graphicspath{{figures/}}

\usepackage{enumitem}
\setlist[enumerate]{leftmargin=.5in}
\setlist[itemize]{leftmargin=.5in}

\newtheorem{master}{}
\numberwithin{master}{section}
\newtheorem{theorem}[master]{Theorem}
\newtheorem{lemma}[master]{Lemma}
\newtheorem{corollary}[master]{Corollary}
\newtheorem{proposition}[master]{Proposition}

\theoremstyle{definition}

\theoremstyle{remark}
\newtheorem{remark}{Remark}
\numberwithin{remark}{section}

\newcommand{\R}{\mathbb{R}}
\makeatletter \renewcommand\d[1]{%
	\mathop{}\!\mathrm{d}#1}
\makeatother
\newcommand{\supp}{\operatornamewithlimits{supp}}
\newcommand{\argmin}{\operatornamewithlimits{arg\,min}}
\newcommand{\dual}[1]{\left\langle #1 \right\rangle}
\newcommand{\ran}{\operatornamewithlimits{Ran}}

\newcommand\ccrule[3][black]{\textcolor[RGB]{#1}{\rule{#2}{#3}}}

\newcommand\eit{\emph{Electrical Impedance Tomography} (EIT)\gdef\eit{EIT}}
\newcommand\aet{\emph{Acousto-Electric Tomography} (AET)\gdef\aet{AET}}
\newcommand\ct{computed tomography (CT)\gdef\ct{CT}}
\newcommand\mri{magnetic resonance imaging (MRI)\gdef\mri{MRI}}
\newcommand\pde{partial differential equation (PDE)\gdef\pde{PDE}}

\newcommand{\matlab}{\textsc{Matlab}}
\newcommand{\kwave}{\textsc{k-Wave}}

\newcommand{\fenics}{\textsc{FEniCS}}

\title{Sound speed uncertainty in Acousto-Electric Tomography}

\date{June 2021}

\newif\ifimg    
\imgtrue        
\newcommand{\gfx}[2]{\ifimg\includegraphics[#1]{#2}\else\includegraphics[#1,draft]{#2}\fi}

\author{
Bj{\o}rn Christian Skov Jensen$^{1,*}$  and Kim Knudsen$^2$ \\ 
\footnotesize $^1$Department of Mathematics and Statistics, \\[-10pt]
\footnotesize University of Helsinki, 00560 Helsinki, Finland \\
\footnotesize $^2$Department of Applied Mathematics and Computer Science, \\[-10pt]
\footnotesize Technical University of Denmark, 2800 Kgs.\ Lyngby, Denmark \\
\footnotesize $^{*}$Corresponding author: bjorn.jensen@helsinki.fi
}

\begin{document}

\maketitle

\begin{abstract}
The goal in {A}cousto-{E}lectric Tomography (AET) is to reconstruct an image of the unknown electric conductivity inside an object from  boundary measurements of electrostatic currents and voltages collected while the object is penetrated by propagating ultrasound waves. This problem is a coupled-physics inverse problem.
Accurate knowledge of the propagating ultrasound wave is usually assumed and required, but in practice tracking the propagating wave is hard due to inexact knowledge of the interior acoustic properties of the object. In this work, we model uncertainty in the sound speed of the acoustic wave, and formulate a suitable reconstruction method for the interior power density and conductivity. 
We also establish theoretical error bounds, and show that the suggested approach can be understood as a regularization strategy for the inverse problem. 
Finally, we numerically simulate the sound speed variations from a numerical breast tissue model, and computationally explore the effect of using an inaccurate sound speed on the error in reconstructions. Our results show that with reasonable uncertainty in the sound speed reliable reconstruction is still possible. 

\vspace{1pc}

\vspace{1pc}
\end{abstract}

\section{Introduction}
\eit{} \cite{cheney1999a,holder2010a} is a well established technology for imaging of the interior electrical conductivity in a body or object from electro-static surface measurements. Applications in medical imaging include early detection of breast cancer \cite{zou2003a}, bedside monitoring of the lung function \cite{reinartz2019a}, and  hemorrhagic stroke detection \cite{malone2014a}. The inverse problem in \eit{} is highly ill-posed due to the diffusive nature of electric signals, and consequently reconstructed images have low resolution. A potential remedy can be found in hybrid imaging techniques that exploit the interplay between different physical phenomena.

\aet{} \cite{zhang2004a} is a hybrid imaging technique combining the electro-static boundary measurements of \eit{} with  ultrasound. The aim is to provide tomographic images having much better contrast and resolution compared to images produced by \eit{} alone. 

The experimental procedure is as follows: An ultrasound wave is emitted by a transducer and propagates through the object. The acoustic pressure causes local contractions and expansions, and these small volume deformations induce a slight change of conductivity; this effect is referred to as \emph{acousto-electric interaction}, \emph{acousto-electric modulation} or the \emph{acousto-electric effect} \cite{jossinet1998phenomenology,lavandier2000experiemental}. As a consequence of the change in conductivity the electro-static boundary measurements change as well, and the inverse problem is then to obtain the conductivity from these acoustically excited EIT measurements.

Our work in this paper is inspired by the use of \aet{} for breast cancer imaging and parameters in simulations are chosen accordingly. The study is, however, conceptual and computational rather than data driven.

\aet{} assumes complete knowledge of the propagating waves through the object, and  this, in turn, requires complete knowledge of the sound speed of the object. However, in many applications, the sound speed is not fully known; only a rough estimate of the magnitude is provided and can be used in the inversion framework. For instance, in breast tissue it is reasonable to model the sound speed to be constant with spatial variations up to 10\%  \cite{duck2013physical}.

In this paper we focus on the following question: {\it To what extent can we trust  \aet{}, when the sound speed is uncertain.} To answer the question we adapt the reconstruction framework of \cite{jensen2019feasibility} (with few and important updates) that relies on a decomposition of the inverse problem into two separate optimization problems. We show theoretically that the approach is consistent in the sense of classical regularization theory: in the limit of vanishing uncertainty and errors we perfectly solve the problem. Moreover, we numerically model the uncertainty in the sound speed. Inspired by potential applications in breast cancer detection, we adapt the tissue model in \cite{reiser2013validation} to numerically produce such sound speed variations. Then we use the complete inversion framework to quantify the impact of sound speed uncertainty on the interior power density and conductivity reconstructions.  

 AET was considered experimentally in \cite{zhang2004a}, however, since the measurements have very low signal-to-noise ratio, the technology is still in its infancy and many technological challenges need to be solved. Mathematically, the problem is fairly well understood  \cite{Ammari_Bonnetier_Capdeboscq_Tanter_Fink_2008,BalGuoMonard:2014,Bal2012,Capdeboscq_Fehrenbach_Gournay_Kavian_2009,Kuchment_Kunyansky_2010,Kuchment_Kunyansky_2011,monard2012inverse} and several numerical algorithms have been discussed \cite{jensen2018acousto,adesokan2019nonlinear,bal2013inverse,hubmer2018limited,li2019levenbergmarquardt,Roy-Borzi}. For an introduction to the mathematical theory pertaining to both \aet{} and related problems we refer to the book \cite{alberti2018lectures}. While to the best of our knowledge \aet{} has not seen any study related to uncertainty quantification prior to this work, we should mention that there is some work on the unknown sound speed problem \cite{oksanen2014photoacoustic,Tick_2019} for related hybrid problems \emph{Photo-Acoustic tomography} and \emph{Thermo-Acoustic tomography}. We believe that the quantitative approach developed here applies to these modalities as well.

The outline of this paper is as follows. In Section 2 we describe the \aet{} model, including the sound speed modeling. In Section 3 we obtain continuity results for the \aet{} data with respect to the sound speed variations. In Section 4 we recall the optimization problem formulations of the involved inverse problems; first recovery of the power density and from there recovery of the conductivity. We further prove that our approach forms a proper regularization strategy. In Section 5 we describe the numerical implementation of the forward models. Moreover, we describe the procedure used for generating random sound speed samples with structures for our numerical computations. In section 6 we show reconstructions and describe the numerical results, and Section 7 contains discussion and concluding remarks.

\section{Modeling Acousto-Electric Tomography}\label{sec:modeling}
The modeling of \aet{} follows \cite{jensen2019feasibility}. In $\R^d,\; d =2,3,$ an ultrasound wave generated by a source $S(x,t)$ is modelled by the scalar wave equation
\begin{gather} \label{eq:p}
    \left\{\begin{aligned}
        \left(\partial_t^2 -c^2 \Delta\right)p &= S && \text{in $ \R^d \times \R_+ $}, \\
        p\vert_{t=0} = \partial_tp\vert_{t=0} &= 0 && \text{on $ \R^d $},
    \end{aligned}\right.
\end{gather}
where $ c(x)$ is the spatially dependent \emph{sound speed}.
We assume that the source $ S $ is fully known, smooth and compactly supported, and that $c \in C^\infty(\R^d)$ is bounded, and bounded from below, by a positive constant. (The smoothness assumptions are not essential nor optimal.) Then  \eqref{eq:p} has a unique weak solution, which given the smoothness of the coefficients is in fact $ C^\infty $, see e.g.\ \cite[Sec. 7.2]{evans2010partial}.

The electric conductivity is modeled by a real-valued function $ \sigma $ in a bounded and open subset $ \Omega \subset \R^d $ with smooth boundary $ \partial\Omega $. The function $ \sigma $ belongs to $ L^\infty(\Omega) $ and is bounded from below by a positive constant. When a current flux $ f $ is applied on $ \partial\Omega, $ an electric potential $u$ is generated in $\Omega,$ and assuming no interior sources or sinks of charge, the electrical potential $u $ satisfies the \pde{}
\begin{equation} \label{eq:u}
    \left\{\begin{aligned}
        -\nabla\cdot\sigma\nabla u &= 0 && \text{in $ \Omega $}, \\
        \sigma\partial_\nu u &= f && \text{on $ \partial\Omega $}.
    \end{aligned}\right.
\end{equation}
The vector $ \nu $ denotes the outward pointing unit normal on $ \partial\Omega $ and $ \sigma \partial_\nu u = \nu \cdot \sigma \nabla u $ is the normal component of the current field $J=\sigma\nabla u$. The compatibility condition $ f \in L_\diamond^2(\partial \Omega) = \{ v \in L^2(\partial\Omega) : \int_{\partial\Omega}v\d s = 0\} $ (corresponding to a conservation of charge) guarantees that \eqref{eq:u} has a weak solution $ u \in H^1(\Omega) $ unique up to a constant, which is fixed by selecting the solution $u$ for which $ u\vert_{\partial\Omega} = g \in L_\diamond^2(\partial\Omega) $ \cite{evans2010partial}, $ g $ describing the voltage potential at the boundary, which is measurable via electrodes.

When the wave $p$ propagates through $\Omega,$ the conductivity is perturbed due the acousto-electric effect. The  perturbed conductivity $\sigma_\ast (x,t)$, now temporally dependent, is described by the first order model \cite{jossinet1998phenomenology,lavandier2000experiemental}
\begin{equation} \label{eq:coupling}
    \sigma_\ast = \sigma(1+\eta p),
\end{equation}
where $ \eta > 0 $ is called the acousto-electric coupling constant and is assumed to be known.

Substituting $ \sigma_\ast $ for $ \sigma $ in \eqref{eq:u} yields the PDE
\begin{equation} \label{eq:up}
    \left\{\begin{aligned}
        -\nabla\cdot(\sigma_\ast\nabla u_\ast) &= 0 && \text{in $ \Omega $}, \\
        \sigma_\ast\partial_\nu u_\ast &= f && \text{on $ \partial\Omega $}.
    \end{aligned}\right.    
\end{equation}
characterizing for fixed $ t\in\R_+ $ the resulting time-dependent electrical potential $ u_\ast(x,t) $. Again, $ u_\ast $ is unique up to a (time-dependent) constant that is fixed by requiring $ u_\ast\vert_{\partial\Omega}(\cdot,t) = g_\ast(\cdot, t)  \in L^2_\diamond(\partial\Omega) $ for each $ t $; $ g_\ast $ being the boundary voltage potential for $ u_\ast $.

The inverse problem of AET is now to reconstruct $\sigma$ from knowledge of several triplets $(f,g,g_\ast)$ corresponding to different choices of $f$. As a first step in order to transform the boundary functions to interior information, the product of the current density $ f $ and the difference in boundary potentials $g_\ast-g$ is integrated along the boundary. This defines the time signal 
\begin{align*}
    I(t) &= \int_{\partial \Omega} f\cdot(g_\ast(\;\cdot\;,t) - g) \d s.
\end{align*}
The signal describes the time evolution of the difference in power for the system under the influence of the wave perturbation. The map $ I(t) $ is illustrated in the cartoon-like Figure \ref{fig:I}. The blue curve illustrates $I(t)$ for the depicted conductivity phantom with a circular inclusion. The gray curve is for reference $I(t)$ for the homogeneous background conductivity. Note that the signal difference is large when the wave is in contact with the inclusion.
\begin{figure}[bth]
    \centering
    \gfx{width=1.0\textwidth}{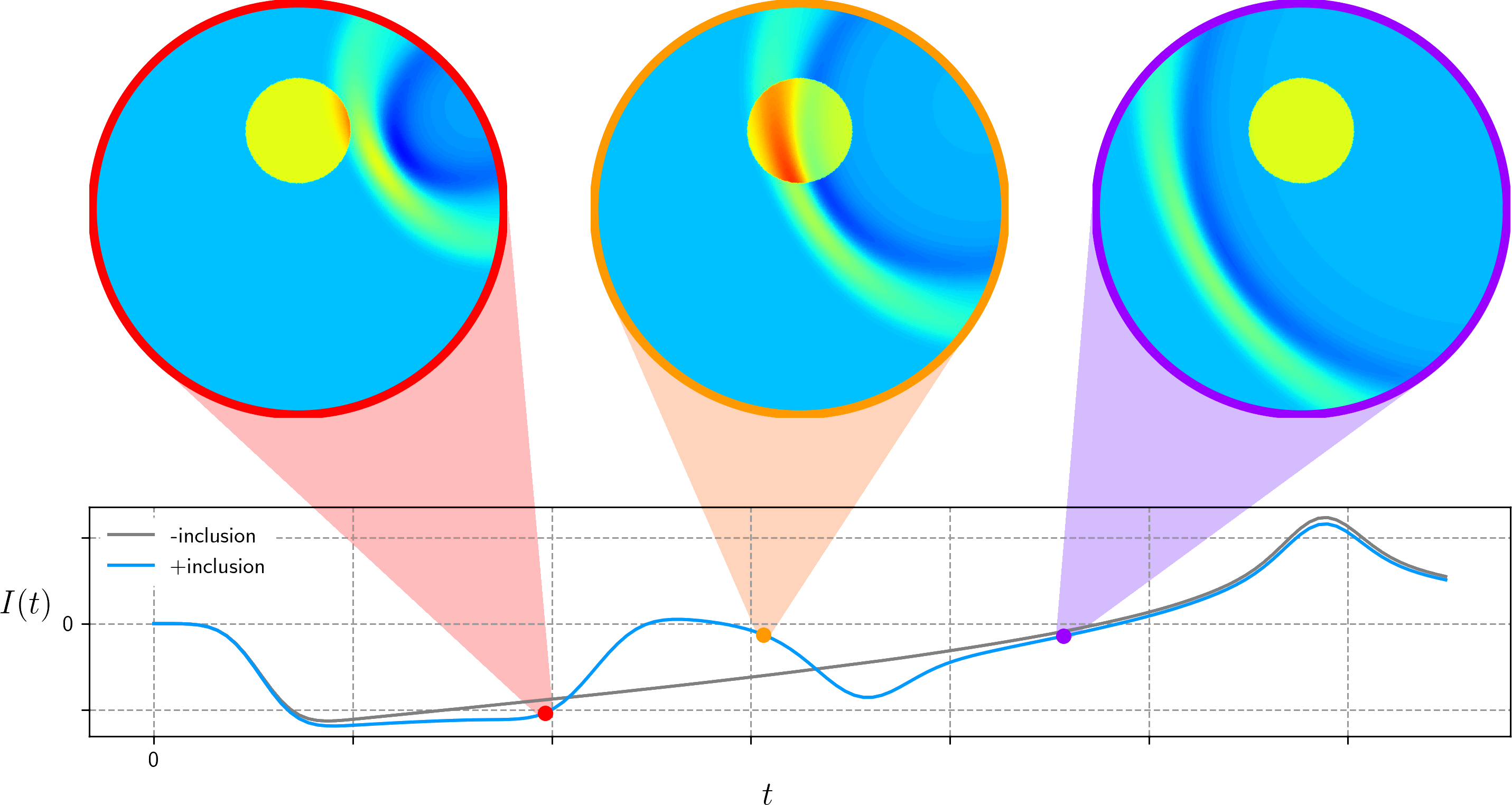}
    \caption{Upper plot illustrates the perturbed conductivity $\sigma_\ast$ due to a single propagating acoustic wave at key times. The lower plot illustrates the corresponding time signal $ I(t) $ for the conductivity phantom with a circular inclusion (blue) and for a homogeneous reference (gray).}
    \label{fig:I}
\end{figure}
Using the function $I$, we can now pose the inverse problem of AET as follows: Given $I$ for several boundary conditions $f$ and wave sources $S,$ reconstruct the conductivity $\sigma.$ Note that when multiple boundary currents $f$ are used we could also consider cross-terms by integrating $f$ to a $g-g_\ast$ coming from different boundary currents. This approach might stabilize the reconstruction problem; we will instead use three boundary conditions; more than the two required for reconstruction by theory \cite{alberti2018lectures}.

In solving the inverse problem, the crucial intermediate object is the interior power density $H(x),x\in\Omega,$ for \eqref{eq:u} given by
\begin{equation} \label{eq:H}
    H = \sigma|\nabla u|^2.
\end{equation}
where $ u = u[\sigma] $ is the solution of \eqref{eq:u} as a function of the conductivity $ \sigma $. The power density shows up by considering the weak forms of \eqref{eq:u} and \eqref{eq:up}, each with the solution of the other taken as a test function therein.
\begin{align*}
    \int_{\partial\Omega} fg_\ast\d s &= \int_\Omega \sigma\nabla u\cdot\nabla u_\ast\d x, \\
    \int_{\partial\Omega} fg\d s &= \int_\Omega \sigma_\ast\nabla u_\ast\cdot\nabla u\d x
\end{align*}
Taking their difference and substituting in \eqref{eq:coupling}
\begin{align}\label{eq:I-exact}
    I = \int_{\partial\Omega} fg_\ast - fg\d s = -\eta\int_\Omega p\sigma\nabla u\cdot\nabla u_\ast\d x.
\end{align}
The approximation $u_\ast \approx u$ then yields
\begin{align}
    I &\approx  -\eta\int_\Omega p\overbrace{\sigma|\nabla u|^2}^{H}\d x. \label{eq:I-linearization}
\end{align}
We thus pose the linear inverse problem to find $H$ from the equation 
\begin{align} 
    I = KH \label{eq:I-linear}
\end{align}
with $K$ denoting the integral operator with kernel $ -\eta p $
\begin{align} 
    (KH)(t) = -\eta\int_\Omega p(x,t)H(x)\d x. \label{eq:K}
\end{align}
Such an integral operator $ K \colon L^2(\Omega) \to L^2(0,T) $ is compact when the integration kernel belong to $ L^2(\Omega \times (0,T)) $; as the wave $ p $ is smooth, this is indeed the case here.

The reconstruction is decomposed in two steps: First recover the power density $ H $ by solving $ \eqref{eq:I-linear} $ for each applied boundary current;  second recover $ \sigma $ from several $ H $ corresponding to different choices of $ f $.
\begin{remark} \label{rem:p-comp-set}
    Equations \eqref{eq:I-linear}--\eqref{eq:K} illustrate the importance of $ p(x,t) $ in the reconstruction. The generated waves must be a sufficiently expressive set that it captures all the facets of $ H $. In other words, the size of the $ \ker(K) $ is determined by the set of waves; of course we cannot hope to recover components of $ H $ in $ \ker(K) $.
\end{remark}
There are in general multiple possible sources of errors in \aet{}. The most obvious include measurement errors, linearization errors, and model errors. 
In this work we are interested in exploring the effect of uncertainty in the sound speed $c$.
We do so by considering the exact sound speed to be unknown to us, however, we assume prior knowledge of the mean value of the sound speed in the reconstruction approach. Using a wrong sound speed corresponds to having a model error. 

We consider sound speeds $ c $ and $ \widetilde{c} $, where the former represents the true sound speed and the latter an approximation available from prior knowledge. Numerically, we model $c$ 
by
\[ c(x) = c_{\rm{bg}}+ c_{\rm{var}}(x), \]
where $c_{\rm{var}},$ carries information about the uncertain variations in $c$.
We take $\tilde{c}= c_{\rm{bg}}$  constant in our studies.

\section{Stability of the forward operator} \label{sec:wave}
In this section, we demonstrate that, under certain assumptions, the wave $ p $ and forward operator $ K $ are continuously dependent on the sound speed $ c $. This guarantees that errors in the data are controlled by errors in the sound speed. The results are likely well-known, but with the lack of a proper reference, we indicate the overall ideas. 

Let $ m= \left\lceil\frac{d+1}{2}\right\rceil $. We define the admissible set of sound speeds as the set of smooth functions bounded from above and below by positive constant $ \lambda \in (0,1) $ 
\[ \mathcal A_\lambda := \left\{ c \in C^\infty(\R^d) : \lambda \leq c \leq \lambda^{-1};\; \|c\|_{C^k(\R^d)} <  \lambda^{-1}, k \leq m+1 \right\}. \]
A positive lower bound on $ c $ ensures that $ \partial_t^2 - c^2\Delta $ is a uniformly hyperbolic operator. The upper bound yields finite propagation speed for the wave $p$ \cite{evans2010partial}; the bounds on the $C^k$ norms allow uniform estimates inside $\mathcal A_\lambda .$ We assume in the following that $\lambda$ is fixed and that $c,\widetilde c \in \mathcal A_\lambda.$ 

Put $ h = \widetilde{c}-c $ and $ q = \widetilde{p} - p $. Clearly $ q $ solves the \pde{}
\begin{align} \label{eq:q}
    \left\{\begin{aligned}
        \left(\partial_t^2 -\widetilde c^2\Delta\right)q &= h(\widetilde{c} + c)\Delta p, && \text{in $ \R^d\times\R_+ $}, \\
        q\vert_{t=0} = \partial_tq\vert_{t=0} &= 0, && \text{on $ \R^d $}.
    \end{aligned}\right.
\end{align}
Note that, due to the finite speed of propagation, $ \supp\left\{h(\widetilde{c}+c)\Delta p(\cdot,t)\right\} $ is compact and so is the support of $q.$ We denote by $B\subset \R^d$ a large ball that contains the support of both $h(\widetilde{c}+c)\Delta p$ and $q$ for all $t \in (0,T).$ For all $ m $, the regularity of $q$ can now be estimated  \cite[p.415]{evans2010partial} by
\begin{equation} \label{eq:reg:evans}
    \operatornamewithlimits{ess\ sup}_{0\leq t\leq T}\sum_{j=0}^{m+1}\left\|\partial_t^jq(\cdot,t)\right\|_{H^{m+1-j}(B)} \leq C\sum_{j=0}^m\left\|h(\widetilde{c}+c)\Delta\partial_t^jp\right\|_{L^2(0,T; H^{m-j}(B))},
\end{equation}
where $C$ depends on the $C^{m+1}$-norm of $c$; thus by $ c \in \mathcal A_\lambda $ it depends only on $ \lambda $. This leads to:
\begin{proposition} \label{thm:reg:q}
    The wave difference is bounded by
    \begin{equation} \label{eq:reg:q}
        \|\widetilde{p}-p\|_{L^\infty(0,T; H^{m+1}(B))} \leq C\|\widetilde{c}-c\|_{H^m(B)},
    \end{equation}
    where $ C $ does not depend on $ \widetilde{c} $.
\end{proposition}
\begin{proof}
    Observe that, for any positive index $k,$ $ \|\cdot\|_{L^2(0,T;H^{m-k}(B))} \leq \|\cdot\|_{L^2(0,T;H^m(B))},$ and moreover that for $ m > d/2 $, $ H^m(\omega) $ is a Banach algebra\cite[4.39]{adams2003sobolev}.
    Applying these observations to \eqref{eq:reg:evans} yield
    \begin{equation*}
        \|\widetilde{p}-p\|_{L^\infty(0,T; H^{m+1}(B))} \leq C\|\widetilde{c}-c\|_{H^m(B)}\sum_{k=0}^m\|(\widetilde{c}+c)\Delta\partial_t^kp\|_{L^2(0,T;H^m(B))},
    \end{equation*}
    for some $ C > 0 $. Since $ \widetilde{c}+c \leq 2\lambda^{-1} $ and $ p \in C^\infty $ we get \eqref{eq:reg:q}.
\end{proof}
\begin{remark}
By boundedness of $ \Omega $ we can without loss of generality assume $ \Omega \subseteq B $.
\end{remark}
The established continuity for the wave upon the sound speed yields operator continuity for the forward operator in the following sense. Like the integration kernel of $ K $ derives from $ p $ and $ c $, we consider $ \widetilde K $ as the operator with integration kernel coming from $ \widetilde p $ and $ \widetilde c $.
\begin{proposition} \label{prop:op-cont}
The operator difference $ \widetilde{K} - K \colon L^2(\Omega) \rightarrow L^2(0,T)$ is bounded by
\[ \|\widetilde{K} - K\| \leq C\|\widetilde{c} - c\|_{H^m(B)}, \]
where $ C $ does not depend on $ \widetilde{c} $. 
\end{proposition}
\begin{proof}
    By the Cauchy-inequality
    \begin{gather*}
    \frac{\|(\widetilde{K}-K)H\|_{L^2(0,T)}^2}{\|H\|_{L^2(\Omega)}^2} = \frac{\eta^2\!\int_0^T\left[\int_\Omega (\widetilde{p}-p)H\d x\right]^2\!\!\!\d t}{\|H\|_{L^2(\Omega)}^2} \leq \eta^2T \|\widetilde{p}-p\|_{L^\infty(0,T;L^2(\Omega))}^2,
    \end{gather*}
    which shows $ \|\widetilde{K}-K\| \leq \eta\sqrt{T}\|\widetilde{p}-p\|_{L^\infty(0,T;L^2(\Omega))} $. Proposition \ref{thm:reg:q} now gives the estimate.
\end{proof}

\section{Inversion procedure}
In this section, we introduce the applied inversion procedure. The problem is dealt with in two parts, both handled as regularized minimization problems. First, the power densities are reconstructed from the measured differences in power $ I(t) $ by a standard least-squares approach, and second, the conductivity is reconstructed from the recovered power densities. We elaborate on the approach below.

\subsection{Reconstruction of the power density}

To reconstruct the power densities from \eqref{eq:I-linear} we use a regularized least squares approach \cite{jensen2019feasibility}. However, since the actual sound speed is known only approximately, we suggest to use the approximation $\widetilde c$ and the derived operator $ \widetilde K $ in place considering
\begin{equation} \label{eq:H-opt:cont}
    \argmin_{H\in L^2(\Omega)}\mathcal J_1(H), \quad \mathcal J_1(H) = \frac12\|\widetilde K H - I\|_{L^2(0,T)}^2 + \frac{\beta}{2}\|H\|_{L^2(\Omega)}^2.
\end{equation}

In contrast to \cite{jensen2019feasibility}, we discretize \eqref{eq:H-opt:cont} using a finite element basis $ \{\phi_j\}_{1\leq j\leq N} $. The regularization term is then discretized as
\begin{equation*}
    \|\mathbf L^T\mathbf v\|_2^2 = \|\mathbf v\|_\mathbf{M}^2 \approx \|v\|_{L^2(\Omega)}^2,
\end{equation*}
where $ \mathbf M = \mathbf L\mathbf L^T $ is the mass matrix for the finite element basis, and $ \mathbf L $ the Cholesky factor of $ \mathbf M. $ The vector $ \mathbf v $ is the coefficient vector for the finite element discretization of the continuous function $ v(x) $, and $\|\cdot\|_2$ is the usual Euclidean 2-norm. The finite dimensional regularized least squares problem therefore takes the form
\begin{equation} \label{eq:H-opt}
    \argmin_{\mathbf H \in \R^N}J_1(\mathbf H), \quad J_1(\mathbf H) = \frac{1}{2}\|\mathbf K\mathbf H - \mathbf I\|_2^2 + \frac\beta2\|\mathbf L^T\mathbf H\|_2^2,
\end{equation} 
where $ \mathbf K $ is the discretization of the operator $ K $ in \eqref{eq:I-linear}. $ \mathbf K $ has the form
\begin{equation*}
    (\mathbf K)_{ij} = -\eta\int_\Omega \widetilde p(\cdot,t_i)\phi_j\d x,
\end{equation*}
with $ \{t_i\}_{0\leq i \leq m} $ a uniform time-discretization; i.e. $ 0 = t_0 < t_1 < \dots < t_m = T $ and $ t_i - t_{i-1} = \Delta t $ for all $ 1\leq i\leq m $. Because the time-discretizing is uniform we may neglect the scaling coefficient, which would appear from discretizing the $ L^2(0,T) $-norm.
\begin{remark} \label{rem:H-opt-beta}
    To eliminate effects not relevant to our study, we solve our problem for various values of $ \beta $ comparing the reconstruction to the true power density in order to choose a close to optimal $ \beta $. 
\end{remark}

\subsection{Reconstruction of the conductivity}

To reconstruct the conductivity from the formerly reconstructed power densities we follow the general structure of the approach outlined in \cite{jensen2018acousto}; with some modifications. In particular we draw inspiration from \cite{harhanen2015edge} and improve on the approach by rewriting it into a preconditioned linear problem to which we apply a preconditioned conjugate gradient algorithm \cite[p.15]{barrett1994templates}.

We consider in the following only a single power density datum, denoted by $ z $, but note that the problem extends naturally for multiple by summing the separate data fidelity terms. The $ L^1 $--\textup{TV} optimization problem is
\begin{equation} \label{eq:func}
    \argmin_{\sigma}\mathcal J_2(\sigma), \quad \mathcal J_2(\sigma) = \|H[\sigma] - z\|_{L^1(\Omega)} + \gamma|\sigma|_\text{TV},
\end{equation}
with an $ L^1(\Omega) $ data fidelity term and total variation regularization. A minimizer for this problem is known to exist \cite{jensen2018acousto}.

Linearizing the power density with respect to the conductivity we write $ H[\sigma + \kappa] = H[\sigma] + H'[\sigma]\kappa $, where
\begin{align} \label{eq:H-linear}
    H'[\sigma]\kappa = \kappa|\nabla u[\sigma]|^2 + 2\sigma\nabla u[\sigma]\cdot\nabla u'[\sigma]\kappa,%
\end{align}
and $ u'[\sigma]\kappa $ is the Fréchet derivative of $ u[\sigma] $ in direction $ \kappa $. Substituting this linearization into \eqref{eq:func} one finds that the optimality condition becomes approximately that of the following weighted quadratic functional
\begin{equation} \label{eq:func-linear}
    \mathcal J_\sigma(\kappa) = \frac{1}{2}\int_\Omega w[\sigma]|H'[\sigma]\kappa - z_\sigma|^2\d x + \frac\gamma2\int_\Omega w_0[\sigma]|\nabla(\sigma+\kappa)|^2\d x,
\end{equation}
where $ w[\sigma] = |H[\sigma]-z|^{-1} $, $ z_\sigma = z - H[\sigma] $ and $ w_0[\sigma] = |\nabla \sigma|^{-1} $ and the absolute values $ |\cdot| $ are smoothened close to zero, i.e. $ |\cdot| \approx \sqrt{|\cdot|^2 + \tau^2} $ for a small $ \tau > 0 $; see \cite{jensen2018acousto}. Instead of tackling \eqref{eq:func}, we take steps by iteratively minimizing \eqref{eq:func-linear} and then updating $ \sigma $ and computing the new weights.

Discretizing \eqref{eq:func-linear}, we obtain the quadratic
\begin{equation} \label{eq:func-lin-disc}
    J_{\boldsymbol{\sigma}}(\boldsymbol{\kappa}) = \frac12\boldsymbol{\kappa}^T\left(\mathbf W^T\mathbf M_w\mathbf W + \gamma \mathbf K_{w_0}\right)\boldsymbol{\kappa} - \boldsymbol{\kappa}^T\left(\mathbf W^T\mathbf M_w\mathbf z_\sigma - \gamma \mathbf K_{w_0}\boldsymbol{\sigma}\right) + \textit{constant},
\end{equation}
where $ \{\phi_j\}_{1\leq j \leq n} $ is a finite element basis, and
\[
    (\mathbf M_w)_{ij} = \int_\Omega w[\sigma]\phi_i\phi_j\d x, \quad (\mathbf K_{w_0})_{ij} = \int_\Omega w_0[\sigma]\nabla\phi_i\cdot \nabla\phi_j\d{x}
\]
and $ \mathbf W $ is the discretization of the linear map $ \kappa \mapsto H'[\sigma]\kappa $. This discretization is arrives by considering the variational form of \eqref{eq:H-linear} and is given by $ \mathbf W = \mathbf M^{-1}\left(\mathbf M_u - 2\mathbf W_{\sigma,u}\mathbf K_\sigma^{-1}\mathbf L_u\right) $, where
\begin{gather*}
    (\mathbf M)_{ij} = \int_\Omega \phi_i\phi_j\d x, \quad (\mathbf M_u)_{ij} = \int_\Omega |\nabla u[\sigma]|^2\phi_i\phi_j\d x, \quad (\mathbf W_{\sigma,u})_{ij} = \int_\Omega \phi_i\sigma\nabla u[\sigma]\cdot\nabla\phi_j\d x, \\
    (\mathbf K_\sigma)_{ij} = \int_\Omega \sigma\nabla\phi_i\cdot\nabla\phi_j\d x, \quad \text{and} \quad (\mathbf L_u)_{ij} = \int_\Omega \phi_j\nabla u[\sigma]\cdot\nabla \phi_i\d x.
\end{gather*}
The $ \mathbf K_\sigma^{-1}\mathbf L_u $ factors enters from the discretization of $ \kappa \mapsto u'[\sigma]\kappa $.

The minimization of \eqref{eq:func-lin-disc} has the first order optimality condition
\begin{equation*}
    \left(\mathbf W^T\mathbf M_w\mathbf W + \gamma\mathbf K_{w_0}\right)\boldsymbol{\kappa} = \mathbf W^T\mathbf M_w\mathbf z_\sigma - \gamma\mathbf K_{w_0}\boldsymbol{\sigma}.
\end{equation*}
Following the idea in \cite{harhanen2015edge}, we consider a positive definite perturbation of $ \mathbf K_{w_0} $ and for a small $ \epsilon > 0 $ take the Cholesky factorization $ \mathbf L_0^{\phantom{,}}\mathbf L_0^T = \mathbf K_{w_0} + \epsilon \mathbf I $. Substituting this in the above we obtain
\begin{equation*}
    \left(\mathbf L_0^{-1}\mathbf W^T\mathbf M_w\mathbf W\mathbf L_0^{-T} + \gamma \mathbf I\right)\tilde{\boldsymbol{\kappa}} = \mathbf L_0^{-1}\mathbf W^T\mathbf M_w\mathbf z_\sigma - \gamma\mathbf L_0^T\boldsymbol{\sigma},
\end{equation*}
where $ \tilde{\boldsymbol{\kappa}} = \mathbf L_0^T\boldsymbol{\kappa} $. Setting $ \gamma = 0 $ we obtain the preconditioned linear problem
\begin{equation} \label{eq:opt-precond}
    \mathbf L_0^{-1}\mathbf W^T \mathbf M_w\mathbf W\mathbf L_0^{-T}\tilde{\boldsymbol{\kappa}} = \mathbf L_0^{-1}\mathbf W^T\mathbf M_w\mathbf z_\sigma.
\end{equation}
We thus minimize \eqref{eq:func-lin-disc} by solving \eqref{eq:opt-precond}, with the preconditioned conjugate gradient algorithm. This algorithm only needs the evaluation of $ \mathbf L_0^{\phantom{,}}\mathbf L_0^T $, so in practice we neither need to compute the Cholesky factor nor its inverse.

\subsection{Regularization strategy}
In \eqref{eq:H-opt:cont} the correct operator $K$ was replaced by the $\widetilde K$ due to the uncertainty in the sound speed. Furthermore, the data $I$ might be corrupted by noise. We are interested in the behaviour of the solution in the limit of both vanishing model uncertainty and vanishing noise. This is a general question for inverse problems with model errors and data noise. To emphasize the nature of the problem, we restate it in an abstract setting and approach the question in the spirit of classical regularization theory \cite{EnglHankeNeubauer1996,kirsch1996introduction}. 
The following convergence result has independent interest; the result might be available in the literature, but in the lack of a proper reference we give the details. 
\begin{theorem}  \label{thm:reg-strat}
    Let $\delta>0$ and $ A,A^\delta \colon X \rightarrow Y $ be linear and compact operators between Hilbert spaces $X,Y.$ Suppose further that $Ax=y$ for $x\in X$ and $y \in Y,$ and that  $ \|A - A^\delta\| < \delta,$  $ \|y-y^\varepsilon\|_Y < \varepsilon.$ Denote for $\beta >0$ by $ R_\beta^\delta = ((A^\delta)^\ast A^\delta + \beta\mathcal I_X)^{-1}(A^\delta)^\ast $ the Tikhonov regularized inverse of $A^\delta.$
    \begin{itemize}
        \item[i)] If $ x \in {\ran(A^\ast)}, $ i.e.\ $ A^\ast  {w} = x, $ for some $w\in Y,$ take $\beta = \beta(\delta,\varepsilon)$ such that 
        $ \delta/\sqrt{\beta(\delta,\varepsilon)} \to 0 $, $ \varepsilon/\sqrt{\beta(\delta,\varepsilon)} \to 0 $ and $ \beta(\delta,\varepsilon) \to 0 $ as $ \delta,\varepsilon \to 0.$ Then
        \[
            \|x - R_\beta^\delta y^\varepsilon\|_X \leq \frac{\delta}{\sqrt{\beta}}\|x\|_X + \sqrt{\beta}\|w\|_Y + \delta\|{w}\|_Y + \frac{\varepsilon}{\sqrt{\beta}} \to 0 \quad \text{as $ \delta,\varepsilon \to 0 $,} 
        \]
        \item[ii)] If $ x \in \overline{\ran(A^\ast)} $ there exist $ \beta(\delta,\varepsilon) \rightarrow 0  $ as $\delta,\varepsilon \rightarrow 0$ such that 
        \[
            \|x - R_{\beta}^{\delta} y^{\varepsilon}\|_X \to 0  \quad \text{as $\delta,\varepsilon \rightarrow 0$ .}
        \]
    \end{itemize}
\end{theorem}%

\begin{remark}
    The assumption $x \in \overline{\ran(A^\ast)}$ can often be interpreted as a smoothness assumption on the true solution $x$ \cite{kirsch1996introduction}. If $x\not \in \overline{\ran(A^\ast)},$ only the projection of $x$ onto $\overline{\ran(A^\ast)}$ is recovered in the above limits; that is, we can never recover components of $x \in \overline{\ran(A^\ast)}^\perp = \ker A$.
\end{remark}

\begin{remark}
    Assume for instance that $ \delta \propto \varepsilon^\gamma $, $ \gamma > 0 $, asymptotically as they vanish, then $ \beta(\delta,\varepsilon) = \delta^a\varepsilon^b $, where $ 0<a,b<2 $ solves $ \frac{2-a}b > \gamma $ and $ \gamma > \frac{a}{2-b} $, works. For simplicity, taking $ a = b = \frac1k $ all $ k > \frac{\gamma+1}{2} $ are solutions.
\end{remark}

To this end we first establish the following operator bound.
\begin{lemma} \label{lem:reg-strat}
    Let $ E = ((A^\delta)^\ast A^\delta + \beta\mathcal I_X)^{-1} $ and $ R_\beta^\delta $ be as in the theorem above, then $ \|E\| \leq \beta^{-1} $ and $ \|R_\beta^\delta\|\leq \beta^{-\frac{1}{2}} $; i.e. the operators are uniformly bounded independent of $ \delta $.
\end{lemma}
\begin{proof}
    For an arbitrary $ x \in X $ consider
    \begin{align*}
        \beta \|x\|_X^2 \leq \beta \|x\|_X^2 + \|A^\delta x\|_Y^2 = \dual{x,(\beta \mathcal I_X + (A^\delta)^\ast A^\delta)x}_X \leq \|x\|_X\|E^{-1}x\|_X,
    \end{align*}
    thus $ \|E\| \leq \beta^{-1} $.
    
    Note that $ E $ is self-adjoint and that we have $ R_\beta^\delta = E \circ (A^\delta)^\ast $. Let again $ x \in X $ and fix $ z = Ex $, then
    \begin{align*}
        \|(R_\beta^\delta)^\ast x\|_Y^2 &= \dual{A^\delta Ex,A^\delta Ex}_Y = \dual{z,(A^\delta)^\ast A^\delta z}_X \\
        &\leq \dual{z,(A^\delta)^\ast A^\delta z}_X + \beta\dual{z,z}_X = \dual{z,E^{-1}z}_X = \dual{Ex,x}_X \\
        &\leq \|Ex\|_X\|x\|_X \leq \beta^{-1}\|x\|_X^2,
    \end{align*}
    thus $ \|(R_\beta^\delta)^\ast\| \leq \beta^{-\frac12} $. As $ R_\beta^\delta $ is a bounded linear operator $ \|R_\beta^\delta\| = \|(R_\beta^\delta)^\ast\| $.
\end{proof}

With this we deal with the proof of the theorem.
\begin{proof}[Proof of Theorem \ref{thm:reg-strat}]
    We fix $ E = ((A^\delta)^\ast A^\delta + \beta\mathcal I_X)^{-1} $ as in Lemma \ref{lem:reg-strat}. 
    
    i) We note that
    \begin{align*}
        \|x - R_\beta^\delta y^\epsilon\|_X = \|x - R_\beta^\delta y + R_\beta^\delta(y - y^\epsilon)\|_X &\leq \|x- R_\beta^\delta y\|_X + \|R_\beta^\delta\|\|y-y^\epsilon\|_Y \\
        &\leq \|x- R_\beta^\delta y\|_X + \frac{\epsilon}{\sqrt{\beta}},
    \end{align*}
    the last inequality by applying Lemma \ref{lem:reg-strat}. We thus consider simply $ \|x-R_\beta^\delta y\|_X $ from here.
    \begin{align*}
        \|x - R_\beta^\delta y\|_X &= \|x - E\circ(A^\delta)^\ast Ax\|_X = \|E\circ\big[E^{-1} - (A^\delta)^\ast A\big]x\|_X \\
        &= \|E\circ\big[((A^\delta)^\ast A^\delta + \beta\mathcal I_X) - (A^\delta)^\ast A\big]x\|_X \\
        &= \|E\circ (A^\delta)^\ast [A^\delta-A]x + \beta E x\|_X  \addtocounter{equation}{1}\tag{\theequation} \label{eq:star} \\
        &= \|R_\beta^\delta[A^\delta-A]x + \beta E A^\ast w\|_X \\
        &\leq \|R_\beta^\delta\|\|A^\delta - A\|\|x\|_X + \beta\|E A^\ast{w}\|_X \leq \frac{\delta}{\sqrt{\beta}}\|x\|_X + \beta\|E A^\ast{w}\|_X 
    \end{align*}
    We consider now $ \beta\|EA^\ast{w}\|_X $,
    \begin{align*}
        \beta\|EA^\ast{w}\|_X &= \beta\|E((A^\delta)^\ast + A^\ast - (A^\delta)^\ast){w}\|_X\\
        &= \beta\|R_\beta^\delta{w} + E(A^\ast - (A^\delta)^\ast){w}\|_X \\
        &\leq \beta\|R_\beta^\delta\| \|{w}\|_X + \beta\|E(A^\ast-(A^\delta)^\ast){w}\|_X \\
        &\leq \sqrt{\beta}\|{w}\|_Y + \beta \|E\|\|(A-A^\delta)^\ast\|\|{w}\|_Y \\
        &\leq \sqrt{\beta}\|{w}\|_Y + \delta \|{w}\|_Y.
    \end{align*}
    Back-substituting yields i).
    
    ii) First, note that $ x \in \overline{\ran(A^\ast)} $ implies the existence of an $ x^\alpha \in \ran(A^\ast) $ satisfying $ \|x-x^\alpha\|_X < \alpha $ for any $ \alpha > 0 $. There is thus $ w^\alpha $ such that $ A^\ast w^\alpha = x^\alpha $.
    
    Consider then the term $ Ex $ in \eqref{eq:star} and expand
    \begin{align*}
        Ex = E(x^\alpha + x - x^\alpha) = Ex^\alpha + E(x-x^\alpha) = EA^\ast w^\alpha + E(x-x^\alpha).
    \end{align*}
    Substituting this into the above derivation we find
    \[
        \|x-R_\beta^\delta y^\varepsilon\|_X \leq  \frac{\delta}{\sqrt{\beta}}\|x\|_X + \sqrt{\beta}\|w^\alpha\|_Y + \delta\|w^\alpha\|_Y + \frac{\varepsilon}{\sqrt{\beta}} + \alpha
    \]
    We now first choose $ \beta = \beta(\delta,\varepsilon) $ as in i). Next, we choose $\alpha =  \alpha(\delta,\varepsilon,\beta) \to 0 $ such that the growth of $ \|w^\alpha\|_Y $ satisfies $ \max(\sqrt{\beta},\delta)\|w^{\alpha}\|_Y \to 0 $.
\end{proof}

We now adapt Theorem \ref{thm:reg-strat} to the particular problem:
\begin{corollary} \label{cor:reg-strat}
    Assume that $ m $, $ B\supseteq \Omega $ $ c $ , $ \widetilde{c} $, $K$ and $ \widetilde{K} $ are as in Section \ref{sec:wave}, that $ I = KH $ and that $ H \in \operatorname{Ran}(K^\ast) $. 
    If $ \beta \propto \|\widetilde{c}-c\|_{H^m(B)} $ there is a constant $ C > 0 $ independent of $ \widetilde{c} $ and $ \beta $ such that
    \[ \|H - \widetilde R_\beta I\|_{L^2(\Omega)} \leq C\|\widetilde{c}-c\|_{H^m(B)}^\frac{1}{2}, \]
    where $ \widetilde R_\beta = (\widetilde{K}^\ast \widetilde{K} + \beta\mathcal I)^{-1}\widetilde{K}^\ast $.
\end{corollary}
\begin{proof}
    The result follows from Theorem \ref{thm:reg-strat} part i), taking $ K $ and $ \widetilde K $ as our operators $A$ and $ A^\delta $, together with Propositions \ref{thm:reg:q} and \ref{prop:op-cont} from Section \ref{sec:wave}.
\end{proof}
\begin{remark}
    If $ H \in \overline{\ran(K^\ast)}\setminus \ran(K^\ast) $  convergence is still granted by Theorem \ref{thm:reg-strat} part ii), but the rate is no longer guaranteed.
\end{remark}
The above demonstrates that, as the sound speed uncertainty vanish (i.e.\ in the limit $ \widetilde{c} \rightarrow c $) and the regularization parameter is chosen appropriately, the solution of \eqref{eq:H-opt:cont}, with operator $ \widetilde{K} $, converges to the true solution $ H $. This result can be combined with known stability results for the mapping $H \mapsto  \sigma $ (see e.g.\ \cite{alberti2018lectures}) showing that also the correct $\sigma$ can be obtained in the limit of vanishing sound speed uncertainty. 

\section{Forward computations and uncertainty modeling}
Simulations are done for a two-dimensional problem. In general we fix parameter values and follow the approach in \cite[Sec. 3.3]{jensen2019feasibility} though with minor deviations. In the subsections we sketch our approach.

\subsection{Forward models} \label{sec:fwd-mod}
\begin{figure}[htb]
    \begin{minipage}[c]{0.40\textwidth}
        \centering
        \gfx{height=0.7\textwidth}{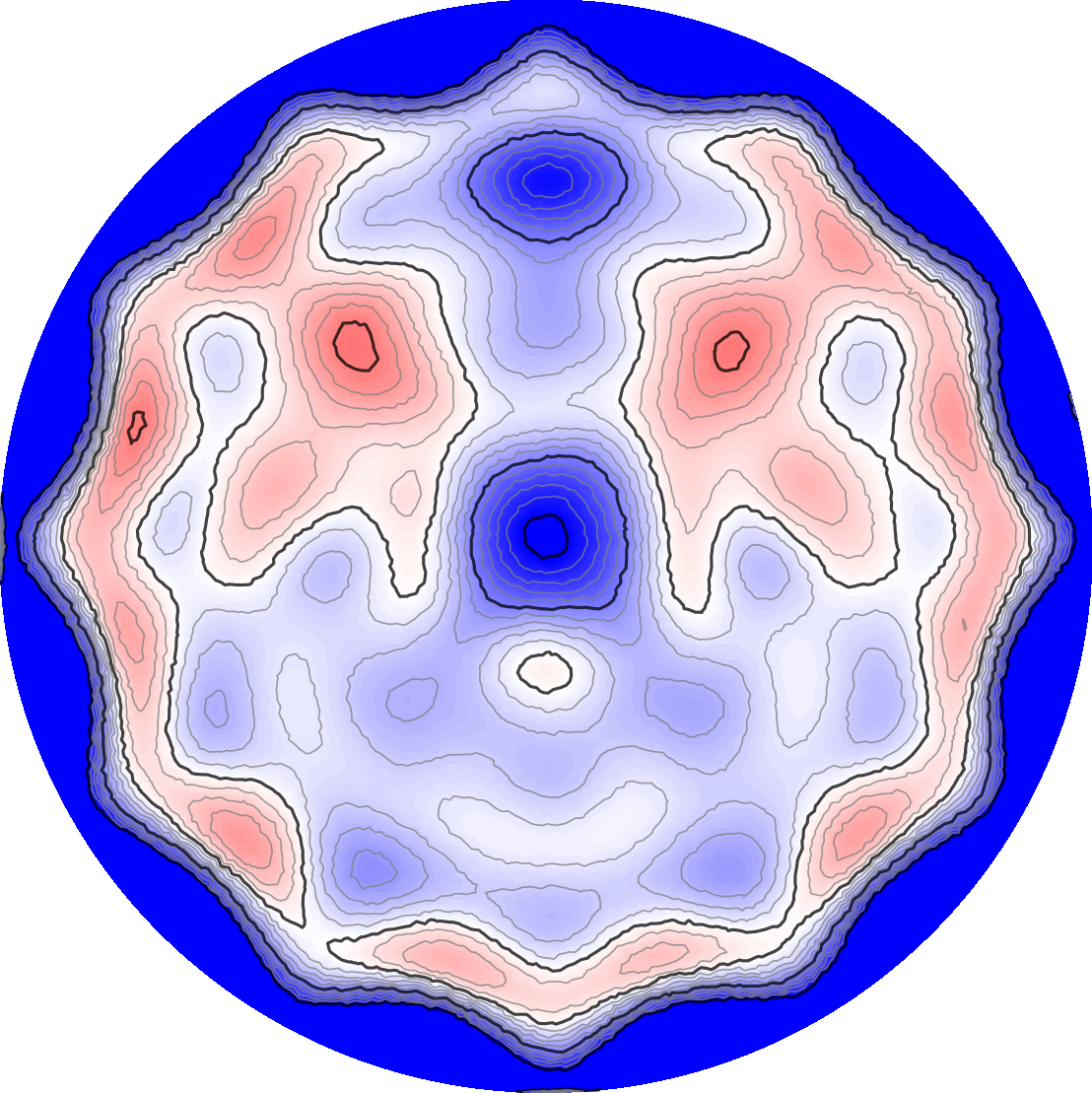}~%
        \gfx{height=0.7\textwidth}{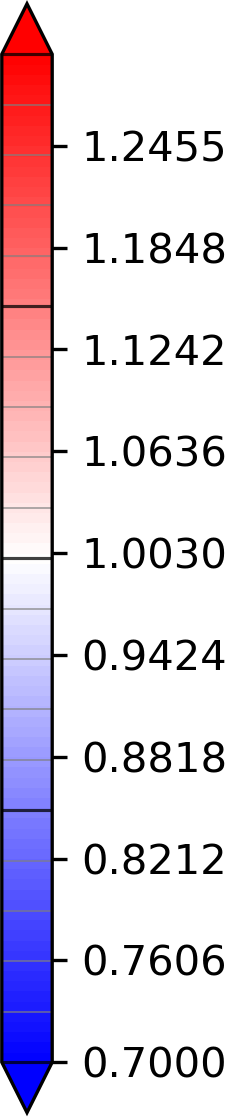}
    \end{minipage}\hfill
    \begin{minipage}[c]{0.60\textwidth}
        \caption{Reconstruction of $ H $ for boundary condition $ f(x,y) = x $ and 12 uniformly distributed wave sources. This reconstruction is done using a forward model with a homogeneous sound speed $ c $ from data generated with a homogeneous sound speed $ \widetilde c = 1.05c $. This illustrates star-like shape artifacts resulting from having only few waves, and the low-energy band close to the boundary resulting from the severely wrong mean sound speed assumption (here too low).
        } \label{fig:artifacts}
    \end{minipage}
\end{figure}

\begin{figure}[htb]
    \centering
    \gfx{height=0.28\linewidth}{rec1-true-s.png}~%
    \gfx{height=0.28\linewidth}{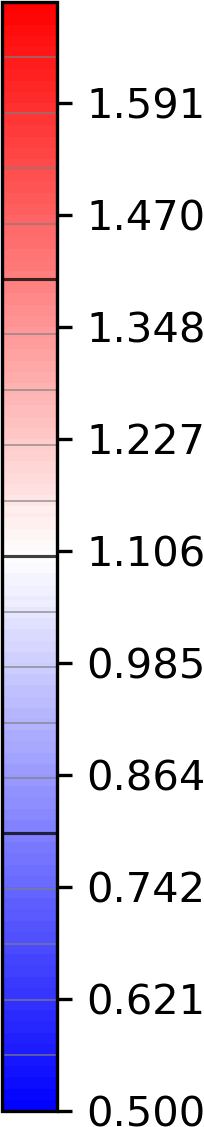}\hspace{1cm}%
    \gfx{height=0.28\linewidth}{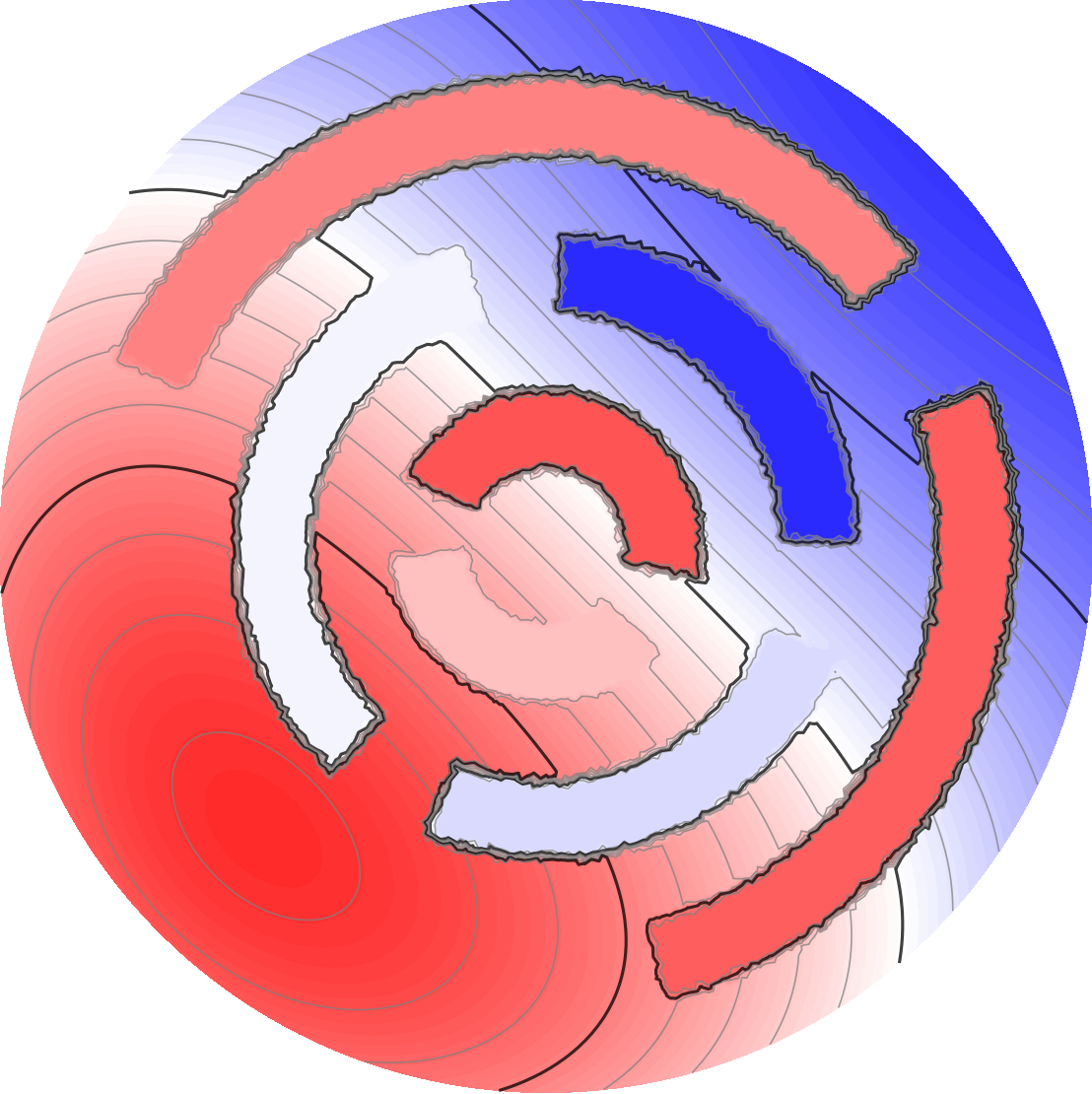}~%
    \gfx{height=0.28\linewidth}{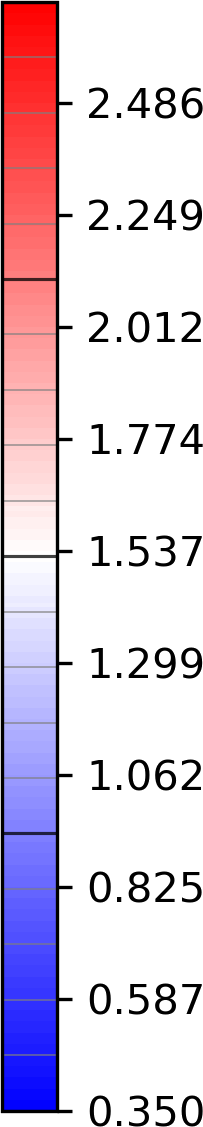}
    \caption{The two conductivity phantoms used for simulations; $ \sigma_1(x) $ on the left and $ \sigma_2(x) $ on the right.} \label{fig:t-sigma}
\end{figure}

We solve the wave equation \eqref{eq:p} using the \kwave{} package for \matlab{}. This is done on a regular square grid in a domain $ [-L,L]^2 \subset \R^2 $ and, to simulate the full domain $ \R^2 $, an absorbing boundary layer is added such that from within the subdomain $ [-L',L']^2 \subset [-L,L]^2 $ the solution is an approximation to the problem in the full space. Note that the domain of our electrical measurements $ \Omega $ is contained in $ [-L',L']^2 $. In \kwave{} we take $ L = 6\times 10^{-2} $ and $ L' = 4.5\times 10^{-2} $. The disc with radius $ 4\times 10^{-2} $ is then mapped to the electrical domain $ \Omega $, which is computationally taken as the unit disc.

The sound speed $ c $ is generated based on a breast tissue model \cite{reiser2013validation}. The model along with the changes introduced to accommodate our problem are outlined in Section \ref{sec:samp-schm}. The sound speed for generating the data has the form
\begin{align} \label{eq:c-formula}
    c(x;\theta) = c_{\rm{bg}} + \mu c_{\rm{bg}}s(x;\theta)
\end{align}
where $ s(x;\theta) $ is the sampled structured perturbation with $\theta$ the random variable. $ c_{\rm{bg}} = 1500$m/s is the background sound speed. $ \mu $ is a small scaling factor fixed at $ \mu = 0.05 $ in our simulations unless otherwise specified. We note that the density of the medium in \kwave{} is taken to be 1000.

We use 36 uniformly distributed wave-source positions in order to eliminate artifacts that would appear when too few waves are used; see Figure \ref{fig:artifacts}. Each transducer source consists of a sequence of point sources on the arc of $ \partial\Omega $. The transducer produces half a cycle of a 50 kHz tone burst.

The generated acoustics fields, all expressed in $ p(x,t) $, are interpolated on an unstructured triangulated mesh $ \mathcal M $ with 20100 nodes; $ \mathcal M $ representing $ \Omega $.

The electrical potentials are computed as in \cite{jensen2019feasibility} using \fenics{} \cite{fenics} with a $ \mathcal P_1 $ finite element basis on the mesh $ \mathcal M $. 
To ensure no other sources of uncertainty than the unknown sound speed, no noise components are added at this step. Moreover, as our primary interest is in the errors induced by the waves, we avoid any linearization error, from the approximation in \eqref{eq:I-exact}--\eqref{eq:I-linearization}, by generating our time series data directly by first forming the power density $H$, and then by integrating it against the waves on $ \Omega $ numerically in the linear expression \eqref{eq:I-linearization}.

The used conductivity phantoms are shown in Figure \ref{fig:t-sigma}. The simple phantom on the left is given by  $ \sigma_1(x) = 1 + \frac{1}{2}\chi_D(x) $, where $ D = B_\frac{1}{4}\left(0,\frac{3}{8}\right) $. The more challenging phantom on the right $ \sigma_2(x) $ is defined by inclusions made up from 7 polar rectangles on a smooth background. 
The  background is a polynomial given by
\begin{align*}
    \frac{2}{k}\operatorname{poly}(x_1,x_2) + \frac12, \quad \operatorname{poly}(x_1,x_2) = \left(1-x_1^2 - x_2^2\right)\left(\left(x_1-\frac{3}{4}\right)^2 + \left(x_2-\frac{3}{4}\right)^2\right)
\end{align*}
where $ k \equiv \max\{\operatorname{poly}(x_1,x_2) : \|(x_1,x_2)\|_2 \leq 1\} \approx 1.8272 $; the polar rectangles are defined in polar coordinates $(r,v),\; r_1 \leq r < r_2,\; v_1 \leq v < v_2$ with the values in Table \ref{tab:arc-regions}.
\begin{table}[ht]
    \centering
    \begin{tabular}{|c|c c|c c|c|}\hline
         layer & $ r_1 $ & $ r_2 $ & $ v_1 $ & $ v_2 $ & value \\\hline
         outer & $\frac57$ & $ \frac67 $ & $ -0.4167\pi$ & $ 0.1099\pi $ & 2.2646 \\[2pt]
         outer & $\frac57 $ & $ \frac67$ & $ 0.2106\pi$ & $ 0.8697\pi $ & 2.0845 \\[2pt]
         middle & $\frac37 $ & $ \frac47$ & $ 0.0057\pi$ & $ 0.4847\pi $ & 0.5461 \\[2pt]
         middle & $\frac37 $ & $ \frac47$ & $ 0.5840\pi$ & $ 1.2579\pi $ & 1.4765 \\[2pt]
         middle & $\frac37 $ & $ \frac47$ & $ 1.3731\pi$ & $ 1.8870\pi $ & 1.3523\\[2pt]
         inner & $\frac17 $ & $ \frac27$ & $ -0.0650\pi$ & $ 0.8526\pi $ & 2.3122 \\[2pt]
         inner & $\frac17 $ & $ \frac27$ & $ 1.0193\pi$ & $ 1.9350\pi $ & 1.8090 \\[2pt]\hline
    \end{tabular}
    \caption{The parameters for the different polar rectangles in the more complicated conductivity phantom seen on the right in Figure \ref{fig:t-sigma}.}
    \label{tab:arc-regions}
\end{table}

\subsection{Generation of random structures} \label{sec:samp-schm}
\begin{figure}
    \centering
    \gfx{width=0.30\linewidth}{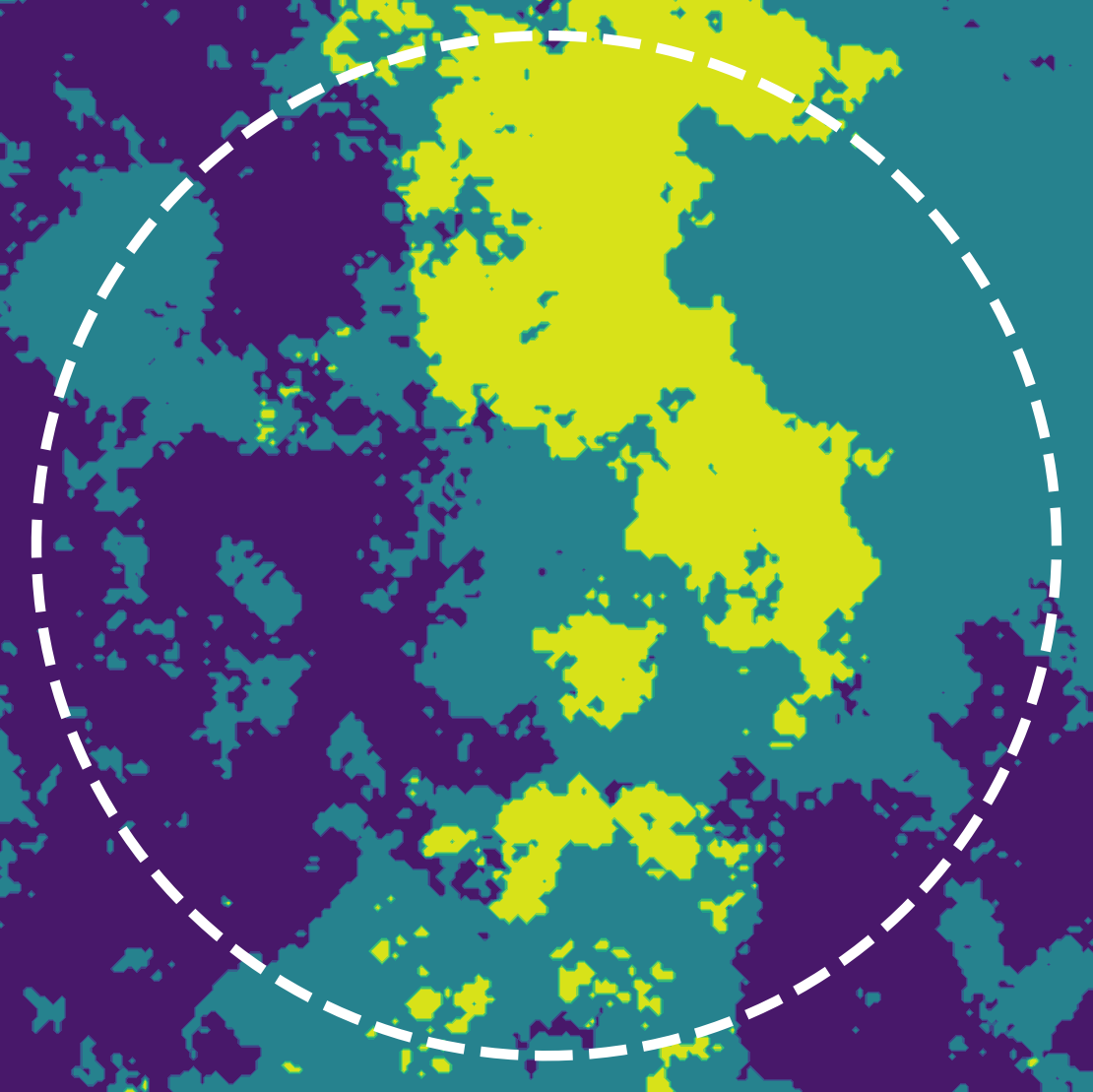}~%
    \gfx{width=0.30\linewidth}{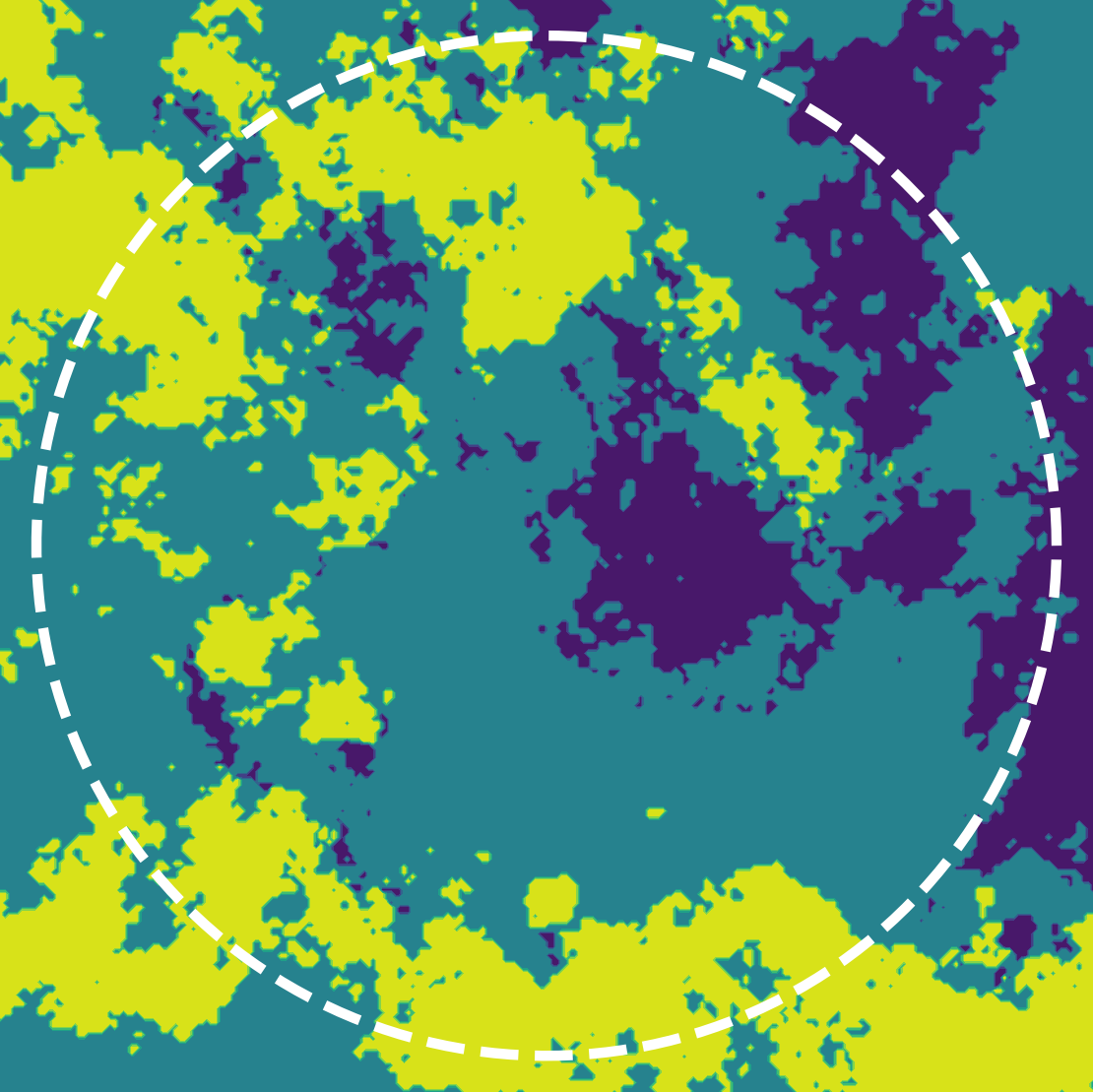}~%
    \gfx{width=0.30\linewidth}{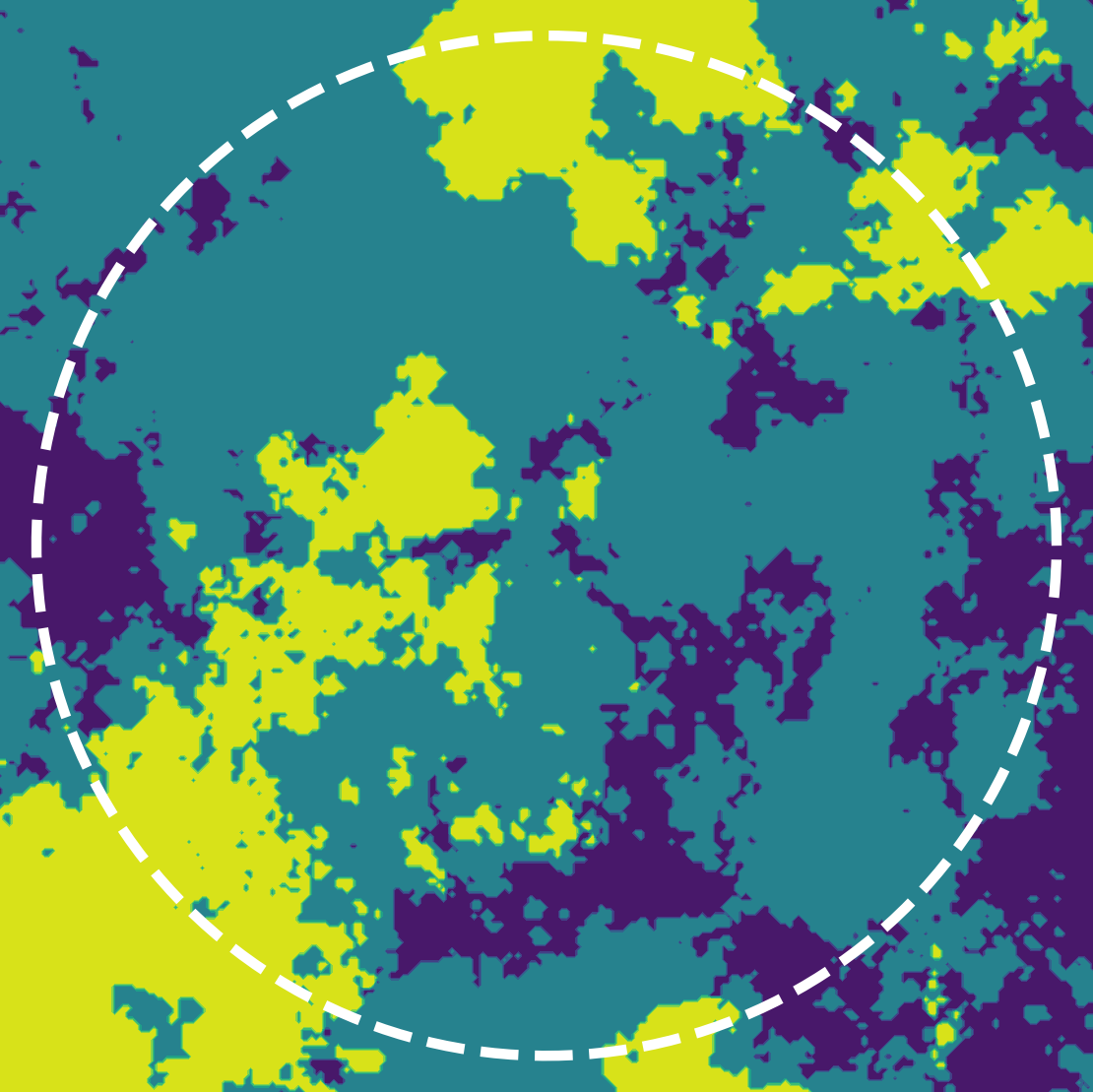}
    \caption{Examples of realizations of $ s(x;\theta) $; $ (\text{\ccrule[246,230,32]{.9em}{.9em}}, \text{\ccrule[32,147,140]{.9em}{.9em}}, \text{\ccrule[69,5,89]{.9em}{.9em}}) = (1,0,-1) $. The white dashed line marks the boundary of the domain $ \Omega $ on which the AET-problem is considered.} \label{fig:real}
\end{figure}
We sample different structures $ s $ for the sound speed $ c $, as used in \eqref{eq:c-formula}, and generate data to observe the influence on the reconstructions. The sound speed is sampled based on the breast tissue model proposed on \cite{reiser2013validation}, though we make a couple of modifications to their proposed model. This model was used as a base since it provides interesting non-trivial variations with some actual structure to them, compared to for instance adding Gaussian noise to the sound speed, which has no structure. Also, a proposed application area for AET is mammography, which makes the model topical. The modifications to the model are made for two primary reasons. First, our problem is in 2D where the original model is created for 3D, hence we move it to 2D. Second, the original model creates smooth structures; to better control the amount of disruptive structure we add to our acoustic medium we here create piecewise constant structures instead.

The remainder of this section details the model used to generate the random structures. The end product is a structure function $ s(x;\theta) $, where $ \theta $ is our random variable. Hence to create different structures we sample different $ \theta $.

We define constants $ f_0 = 20 $, $ \ell = 25 $, $ c_0 = 0.5 $, $ c_1 = 1 $ and an $ N\times N $-grid with points $ \xi_{jk} := (\xi_j,\xi_k) \in [-\ell,\ell]^2, 1\leq j,k \leq N $, $ \xi_j = \ell\left(2\frac{j-1}{N-1}-1\right) $. Then we draw uniformly distributed phase samples for each node in the grid, $ \theta_{jk} \sim \mathcal U(-\pi,\pi) $. We define the function $ V_\beta(\xi; \theta) $, as
\begin{equation}
    V_\beta(\xi; \theta) = \begin{cases}
        c_0, & |\xi| = 0,\\
        c_1|\xi|^{-\frac\beta2}e^{-i\theta}, & 0 < |\xi| < f_0, \\
        0, & \text{otherwise},
    \end{cases}    
\end{equation}
and evaluate at each grid point, $ v_{jk} = V_\beta(\xi_{jk};\theta_{jk}) $. 
We take the discrete inverse 2D Fourier transform of $ v_{jk} $, thus defining $ q_{jk} = |\mathcal F^{-1}_\textup{discrete}v_{jk}| $. Define a region $ U \subseteq [-\ell,\ell]^2 $, where we want to control the structures, and let $ \mathcal J_U = \{(j,k) : (\xi_j,\xi_k) \in U\} $ and
\[ 
    r(\gamma) := \argmin_{r\in\R} \left|\gamma - \frac1{|\mathcal J_U|}\sum_{(j,k)\in \mathcal J_U}\max(\operatorname{sign}(r-q_{jk}),0) \right|. 
\]
That is, $ r(\gamma) $ is the height at which to make a cut such that the ratio of grid points in $ U $ with values less than $ r(\gamma) $ compared to the total amount of grid points in $ U $ is as close to $ \gamma $ as possible. We then put
\begin{equation*}
    \widehat{q}_{jk} = \begin{cases}
        1 & \text{if $ q_{jk} < r(\gamma) $}, \\
        0 & \text{otherwise}.
    \end{cases}
\end{equation*}
We thus have a random (due to the distribution on $ \theta_{ij} $) map $ (\beta,\gamma) \mapsto \widehat{Q}_{jk}(\beta,\gamma) := \widehat{q}_{jk} $. From here we take some liberties in constructing our random structured sound speed. We proceed as follows: Define $ \beta_0 = 3.3 $, $ \beta_1 = 2.8 $ and $ \gamma = 0.35 $. We choose $ U \subset \Omega $ slightly away from the boundary of $ \Omega $, $ U $ is here a disc with radius $ \frac45 $ the radius of $ \Omega $, to ensure that the primary amount of variations will be exhibited in the central part of $ \Omega $. Then we put 
\begin{equation*}
    s_{jk} := Q_{jk}(\beta_0,\gamma) - Q_{jk}(\beta_1,\gamma).
\end{equation*}
Note that $ s_{jk} $ takes only the discrete values -1, 0 and 1. We define $ s(x;\theta) $ as the linear grid interpolation of $ s_{jk} $. Examples of draws from $ s(x;\theta) $ for some different realizations of $ \theta $ are illustrated in Figure \ref{fig:real}.

\begin{remark}
Note that $ c(x;\theta) $, computed as in \eqref{eq:c-formula}, may have an average slightly different from $ c_\textup{bg} $. Since the range of $ s_{jk} $ is discrete, $ c(x;\theta) $ may have very steep slopes depending on the discretization. A smoothed version may be obtained by convolution with a mollifier function, though we do not actually do that in our test cases. 
\end{remark}

\begin{remark}
We scale the spatial domain of $ s(x;\theta) $ to coincide with the square domain we have for wave equation. This is not an issue as the choice of value for $ \ell $ in the sampling scheme is unitless and we could move everything relative to a different scale and obtain the same $ s(x;\theta) $; e.g. let $ \alpha > 0 $ and consider a new $ \ell \to \alpha \ell $, then we should use $ f_0 \to \alpha f_0 $ and scale $ |\xi| \to |\xi/\alpha| $ in $ V_\beta(\xi;\theta) $ to compensate.
\end{remark}

\section{Numerical results}
 
Figures \ref{fig:reconstructions} and \ref{fig:reconstructions2} illustrate reconstructions of power densities for different boundary conditions (rows 1-3) and the conductivity (row 4) for the phantoms $ \sigma_1(x) $ and $ \sigma_2(x) $ respectively. The power densities are reconstructed by solving \eqref{eq:H-opt} and then the conductivity is found by solving \eqref{eq:opt-precond}. Rows 1--3 are power densities corresponding to boundary conditions $ f(x_1,x_2) = x_1$, $ x_2 $, $ (x_1+x_2)/\sqrt{2} $ respectively. The first column contains (for comparison) reconstructions of the ``best possible case'' where $ s(x) = 0 $, i.e. the correct sound speed is used for solving the inverse problem.

The test case in Figure \ref{fig:reconstructions} purposely uses a simple conductivity to clearly illustrate the effect of the added structure from the $ s(x) $ term. Comparatively the test case in Figure \ref{fig:reconstructions2} is a more complicated phantom with various regions and smooth areas serves to demonstrate that the discernible reconstructions obtained for the simpler phantom are replicatable even when the target is more complicated; more separate regions, smoothness and higher contrasts.

\begin{figure}
    \centering
    \gfx{height=0.23\linewidth}{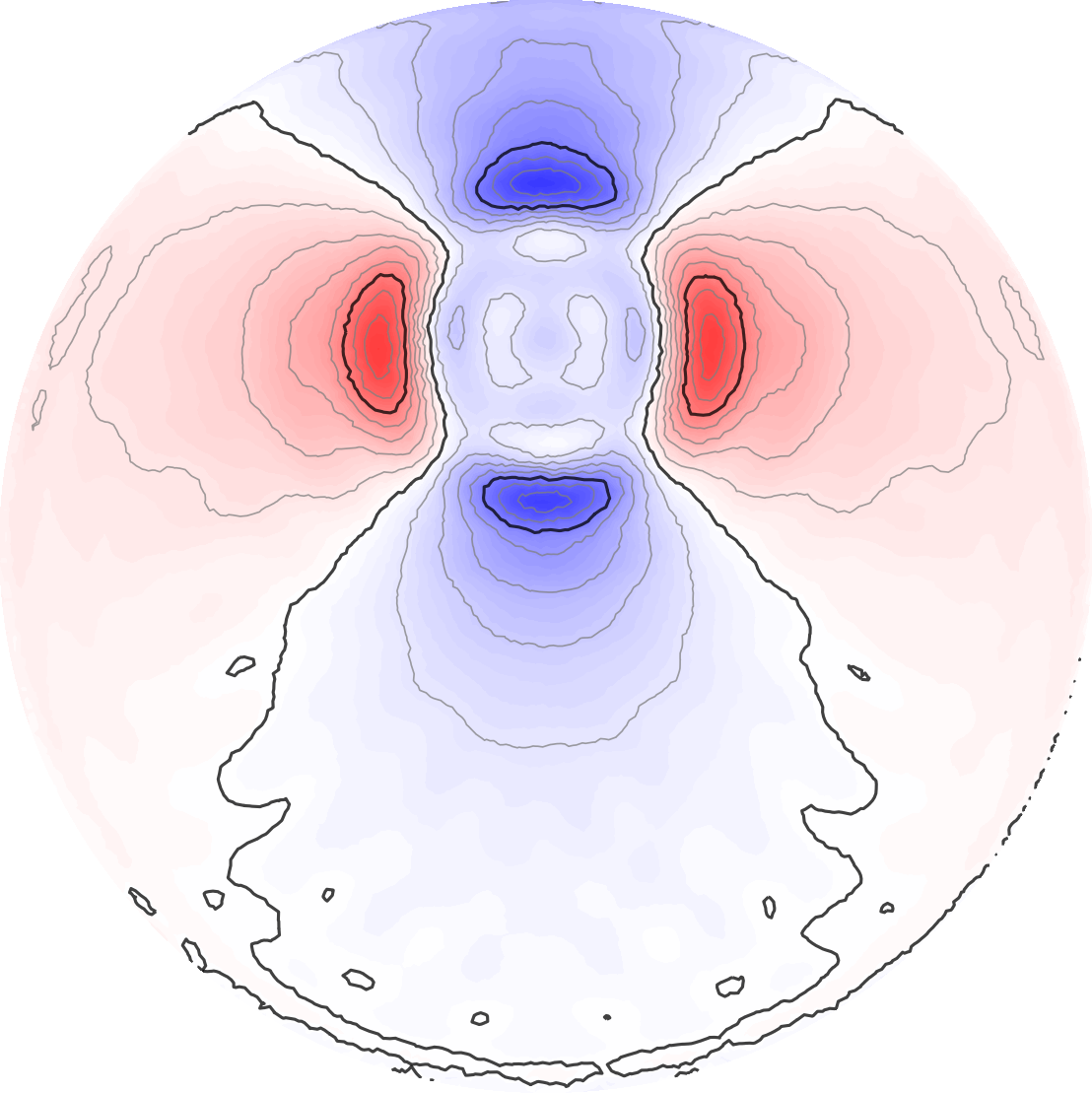}~%
    \gfx{height=0.23\linewidth}{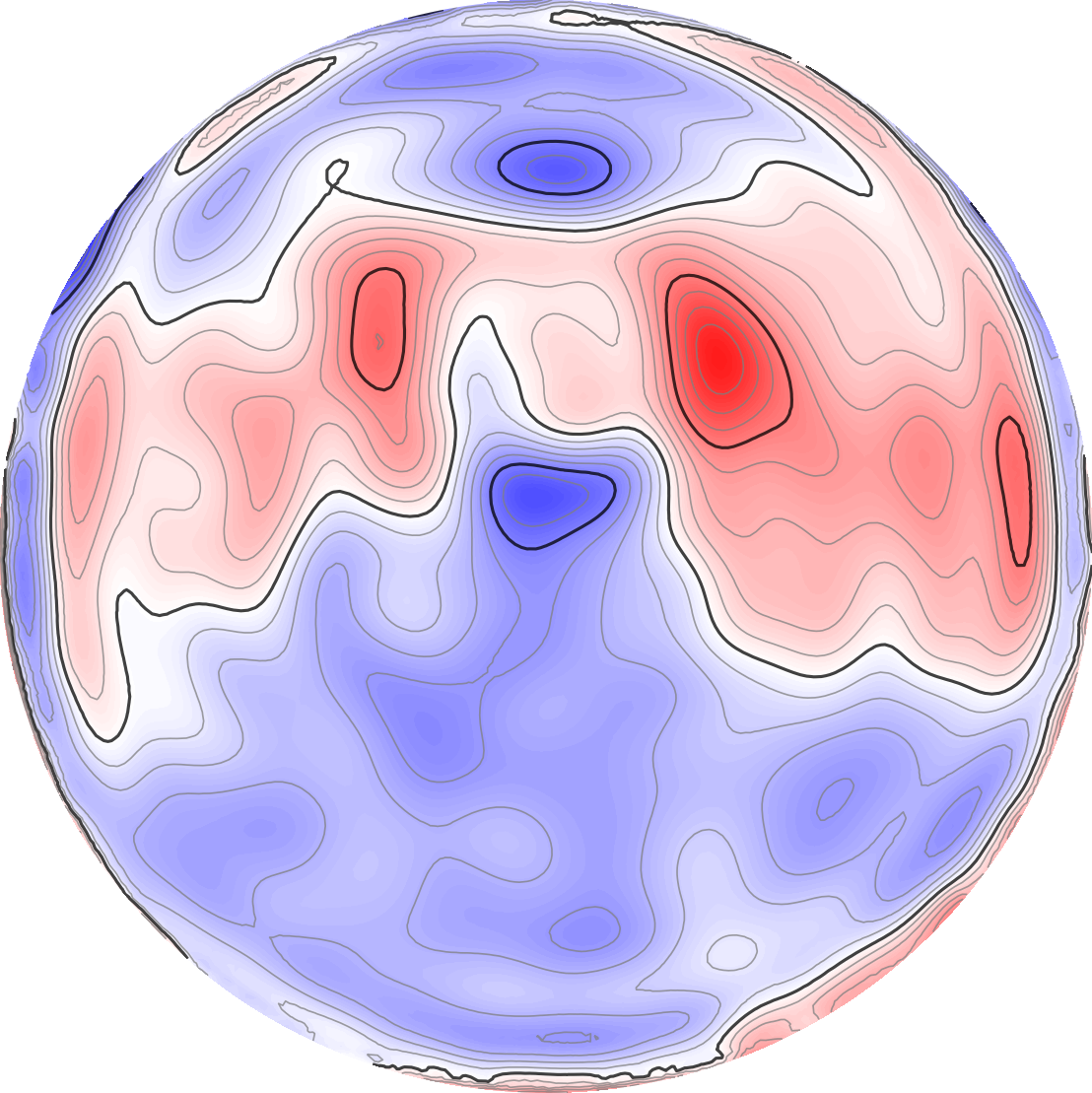}~%
    \gfx{height=0.23\linewidth}{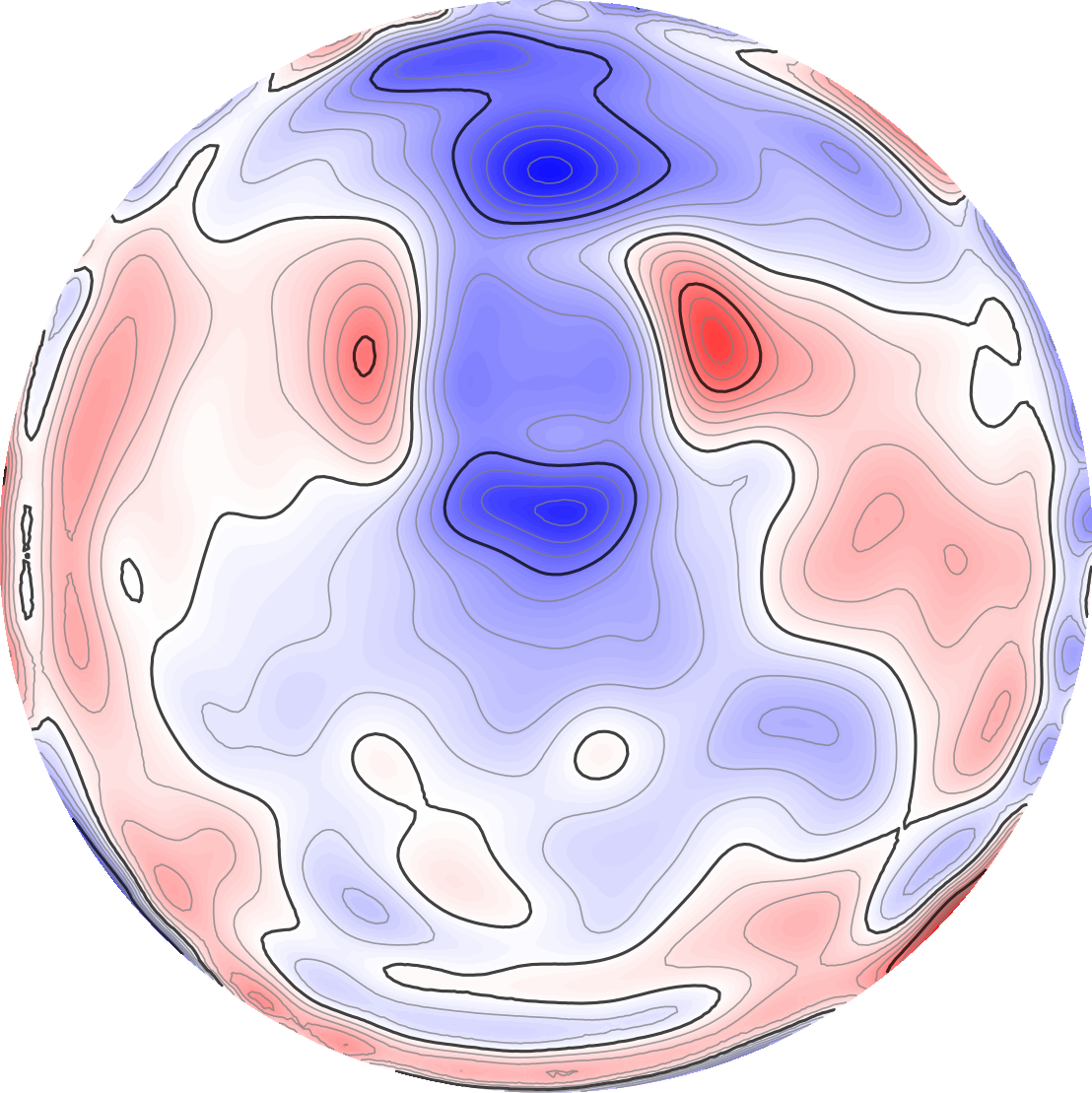}~%
    \gfx{height=0.23\linewidth}{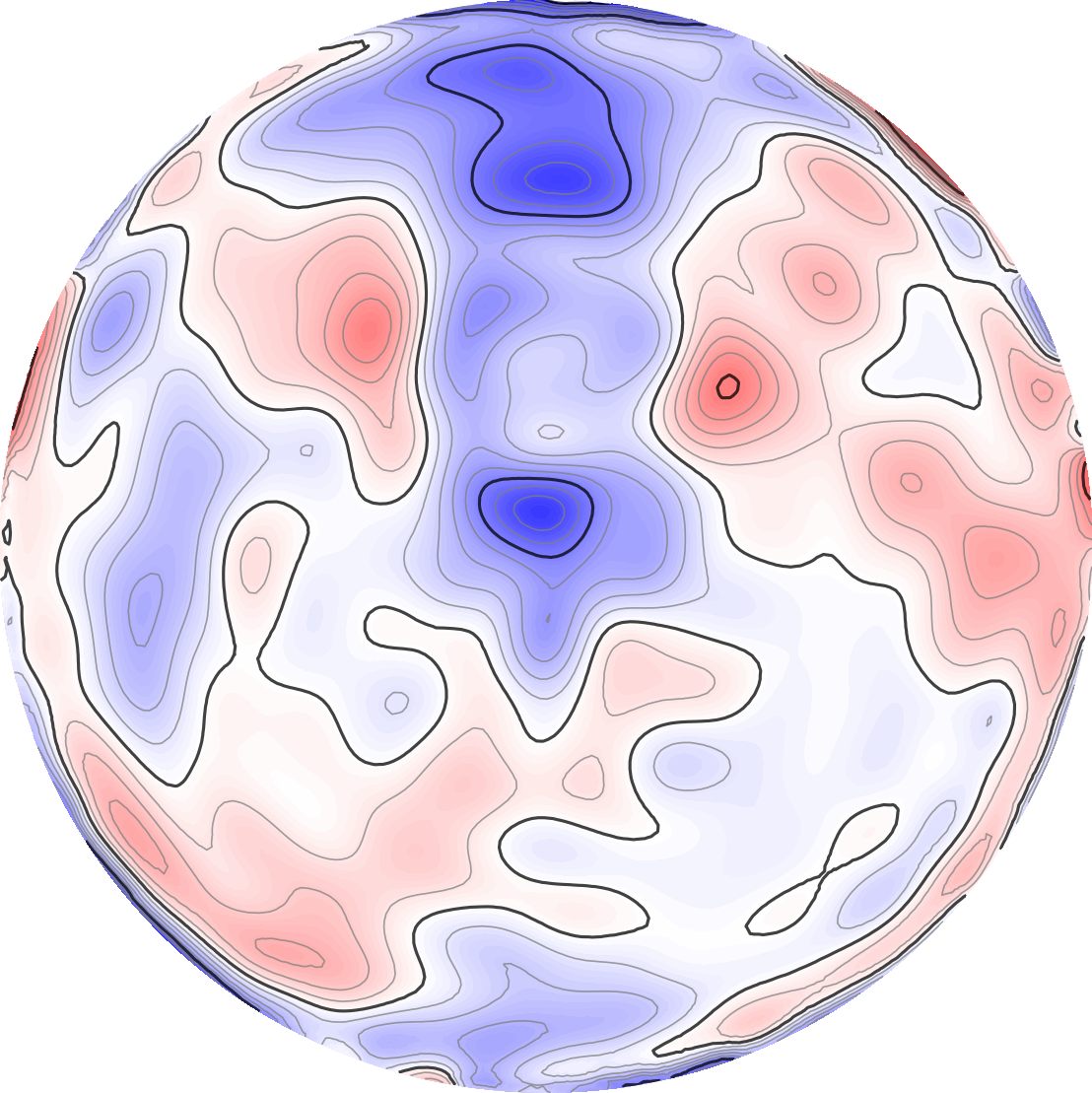}~~%
    \gfx{height=0.23\linewidth}{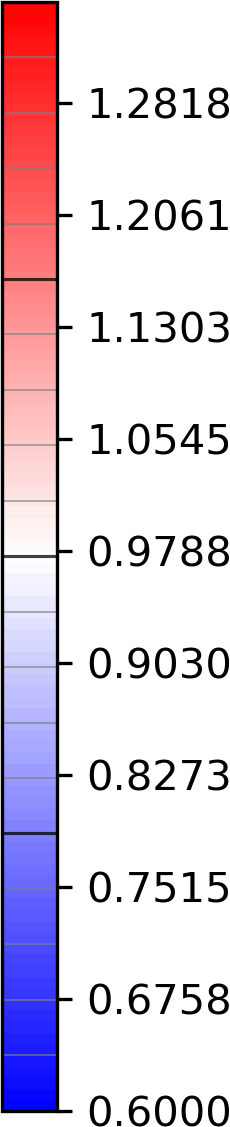}\\
    \gfx{height=0.23\linewidth}{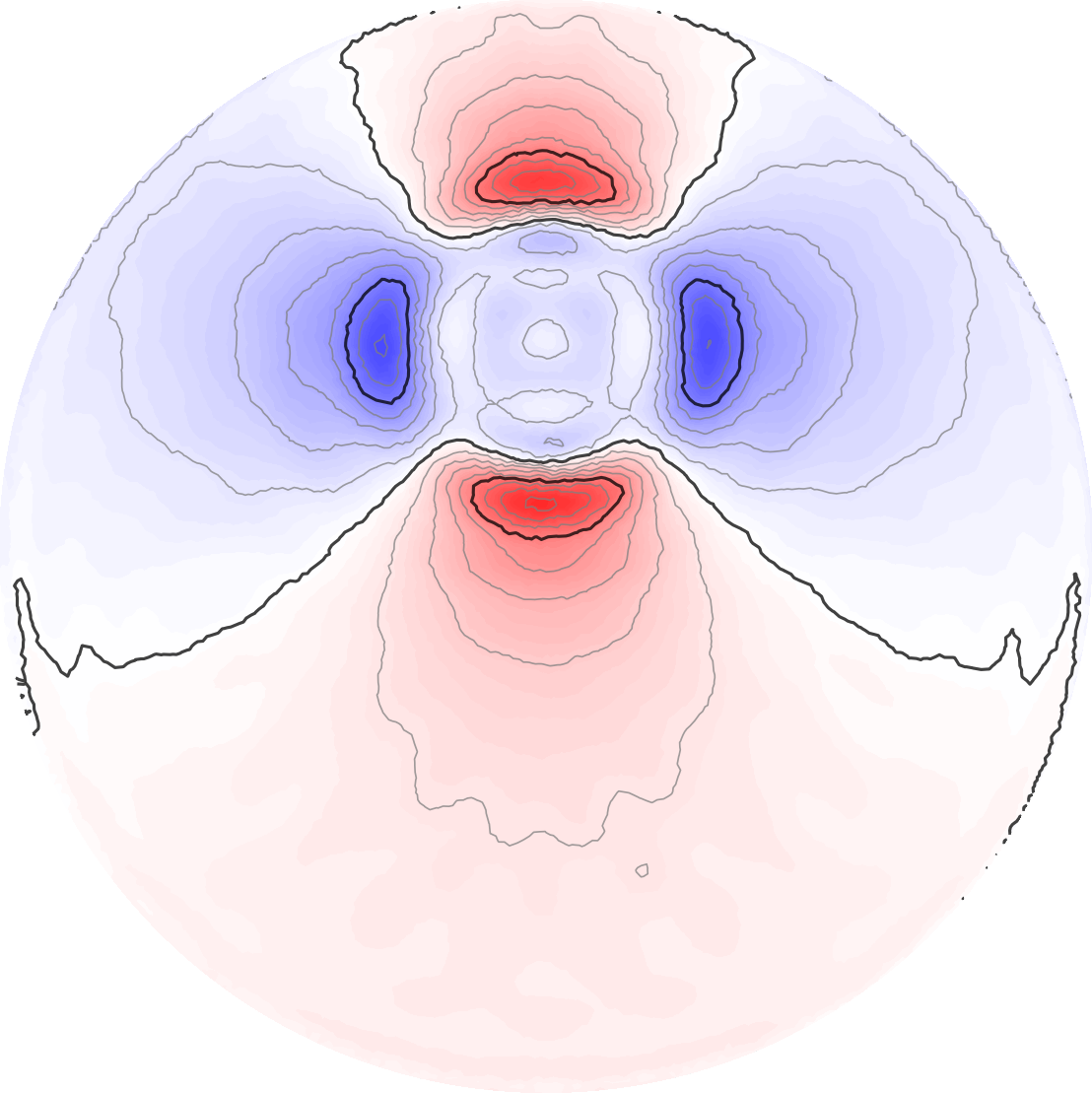}~%
    \gfx{height=0.23\linewidth}{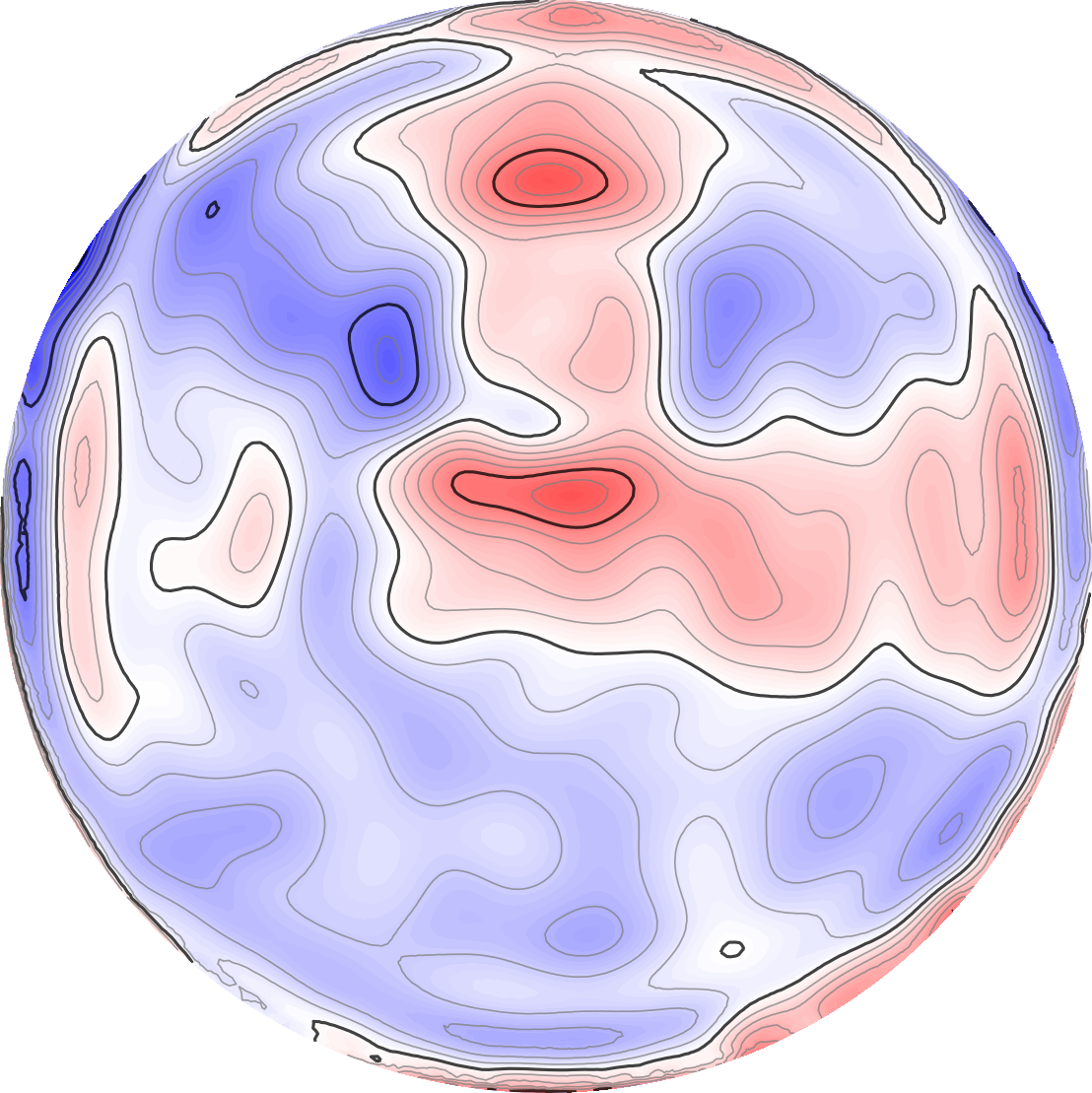}~%
    \gfx{height=0.23\linewidth}{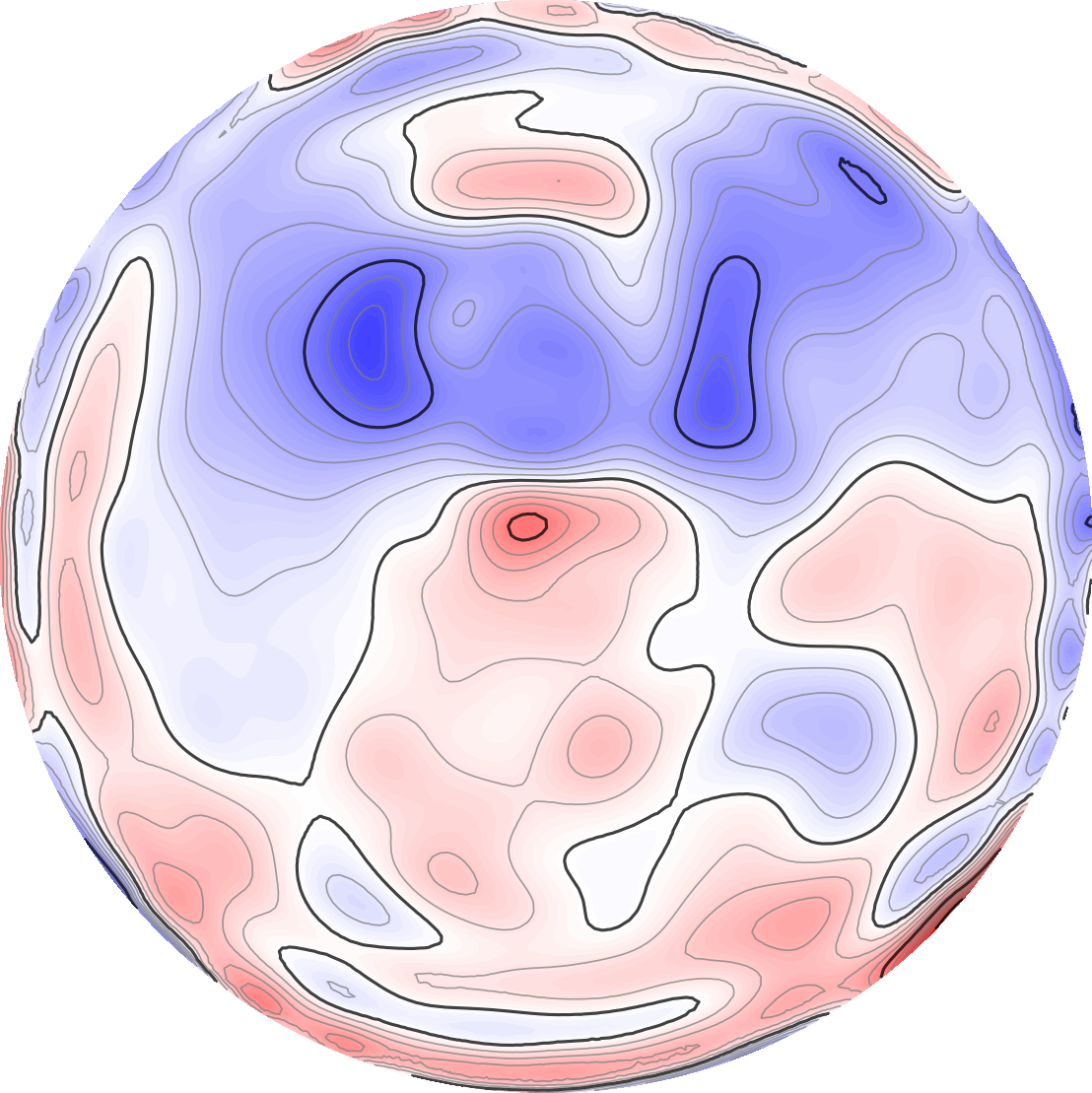}~%
    \gfx{height=0.23\linewidth}{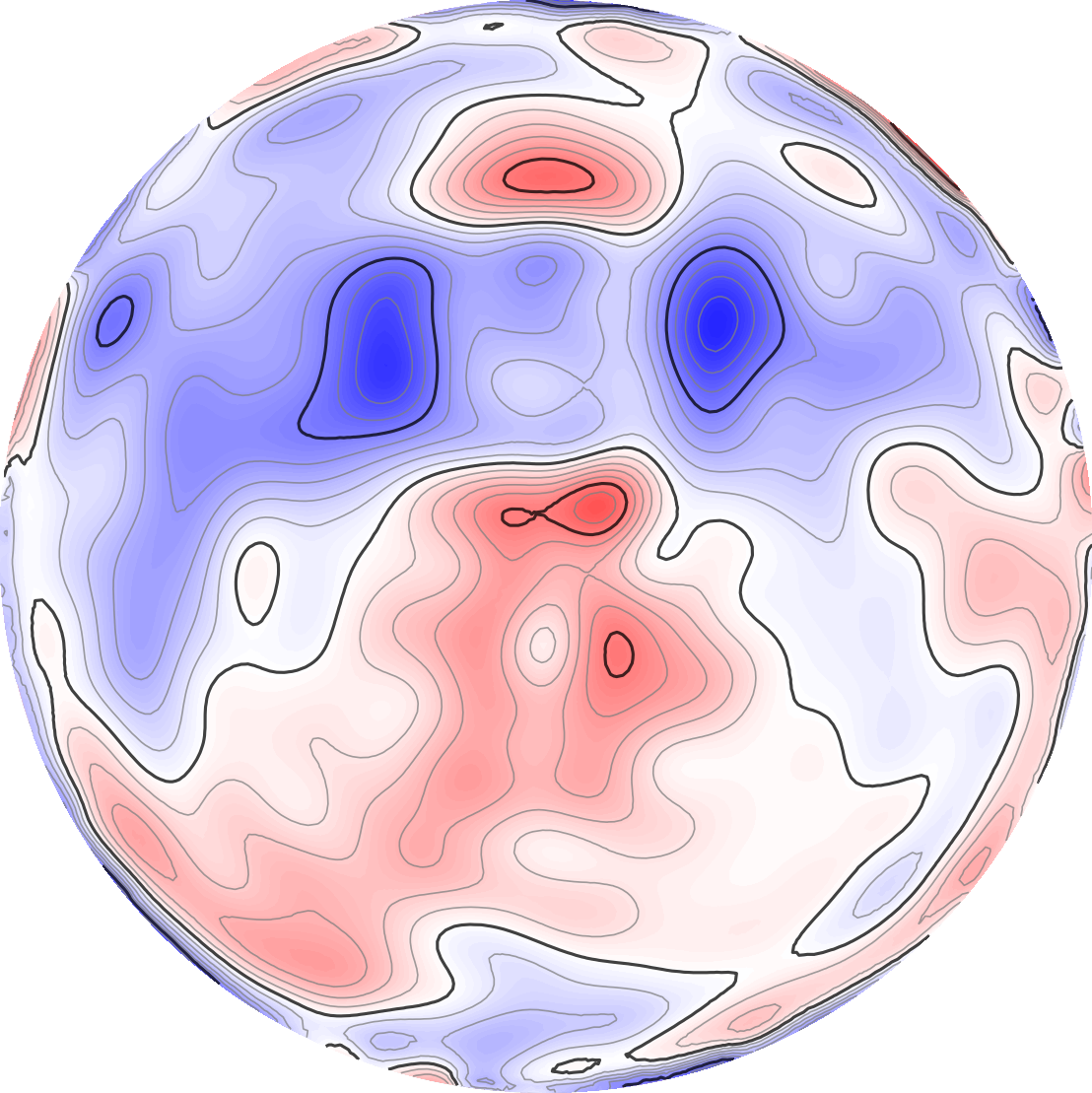}~~%
    \gfx{height=0.23\linewidth}{rec1-cb-H.png}\\
    \gfx{height=0.23\linewidth}{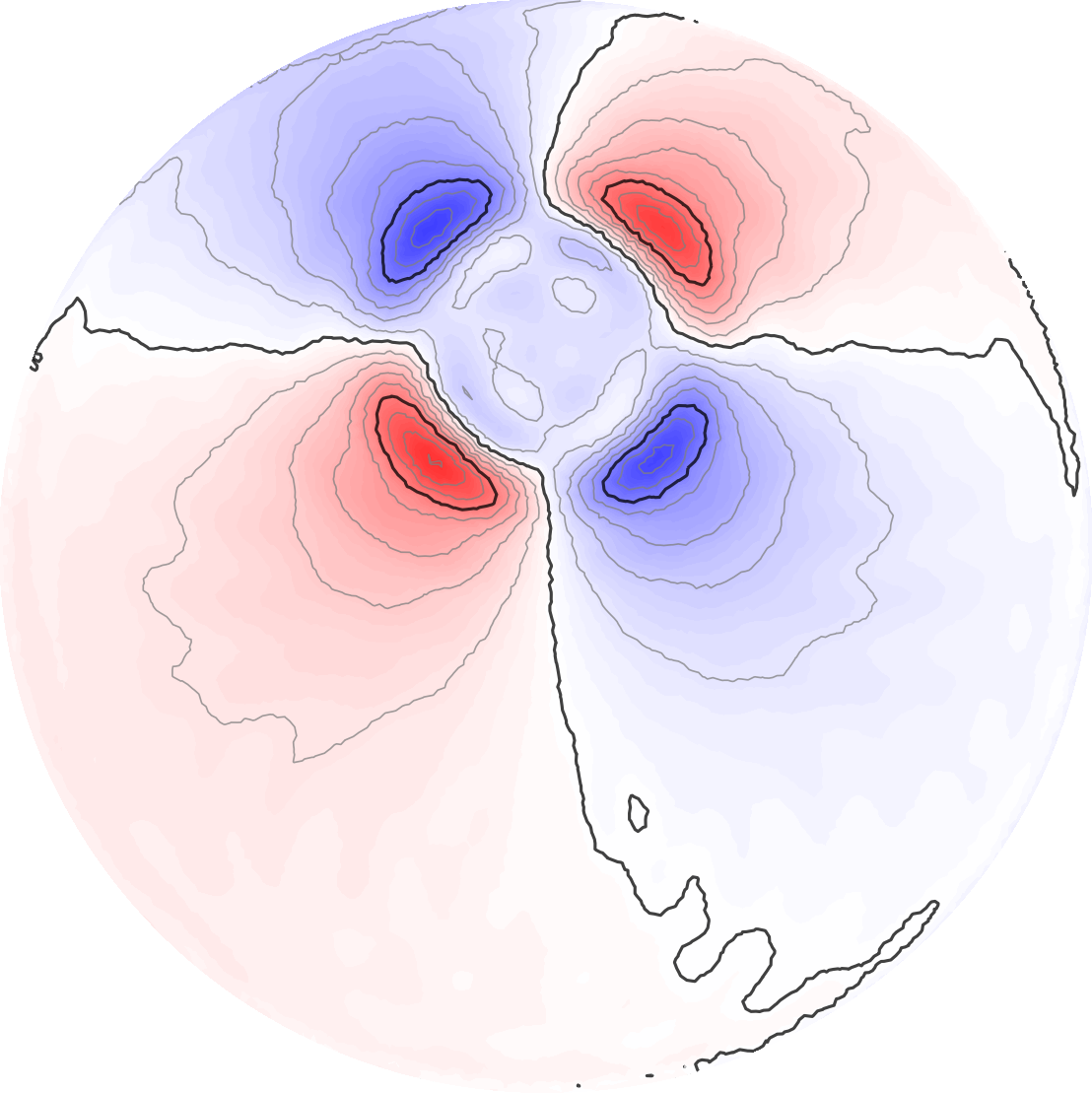}~%
    \gfx{height=0.23\linewidth}{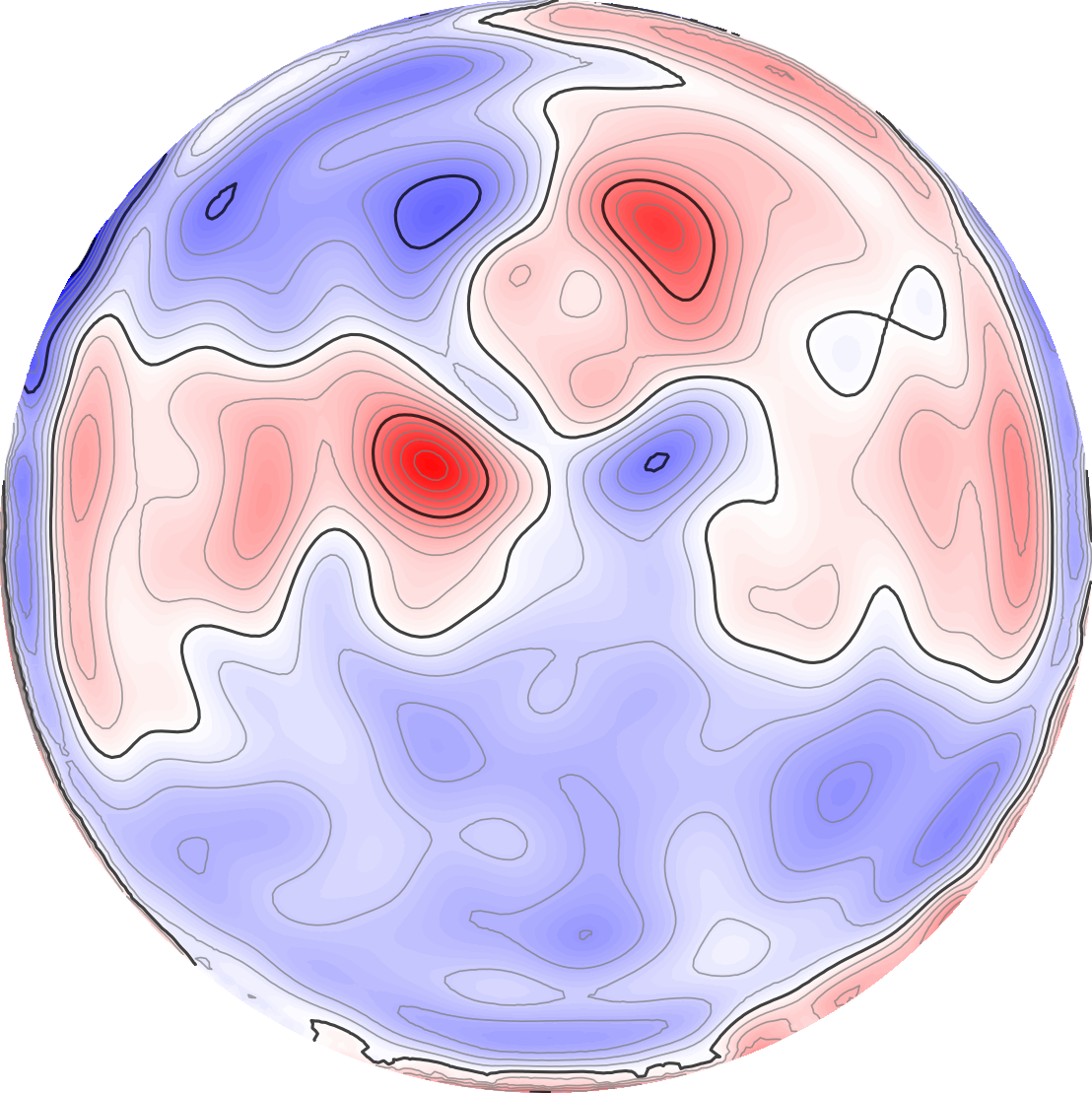}~%
    \gfx{height=0.23\linewidth}{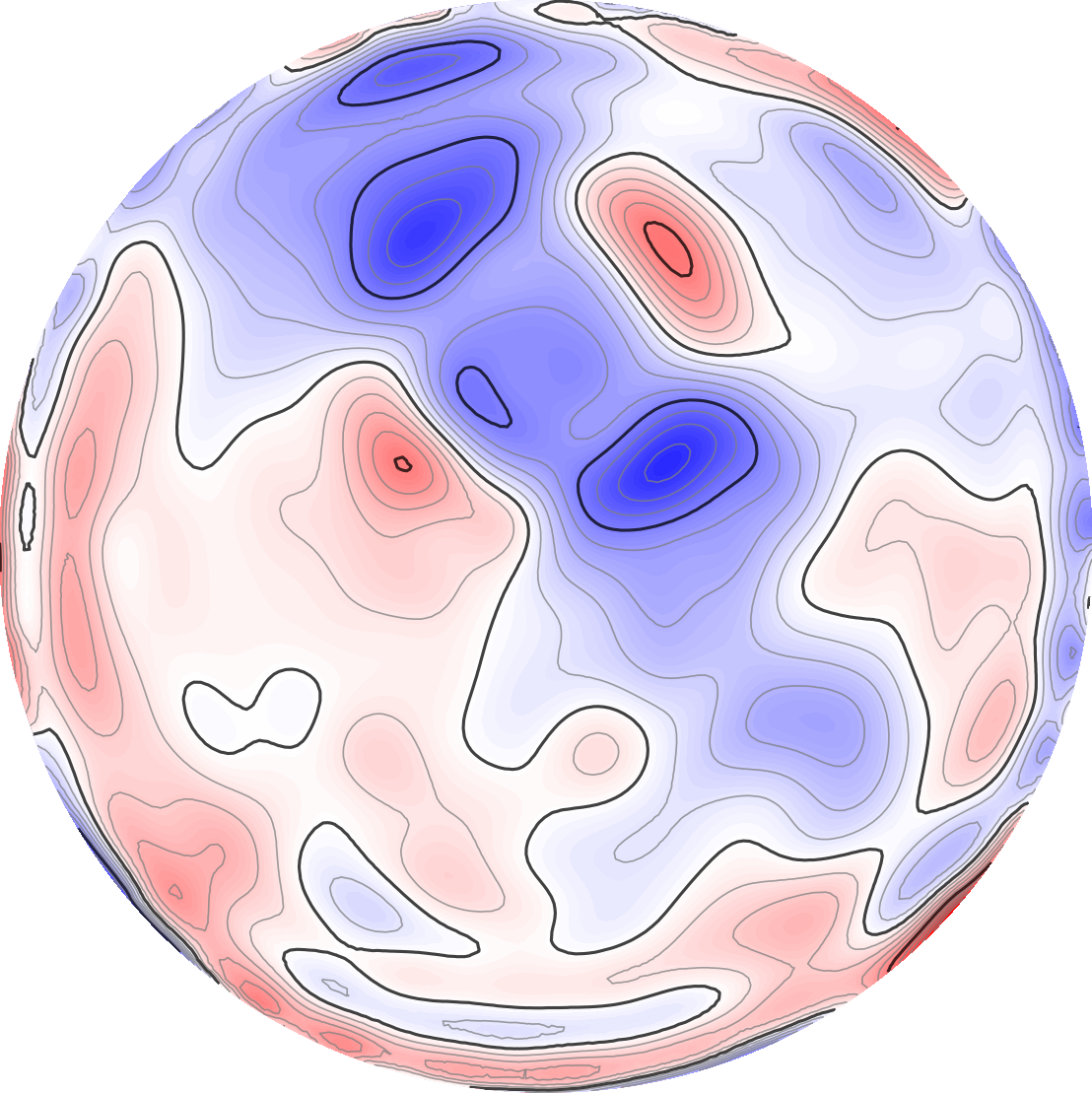}~%
    \gfx{height=0.23\linewidth}{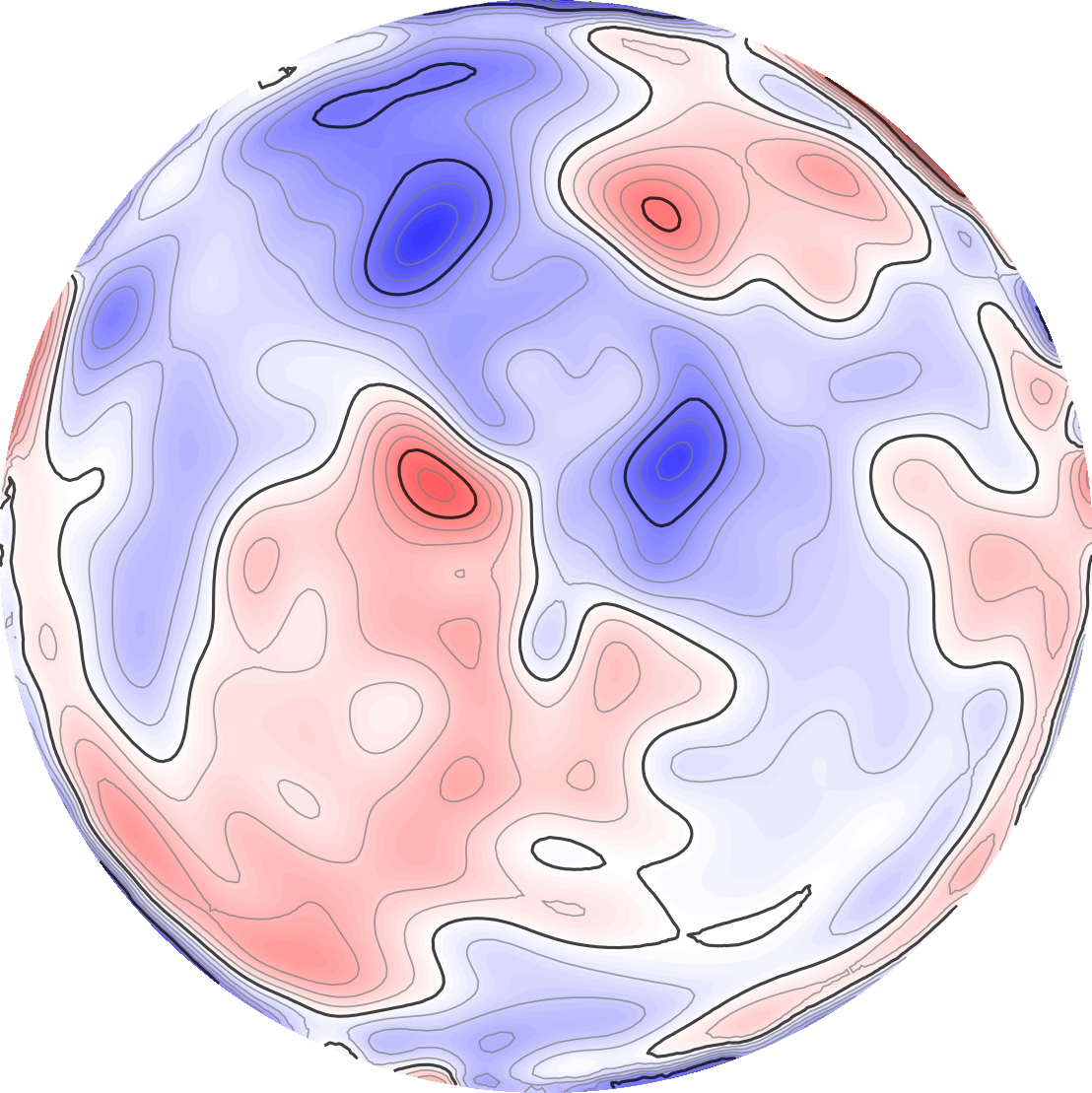}~~%
    \gfx{height=0.23\linewidth}{rec1-cb-H.png}\\
    \gfx{height=0.23\linewidth}{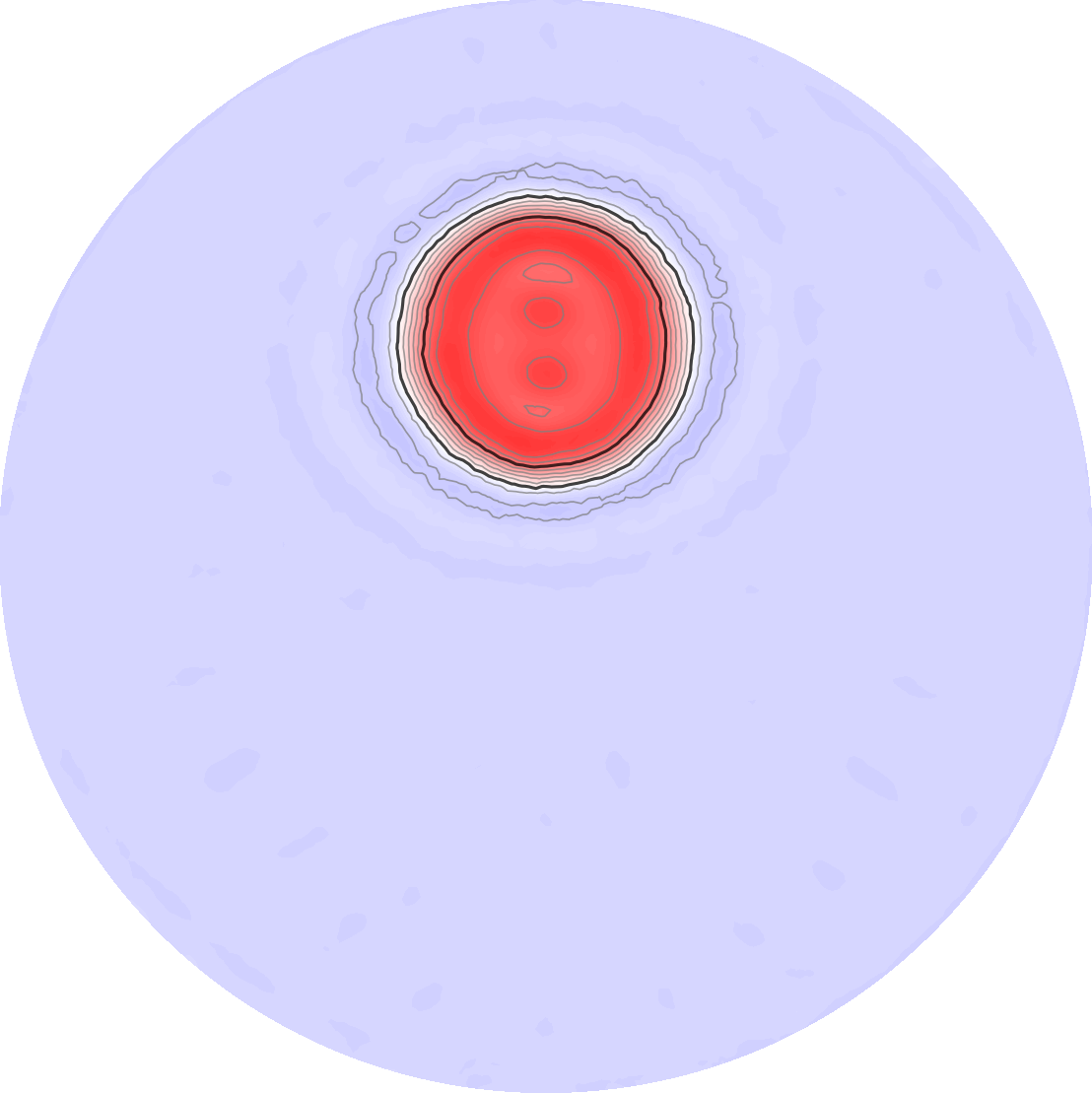}~%
    \gfx{height=0.23\linewidth}{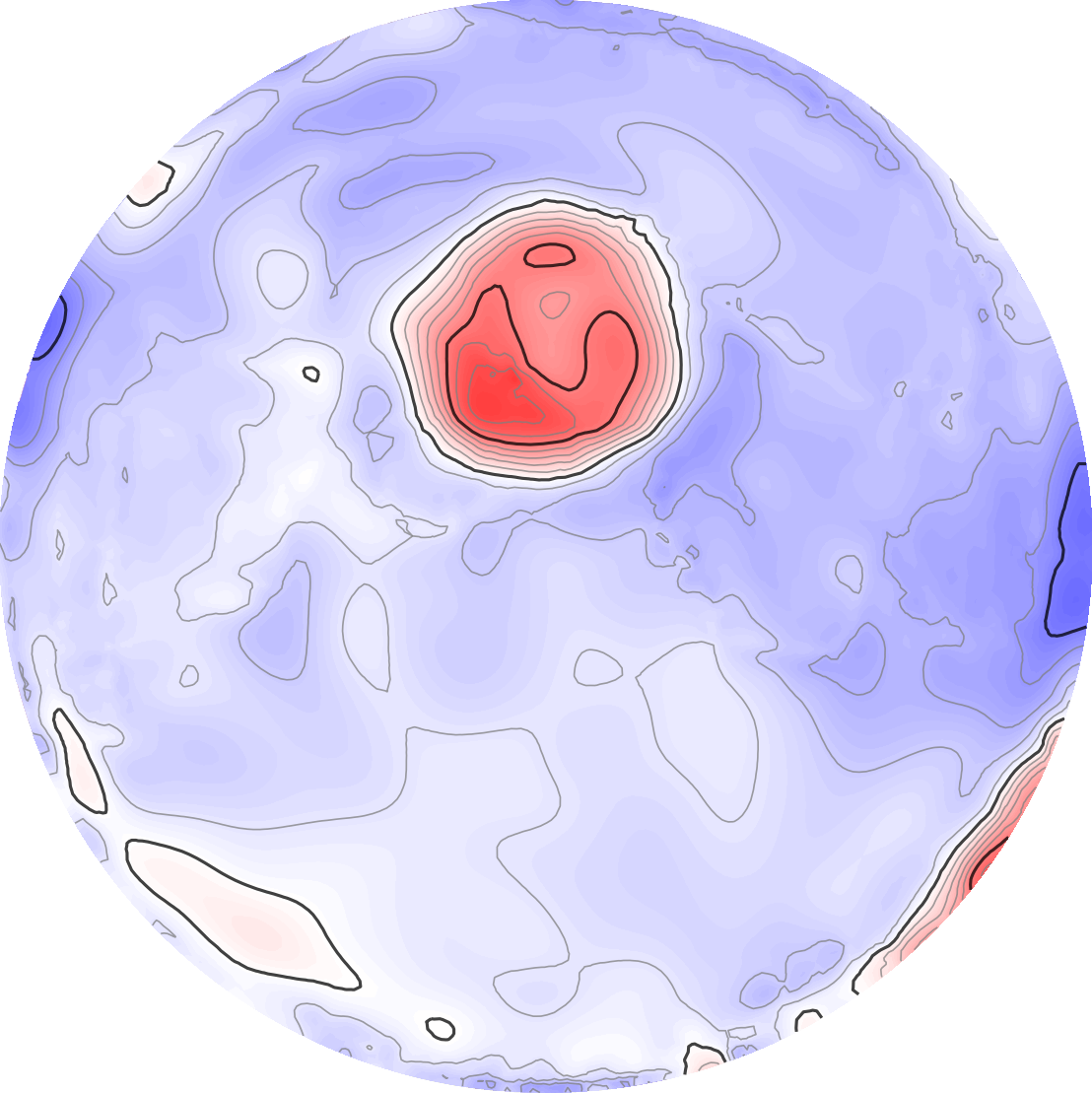}~%
    \gfx{height=0.23\linewidth}{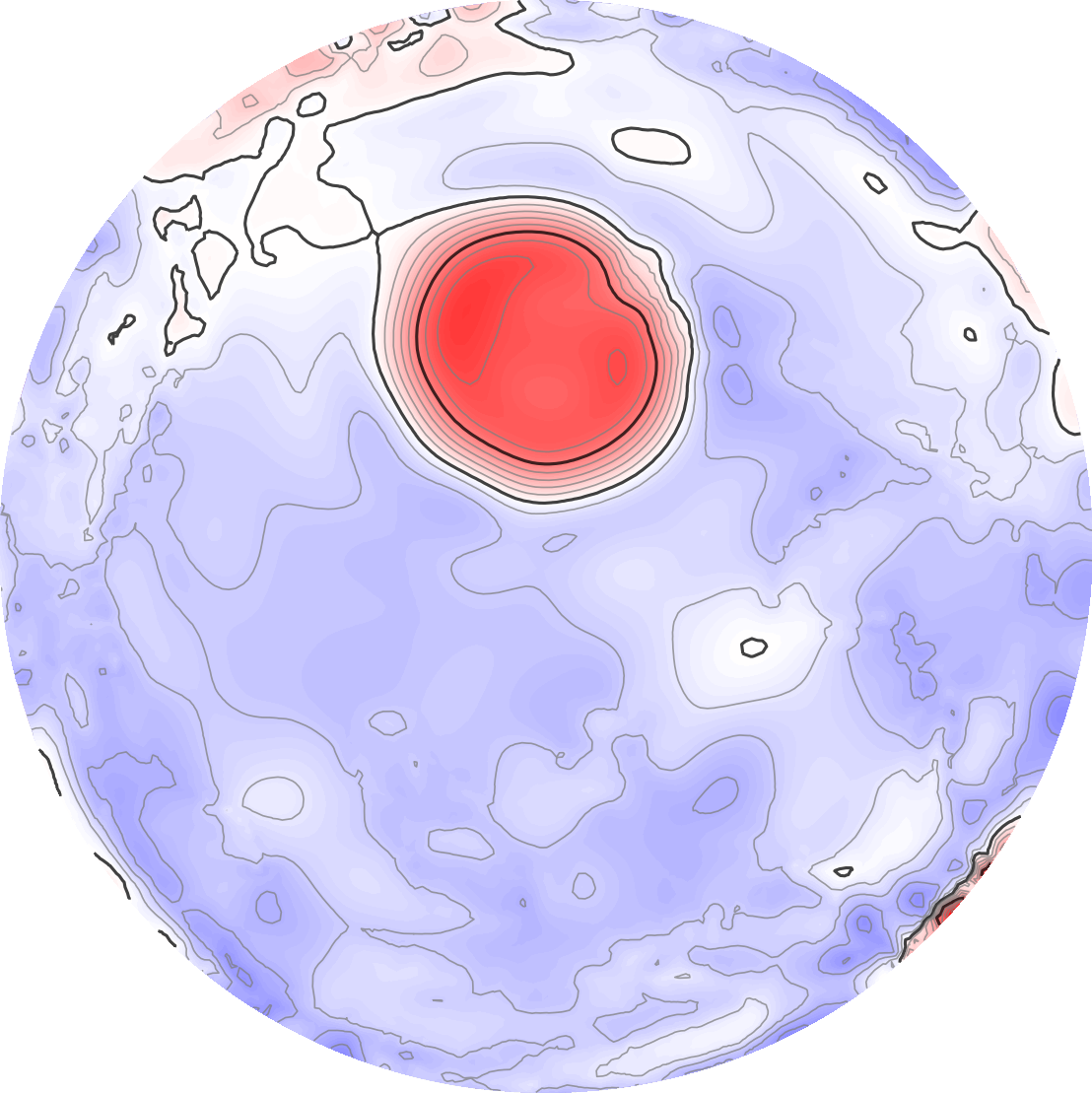}~%
    \gfx{height=0.23\linewidth}{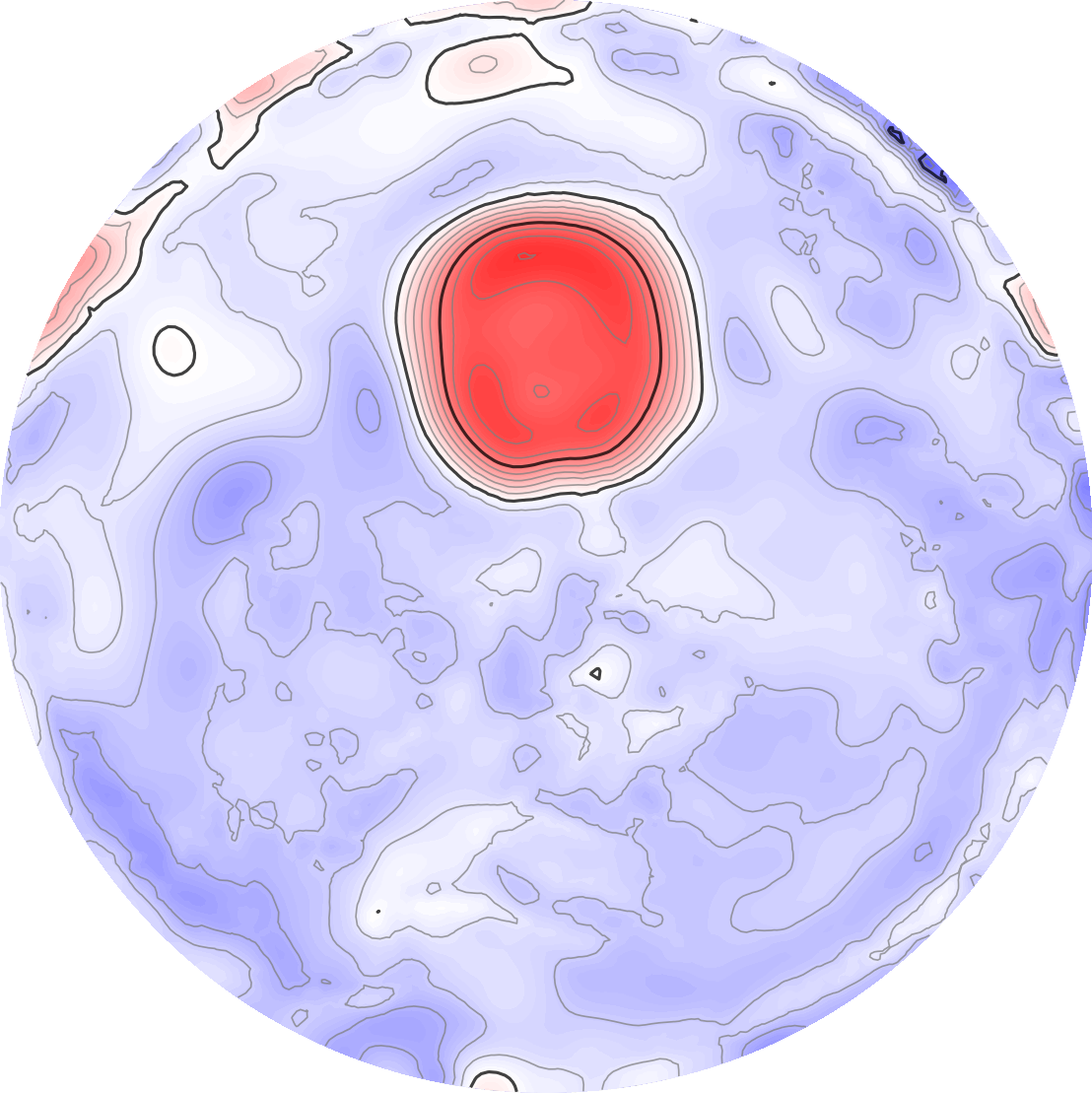}~~%
    \gfx{height=0.23\linewidth}{rec1-cb-S.png}\\
    \caption{Reconstructions of power densities and conductivities corresponding to $ \sigma_1(x) $ and different realizations of sound speeds; $ \mu = 0.05 $. Rows 1--3 are power densities corresponding to the boundary conditions $ f(x_1,x_2) = x_1$, $ x_2 $, $ (x_1+x_2)/\sqrt{2} $ respectively. Row 4 is the reconstructed conductivity corresponding to the three power densities above. For comparison column 1 shows a case with $ s(x) = 0 $; i.e. with the same forward and reconstruction operator. Columns 2--4 correspond to the particular sound speed realizations shown in Figure \ref{fig:real}; left to right. } \label{fig:reconstructions}
\end{figure}

\begin{figure}
    \centering
    \gfx{height=0.23\linewidth}{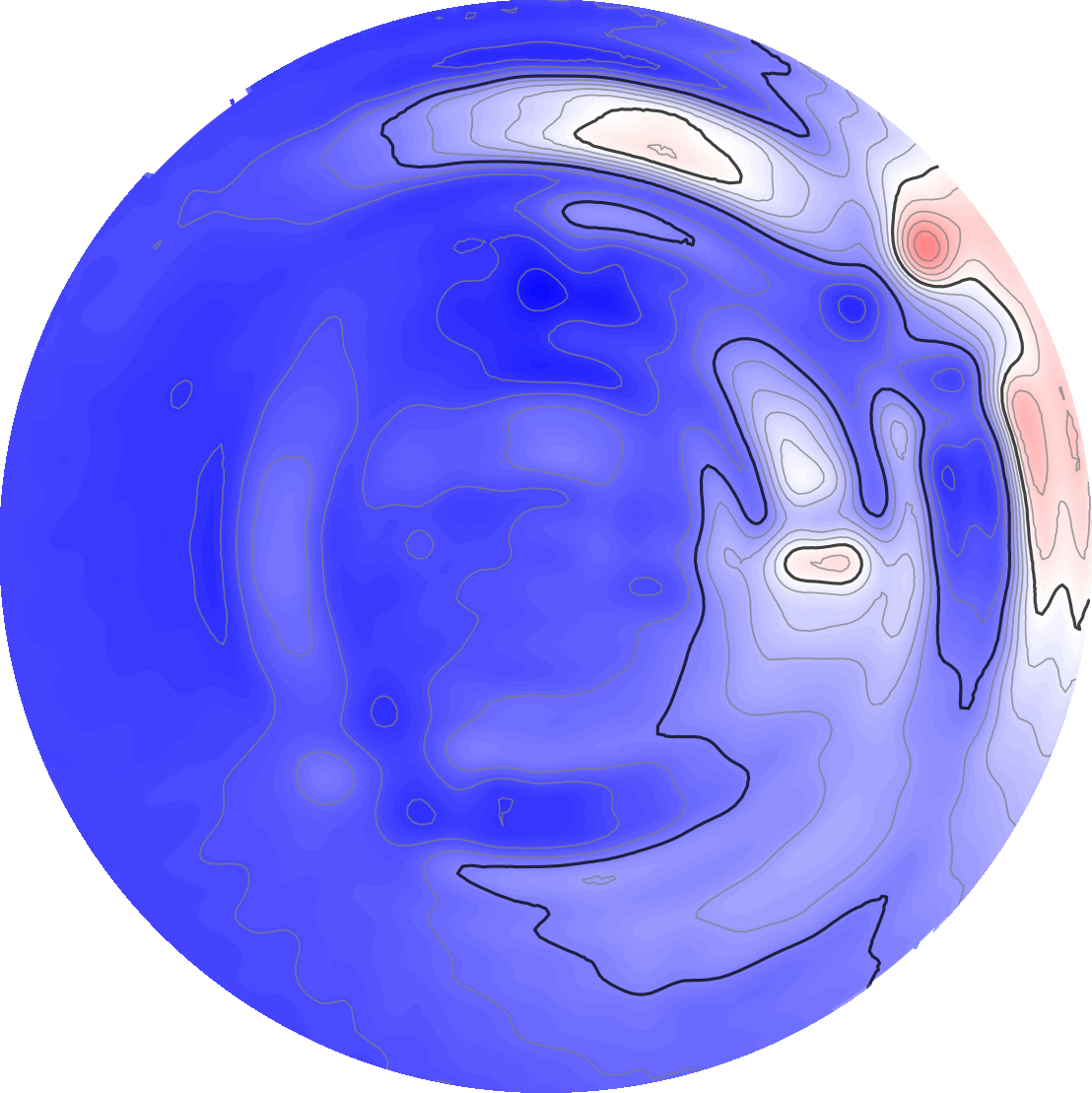}~%
    \gfx{height=0.23\linewidth}{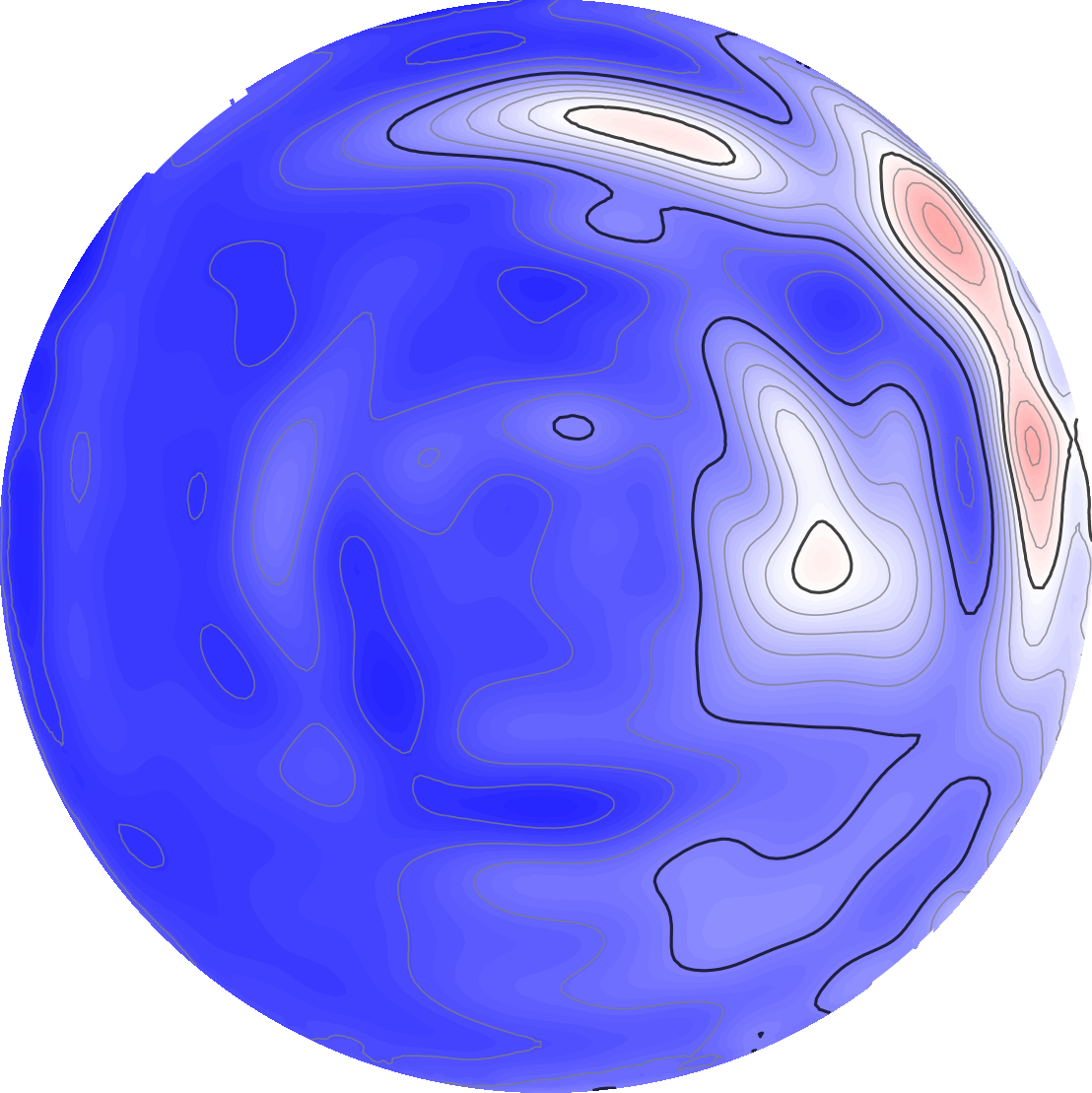}~%
    \gfx{height=0.23\linewidth}{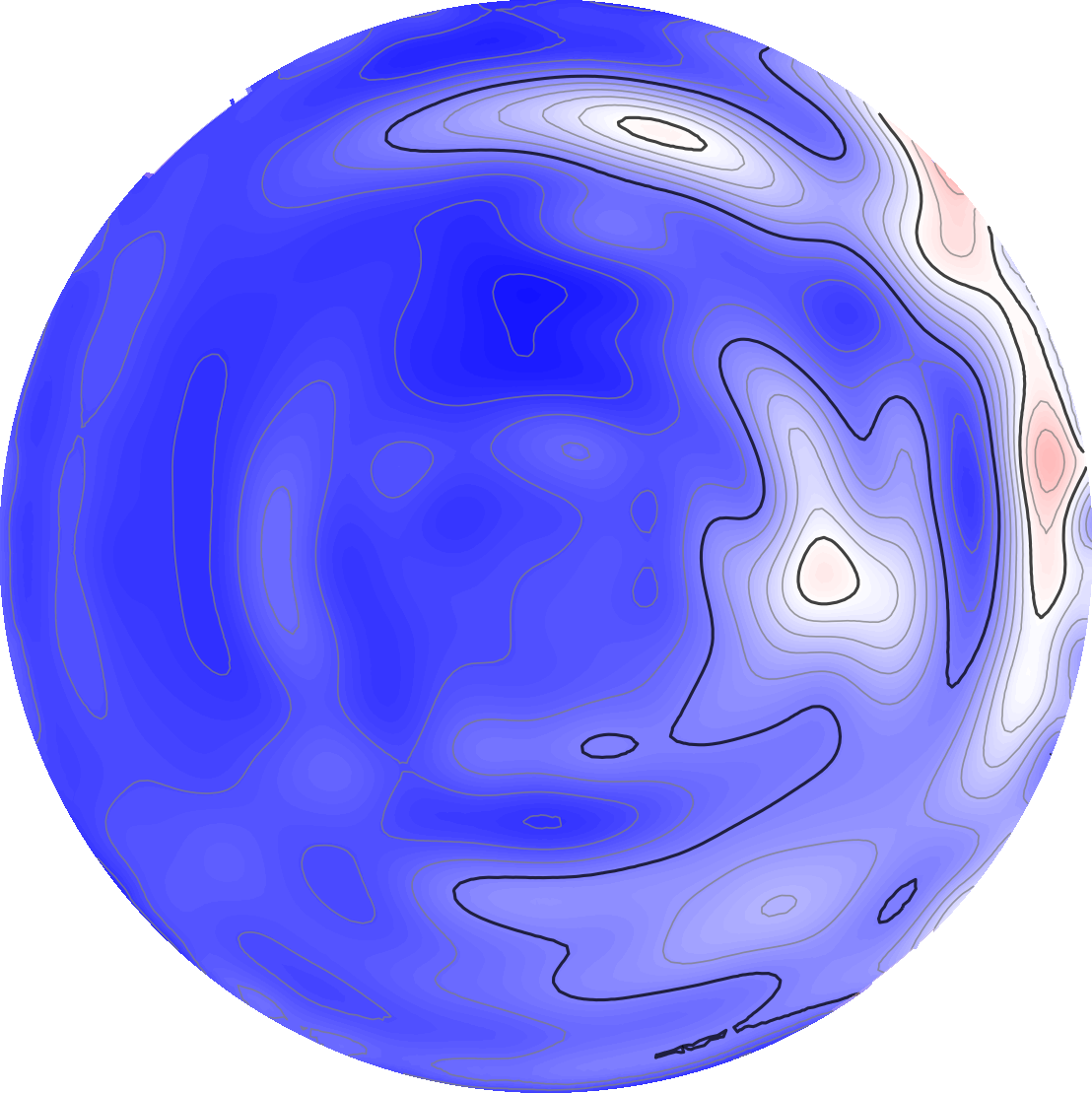}~%
    \gfx{height=0.23\linewidth}{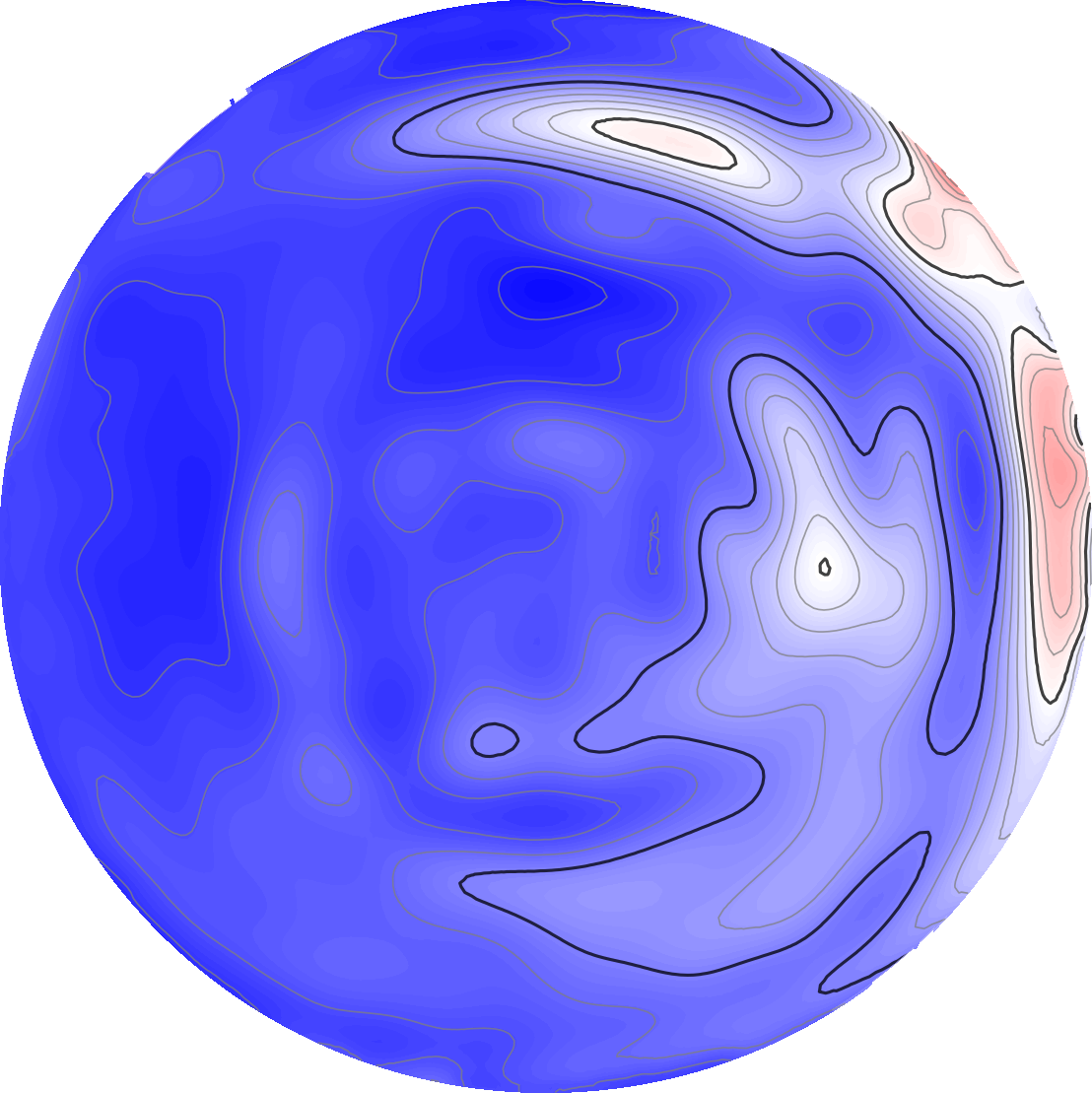}~~%
    \gfx{height=0.23\linewidth}{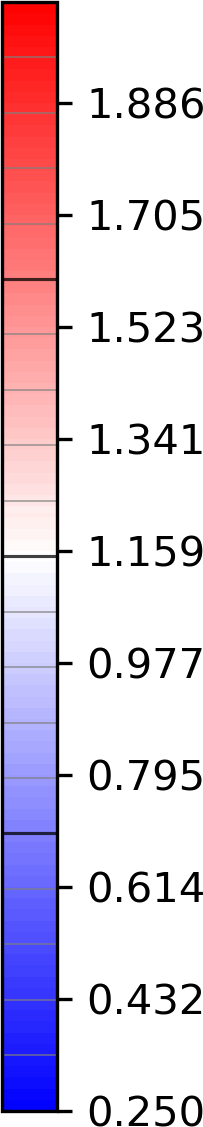}\\
    \gfx{height=0.23\linewidth}{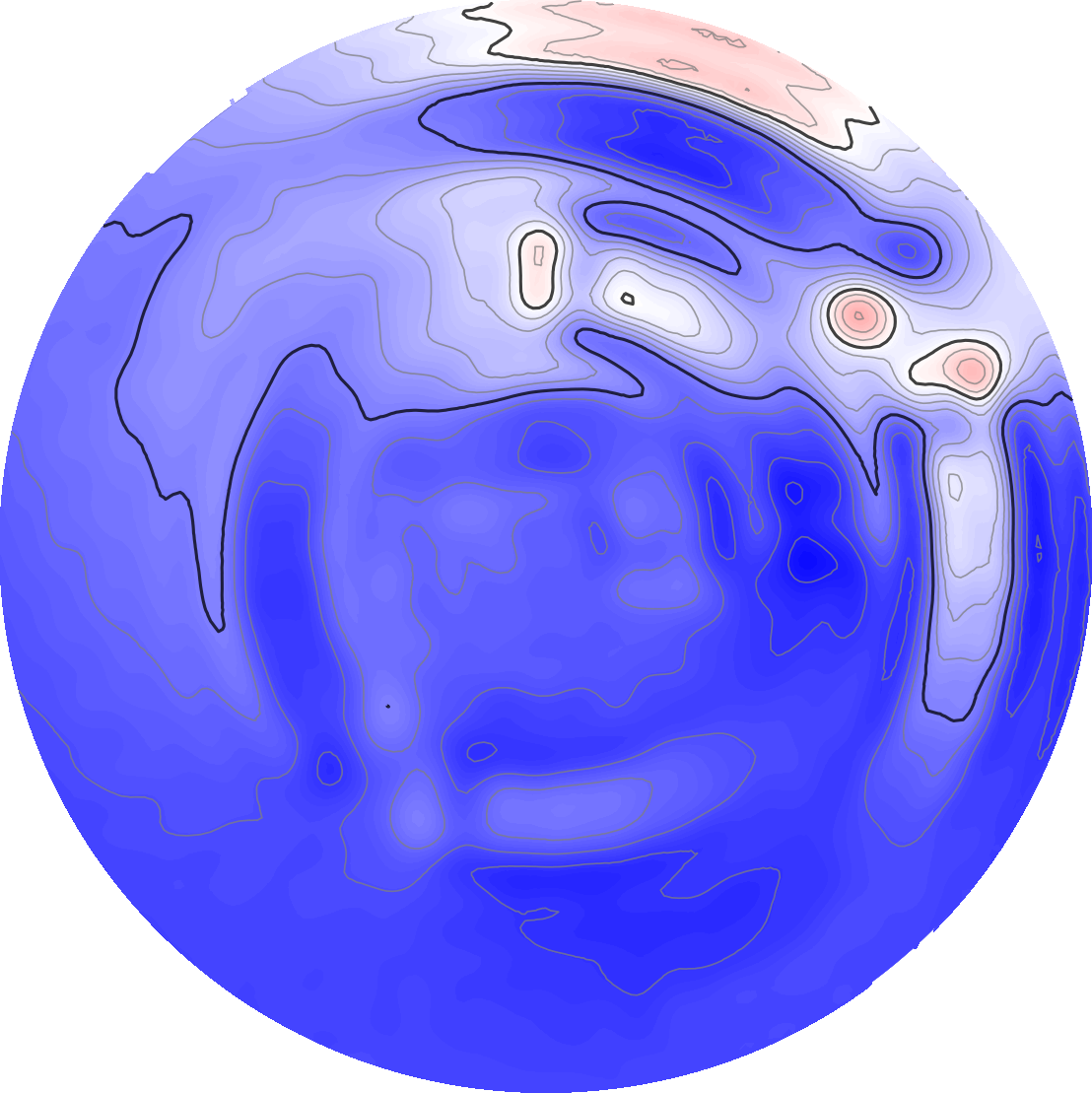}~%
    \gfx{height=0.23\linewidth}{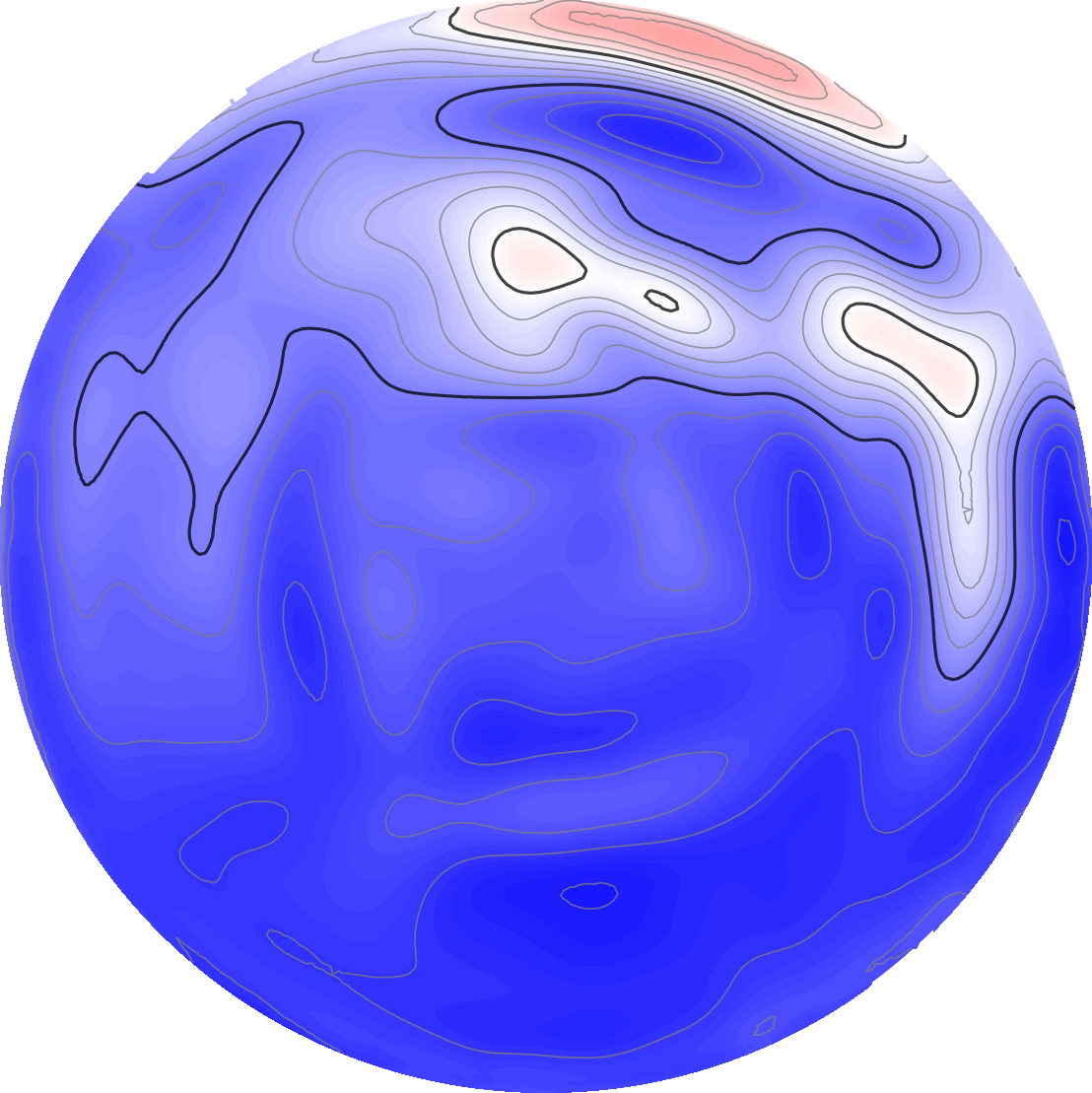}~%
    \gfx{height=0.23\linewidth}{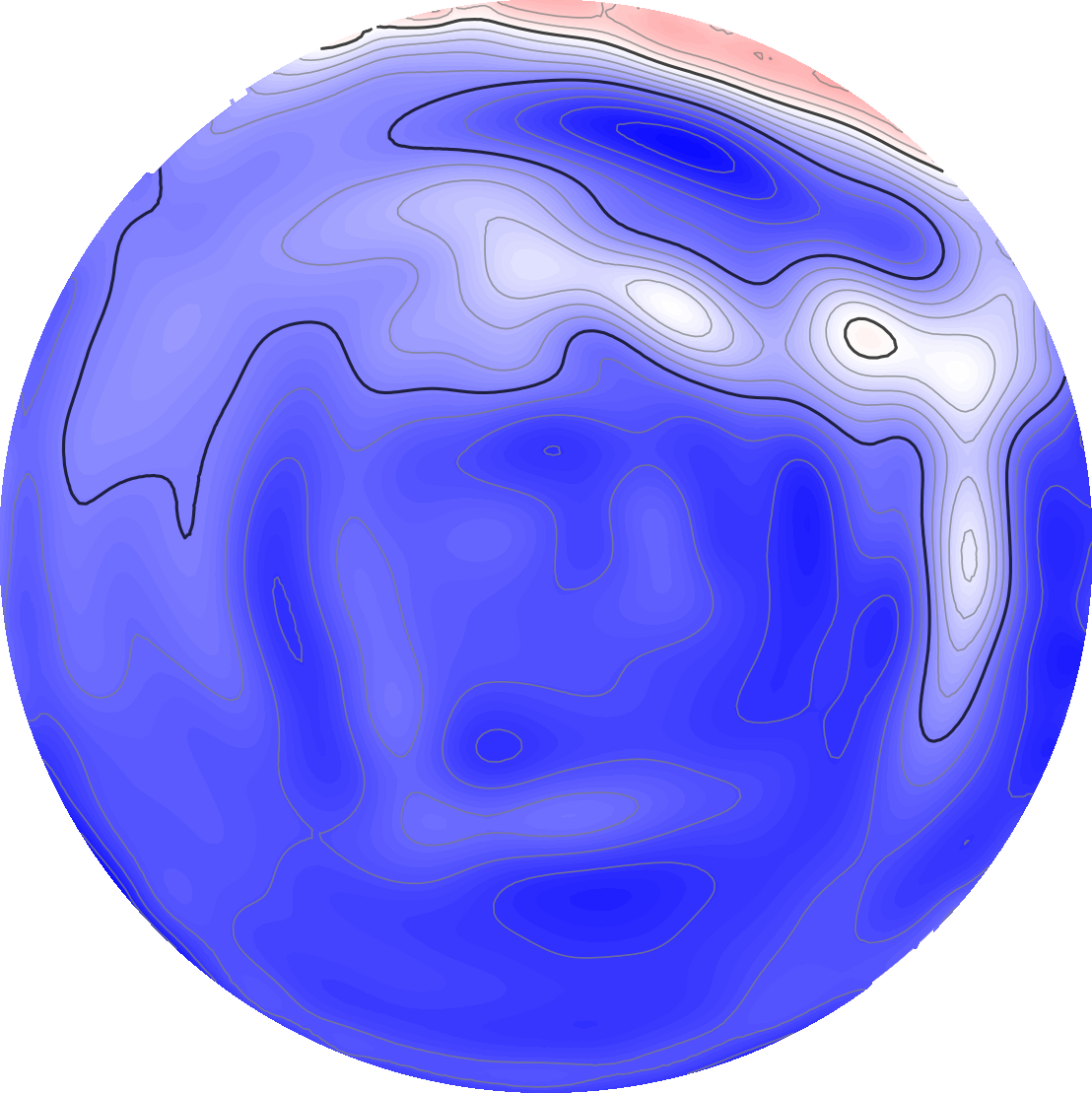}~%
    \gfx{height=0.23\linewidth}{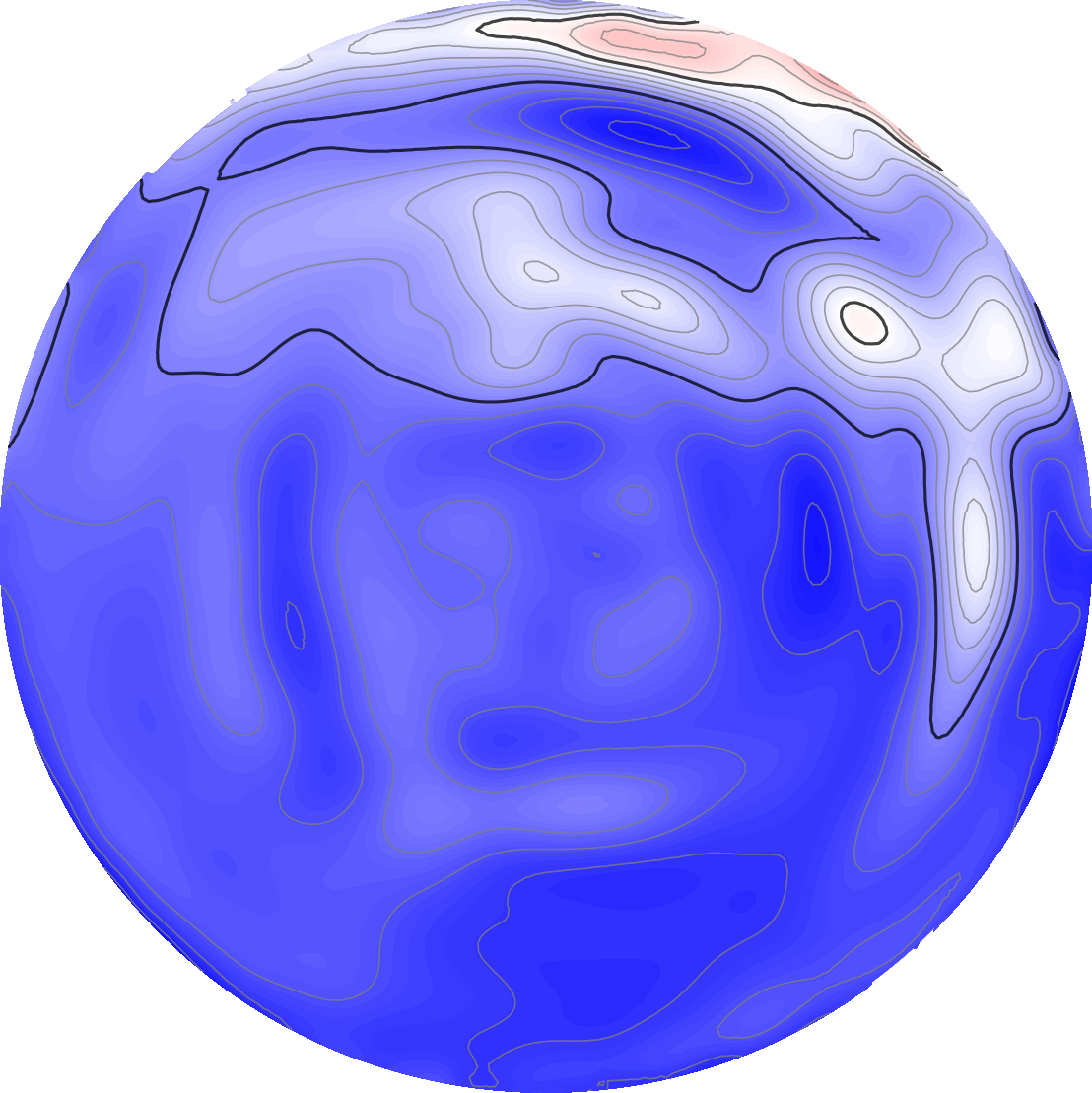}~~%
    \gfx{height=0.23\linewidth}{rec2-cb-H.png}\\
    \gfx{height=0.23\linewidth}{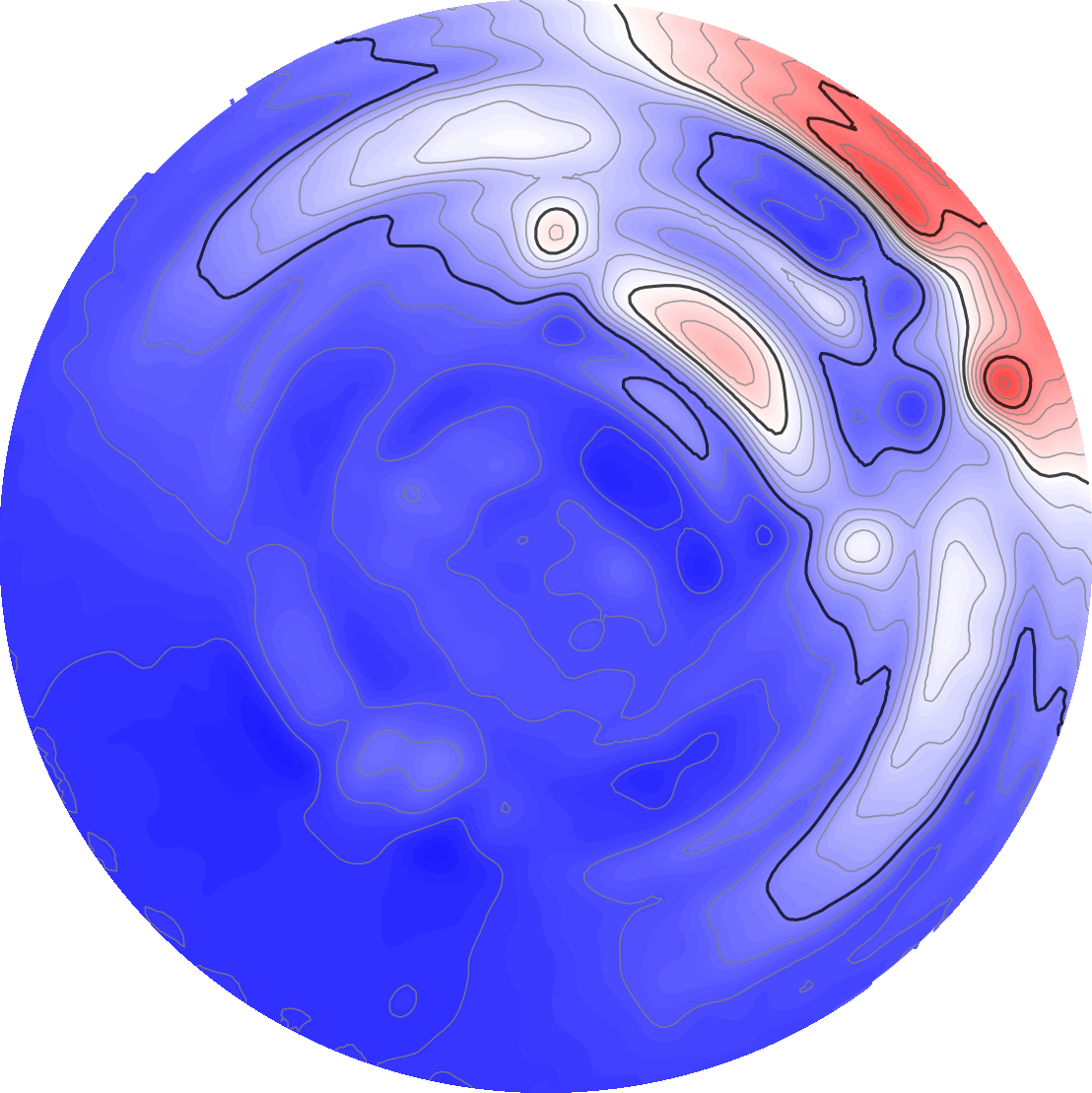}~%
    \gfx{height=0.23\linewidth}{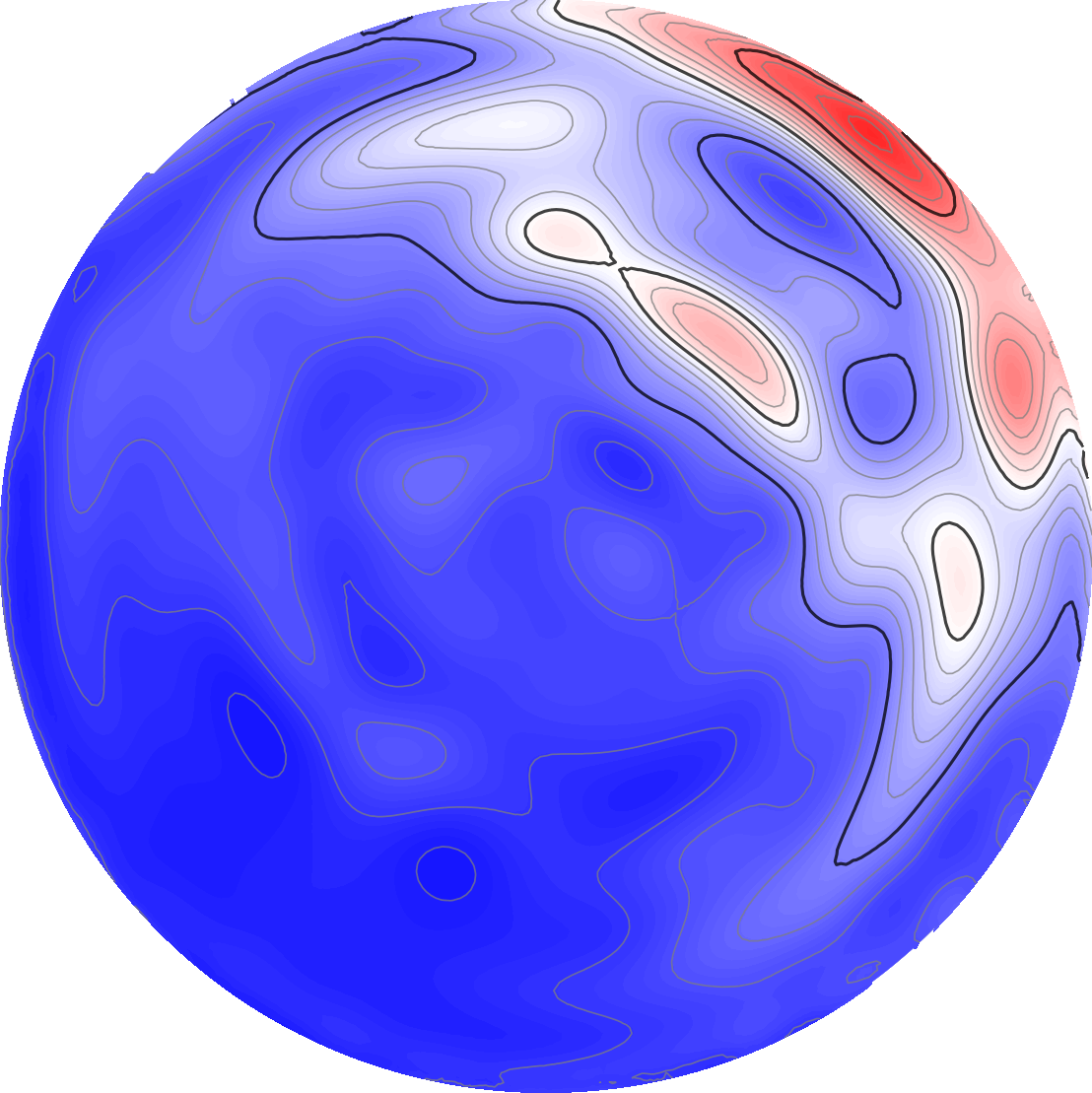}~%
    \gfx{height=0.23\linewidth}{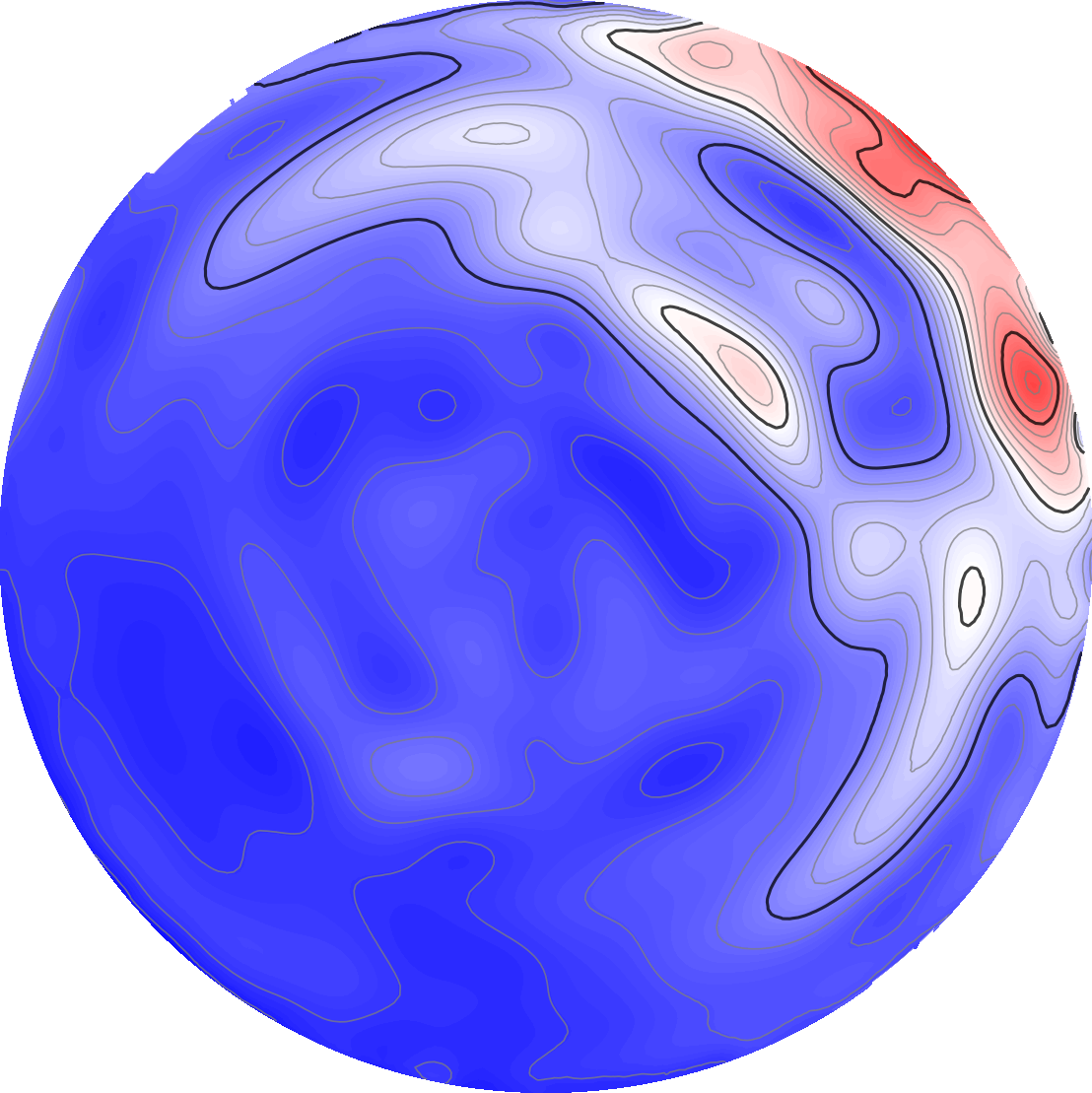}~%
    \gfx{height=0.23\linewidth}{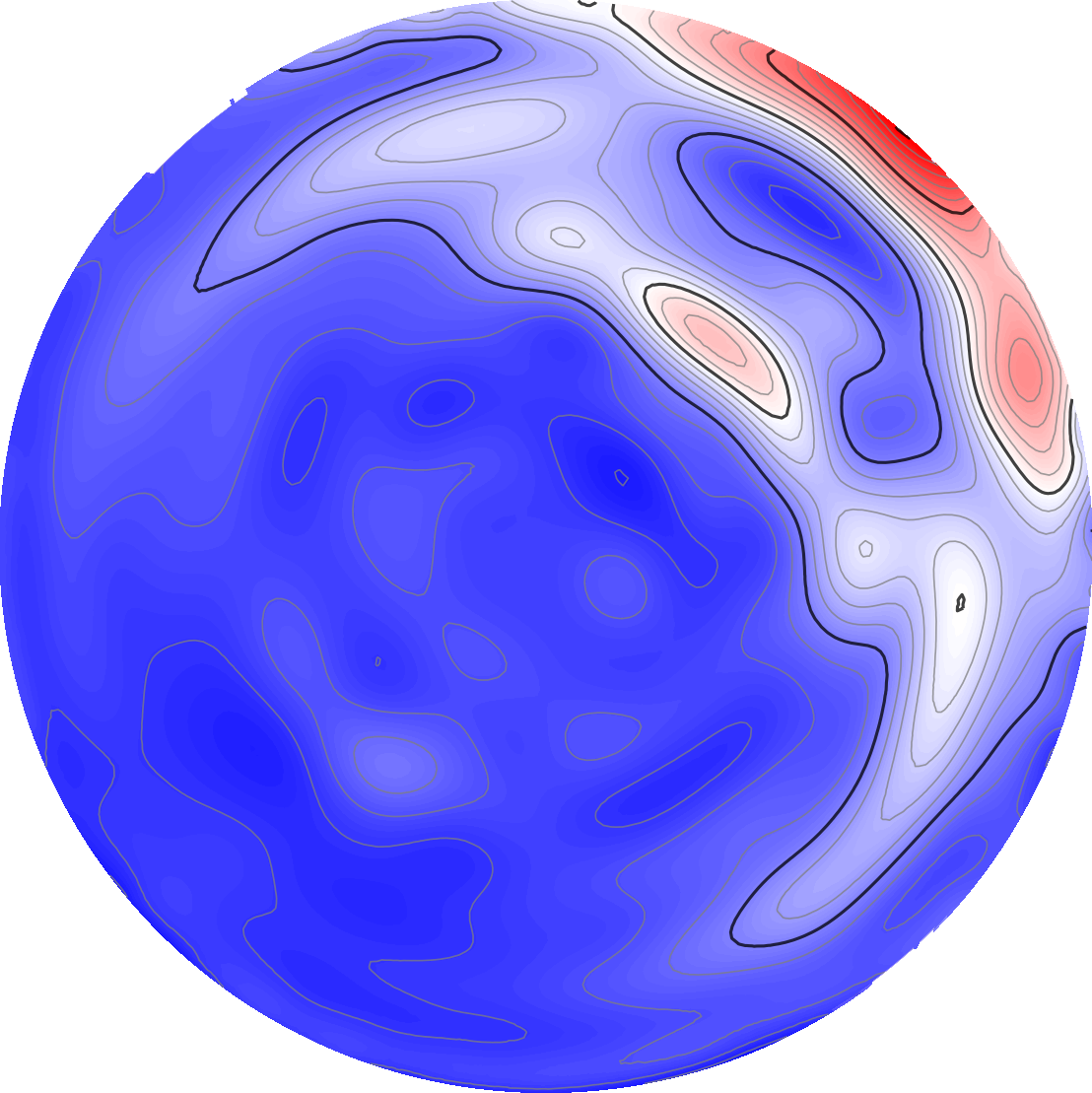}~~%
    \gfx{height=0.23\linewidth}{rec2-cb-H.png}\\
    \gfx{height=0.23\linewidth}{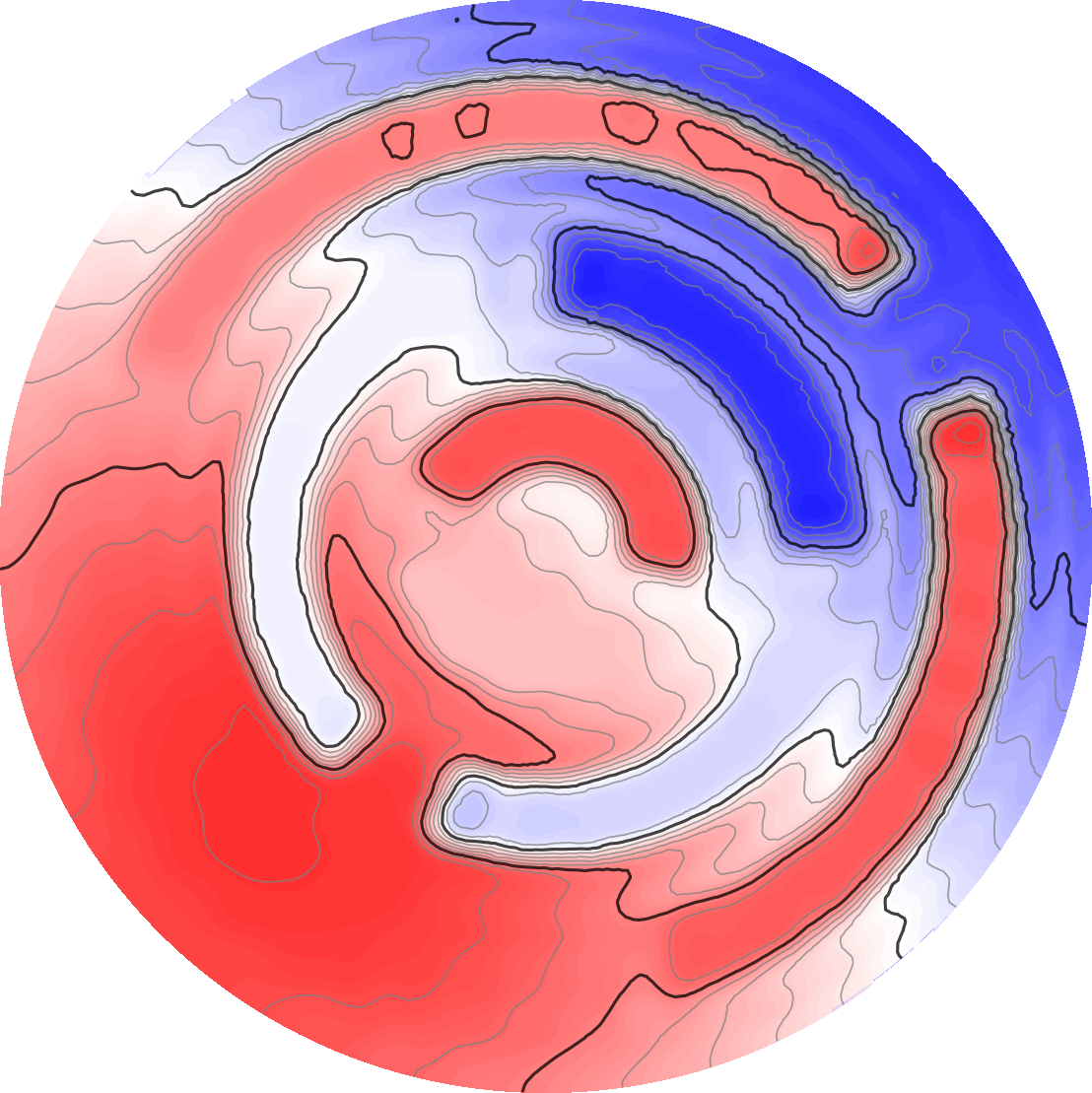}~%
    \gfx{height=0.23\linewidth}{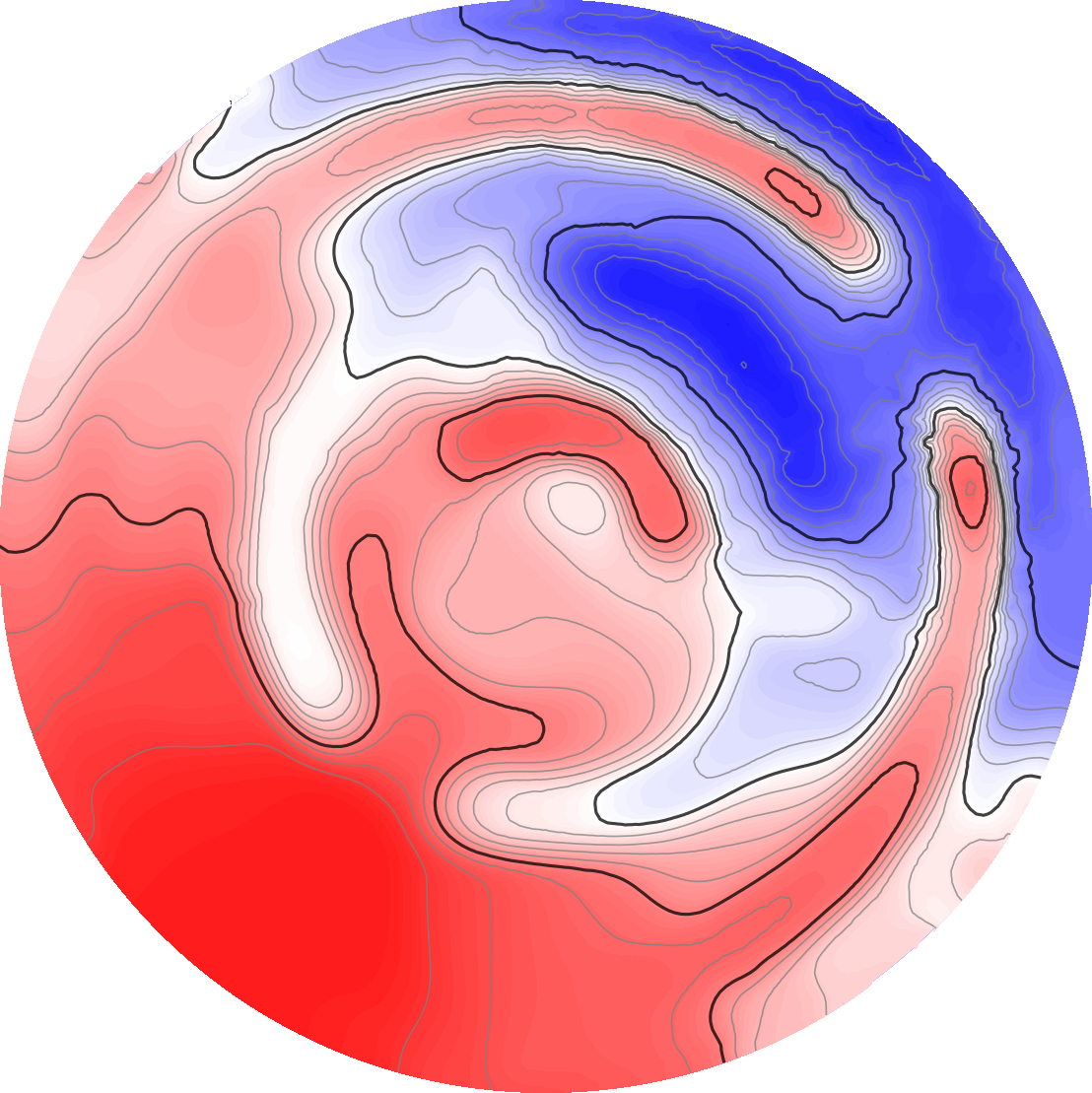}~%
    \gfx{height=0.23\linewidth}{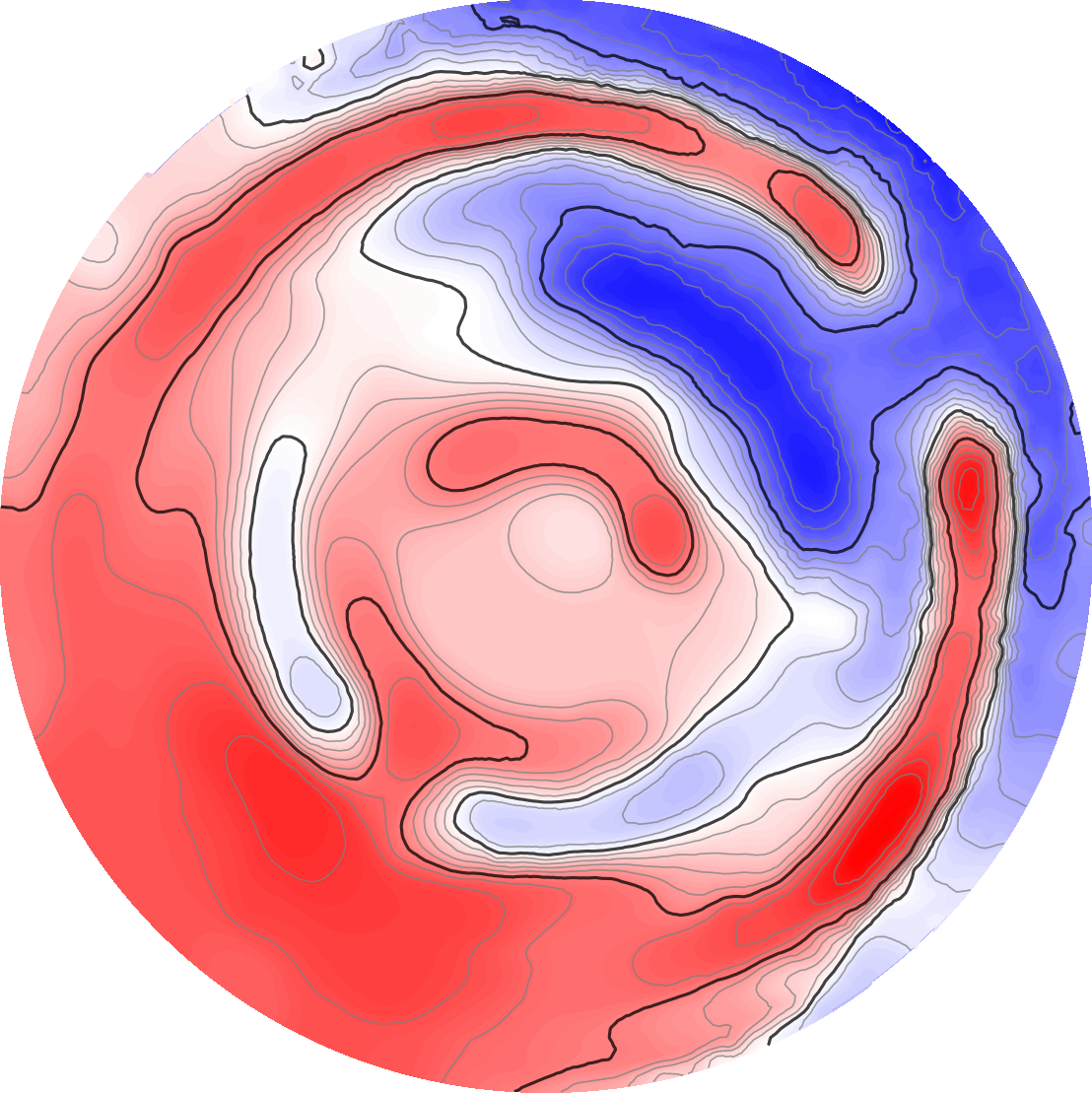}~%
    \gfx{height=0.23\linewidth}{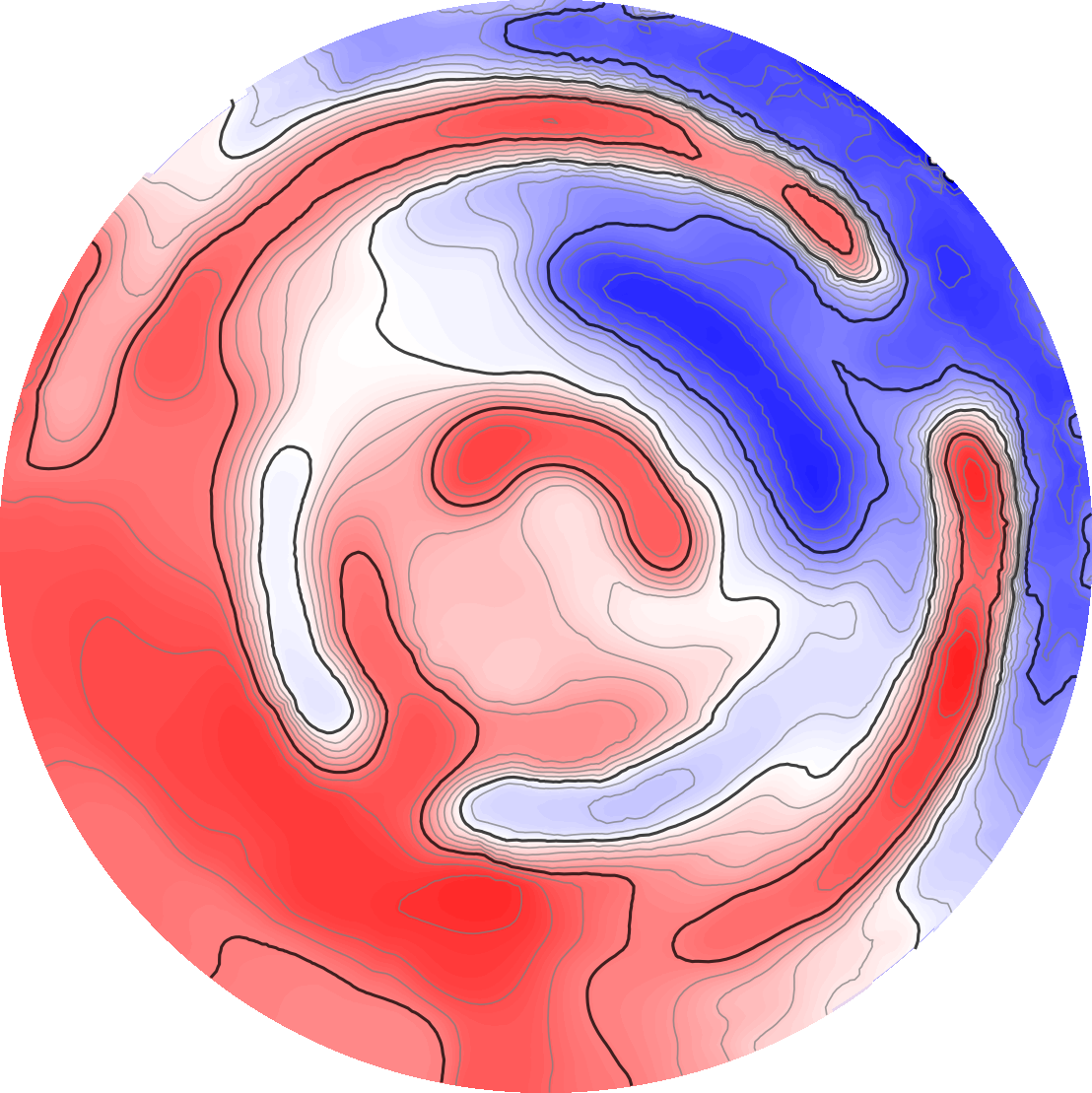}~~%
    \gfx{height=0.23\linewidth}{rec2-cb-S.png}\\
    \caption{Reconstructions of power densities and conductivities corresponding to $ \sigma_2(x) $ and different realizations of sound speeds; $ \mu = 0.05 $. Rows 1--3 are power densities corresponding to the boundary conditions $ f(x_1,x_2) = x_1$, $ x_2 $, $ (x_1+x_2)/\sqrt{2} $ respectively. Row 4 is the reconstructed conductivity corresponding to the three power densities above. For comparison column 1 shows a case with $ s(x) = 0 $; i.e. with the same forward and reconstruction operator. Columns 2--4 correspond to the particular sound speed realizations shown in Figure \ref{fig:real}; left to right. } \label{fig:reconstructions2}
\end{figure}

Looking at columns 2 through 4 in Figure \ref{fig:reconstructions}, we clearly see the propagation of the model error to the power density, $ H(x) $, reconstructions, though, due to the complexity of the error propagation, we remark how the error features have no obvious resemblance to the structures introduced in the sound speed for the data generation; we refer back to Figure \ref{fig:real}. A number of these error features seemingly disappear again when moving on to the reconstruction of the conductivity, $ \sigma(x) $, though clearly the background variations are notable, and in particular near the boundary we find artifacts. The latter should not be surprising, as the mean sound speed might be slightly different from the constant $ c_\textup{bg} $ used in recovery, which during this study has been observed to cause structural errors close to the boundary; we refer again to the example in Figure \ref{fig:artifacts}. The inclusion, however, stands out quite clearly, which we attribute to the high and low peaks around the inclusion area present in all the power density reconstructions. Similar encircling high-low peak structures do not appear elsewhere attributing to the non-presence of other inclusions introduced by model errors.

Similar phenomenons as in Figure \ref{fig:reconstructions} appear in the reconstructions in Figure \ref{fig:reconstructions2}, where the reconstructed power densities clearly exhibit disruptions of a similar fashion. The power densities are mostly concentrated near the north east area, which is due to the higher contrasts in conductivity in that area. In general, despite the disruptive model errors introduced, the reconstructions remain very good with both shapes and structures, along with smooth features, recognizable.

\begin{figure}
    \centering
    \gfx{height=0.23\linewidth}{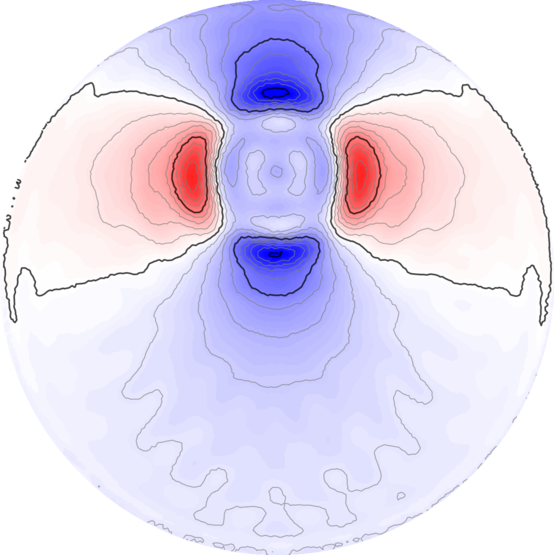}~%
    \gfx{height=0.23\linewidth}{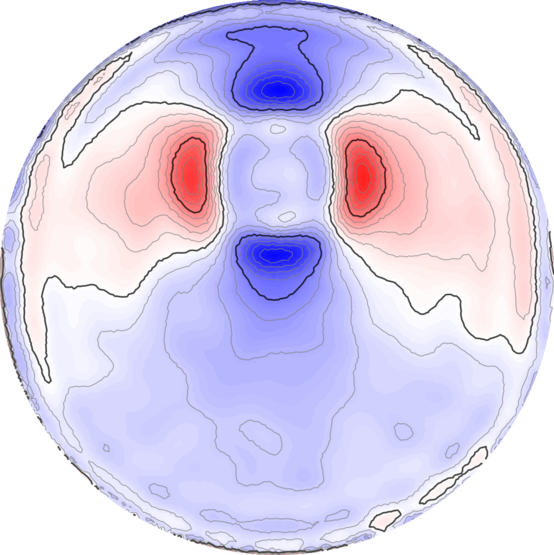}~%
    \gfx{height=0.23\linewidth}{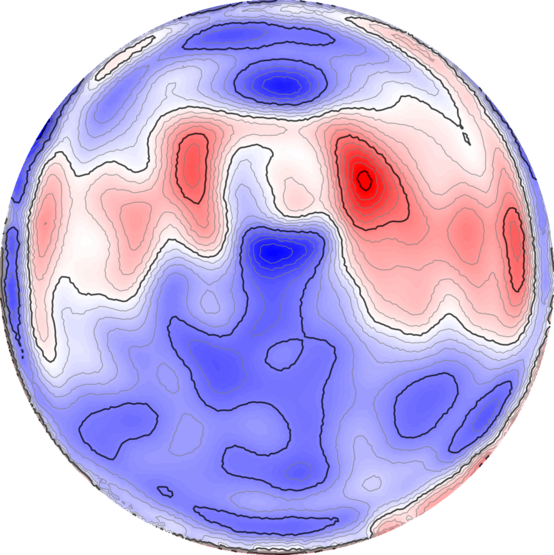}~%
    \gfx{height=0.23\linewidth}{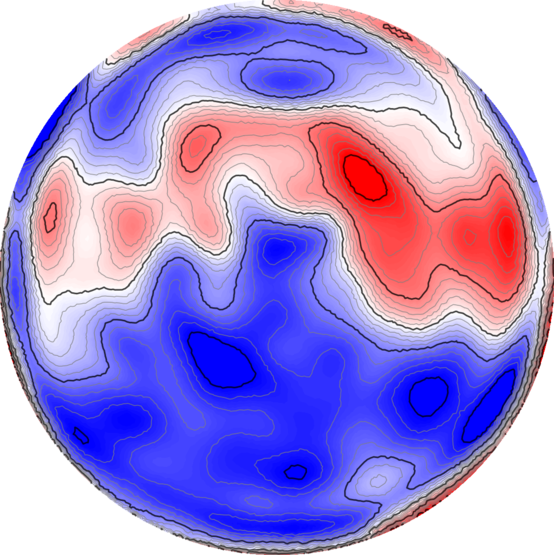}~~%
    \gfx{height=0.23\linewidth}{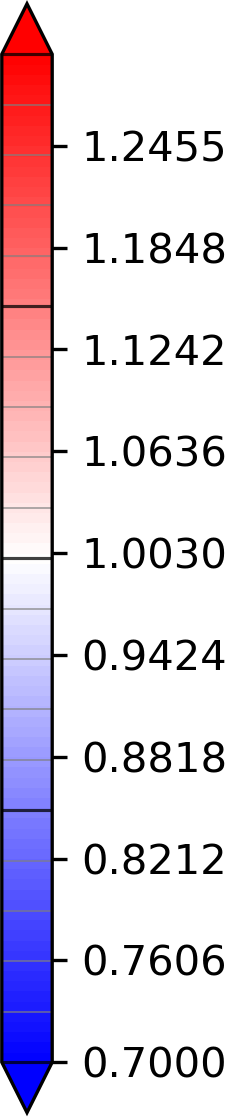}
    \caption{Reconstructions of power densities corresponding to different values of noise level $ \mu $. The power density is the one corresponding to the boundary condition $ f(x_1,x_2) = x_1 $ and the sound speed structure $ s(x) $ used is the left-most in Figure \ref{fig:real}. The values used for $ \mu $ are (left-to-right) 0.00, 0.01, 0.05 and 0.10.} \label{fig:var-lvl}
\end{figure}

We illustrate in Figure \ref{fig:var-lvl} the change in the reconstructed power density relative to the level of difference in sound speed. In the figure progressing from left to right, the scale of the structured variation, controlled by $ \mu $, is scaled up. In the left-most reconstruction $ \mu = 0 $, thus the only error in the sound speed is within the inclusion. The third reconstruction has $ \mu = 0.05 $ and we note that this is the same reconstruction as row 1--column 2 in Figure \ref{fig:reconstructions}. This image demonstrates quite nicely the theory developed in the preceding sections. 

\begin{figure}
    \centering
    \gfx{height=0.26\linewidth}{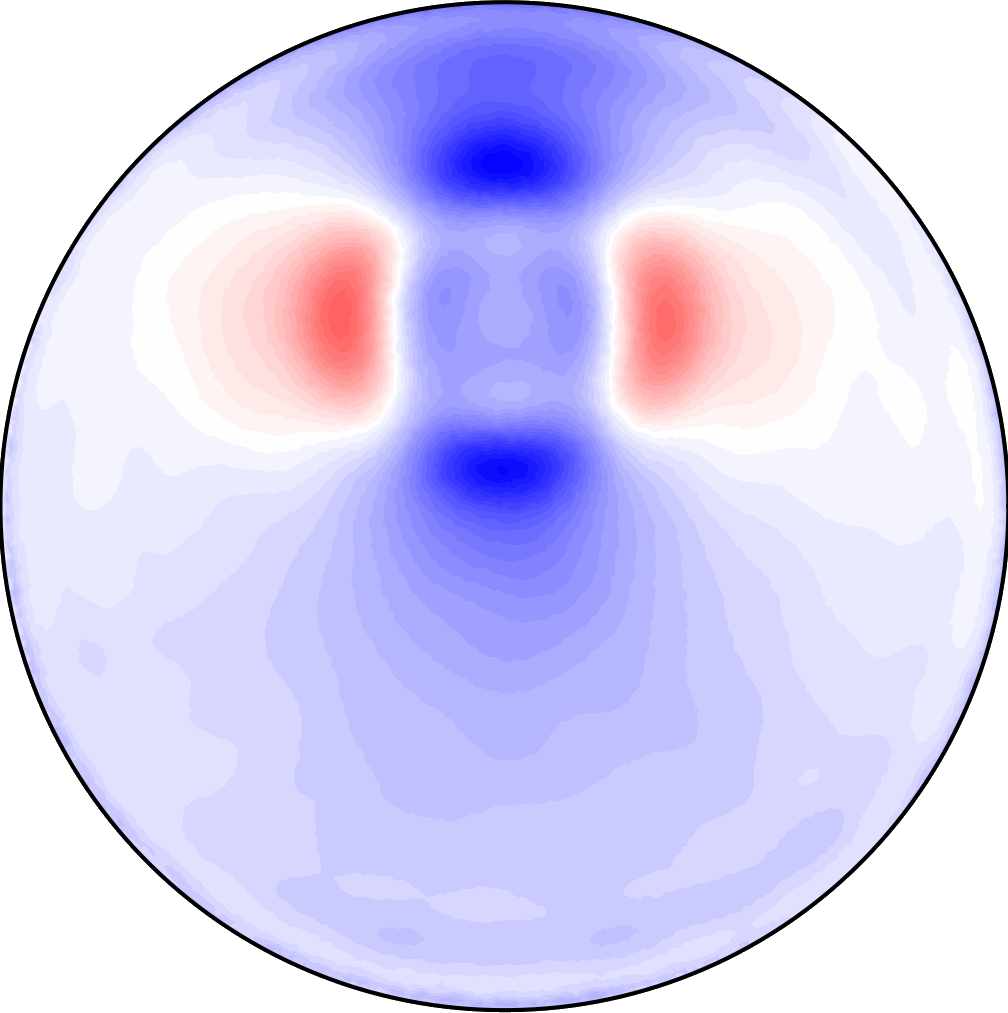}~%
    \gfx{height=0.26\linewidth}{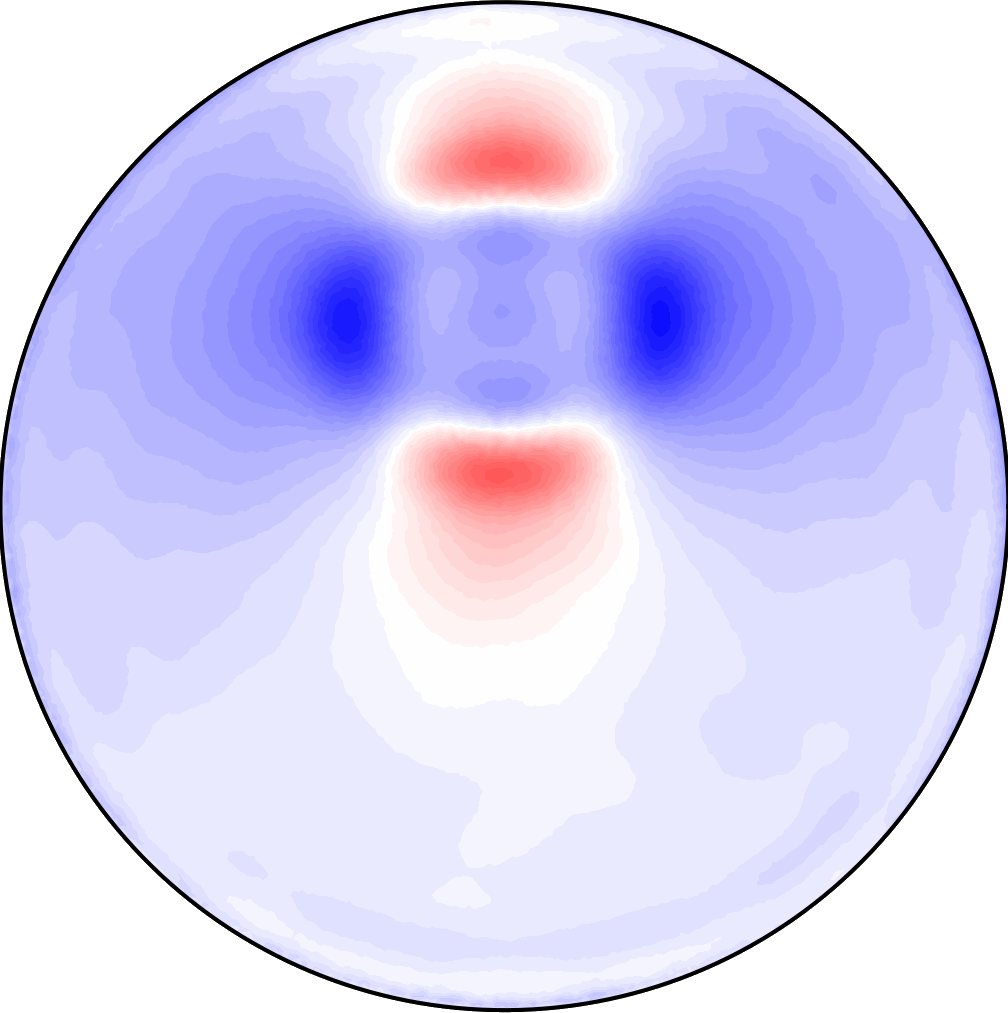}~%
    \gfx{height=0.26\linewidth}{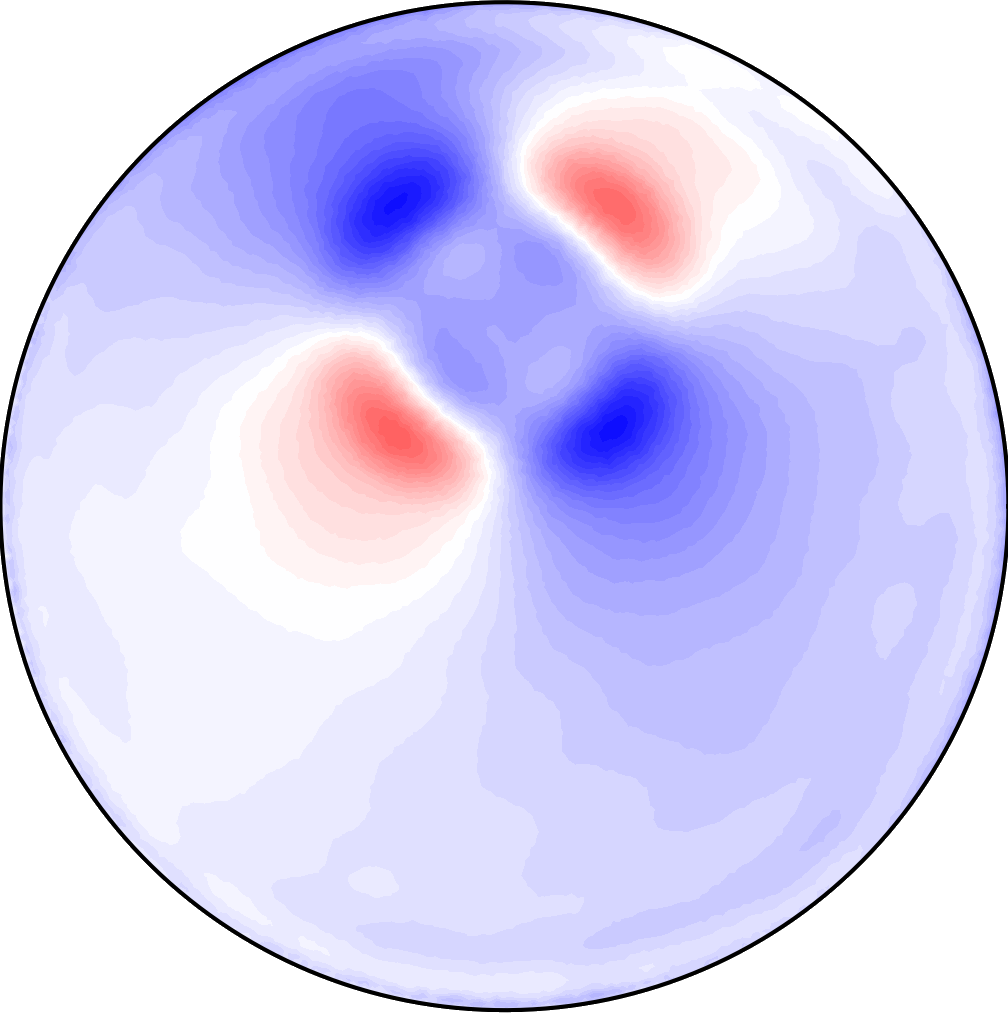}~~%
    \gfx{height=0.26\linewidth}{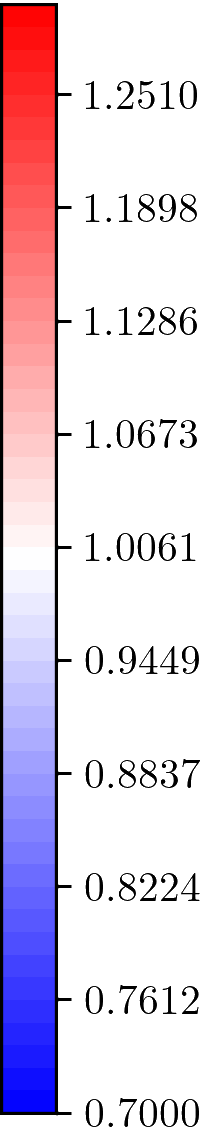}\\
    \gfx{height=0.26\linewidth}{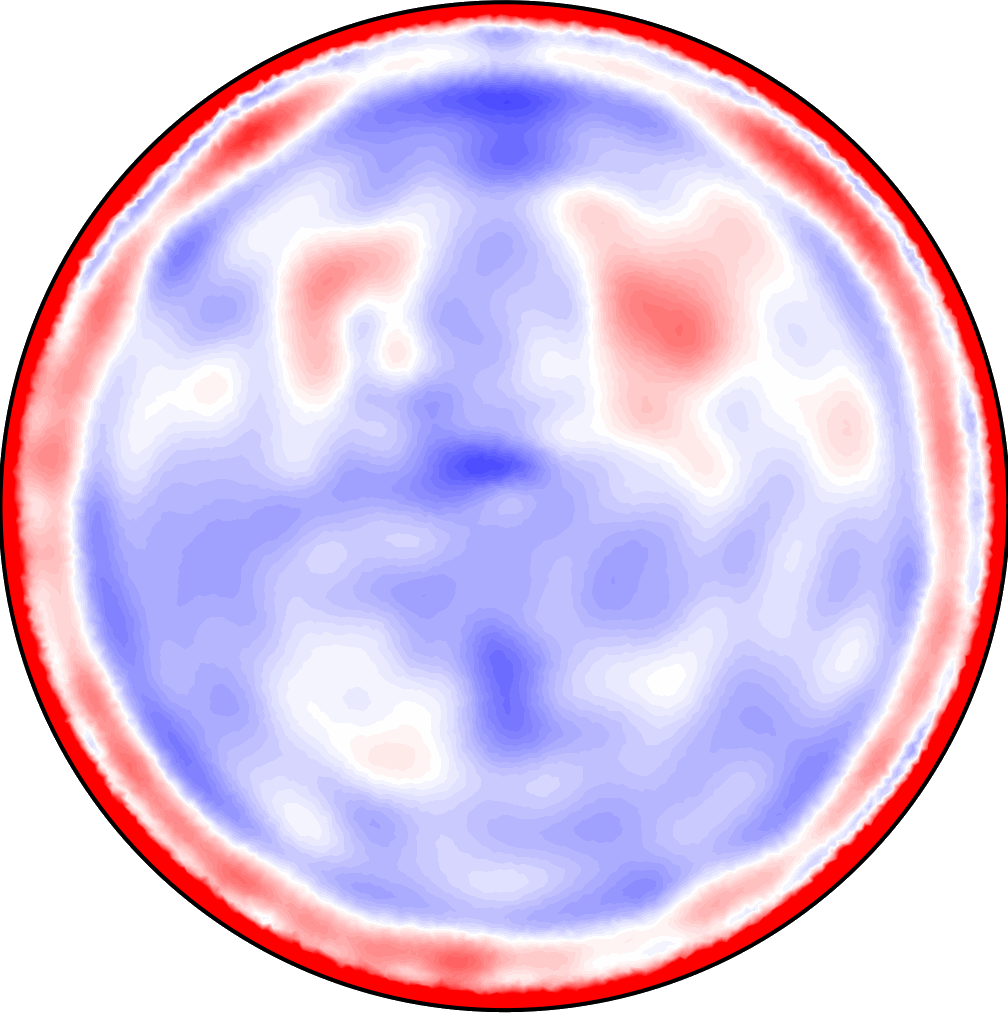}~%
    \gfx{height=0.26\linewidth}{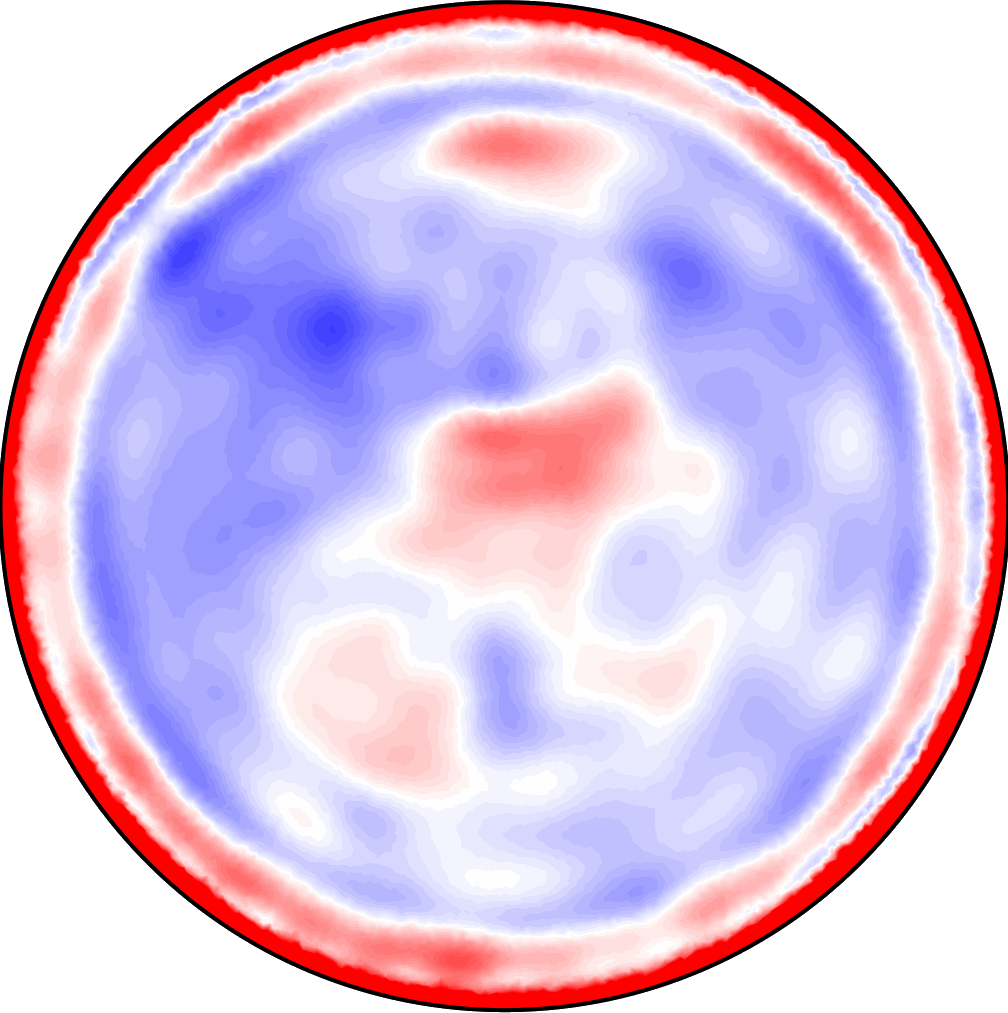}~%
    \gfx{height=0.26\linewidth}{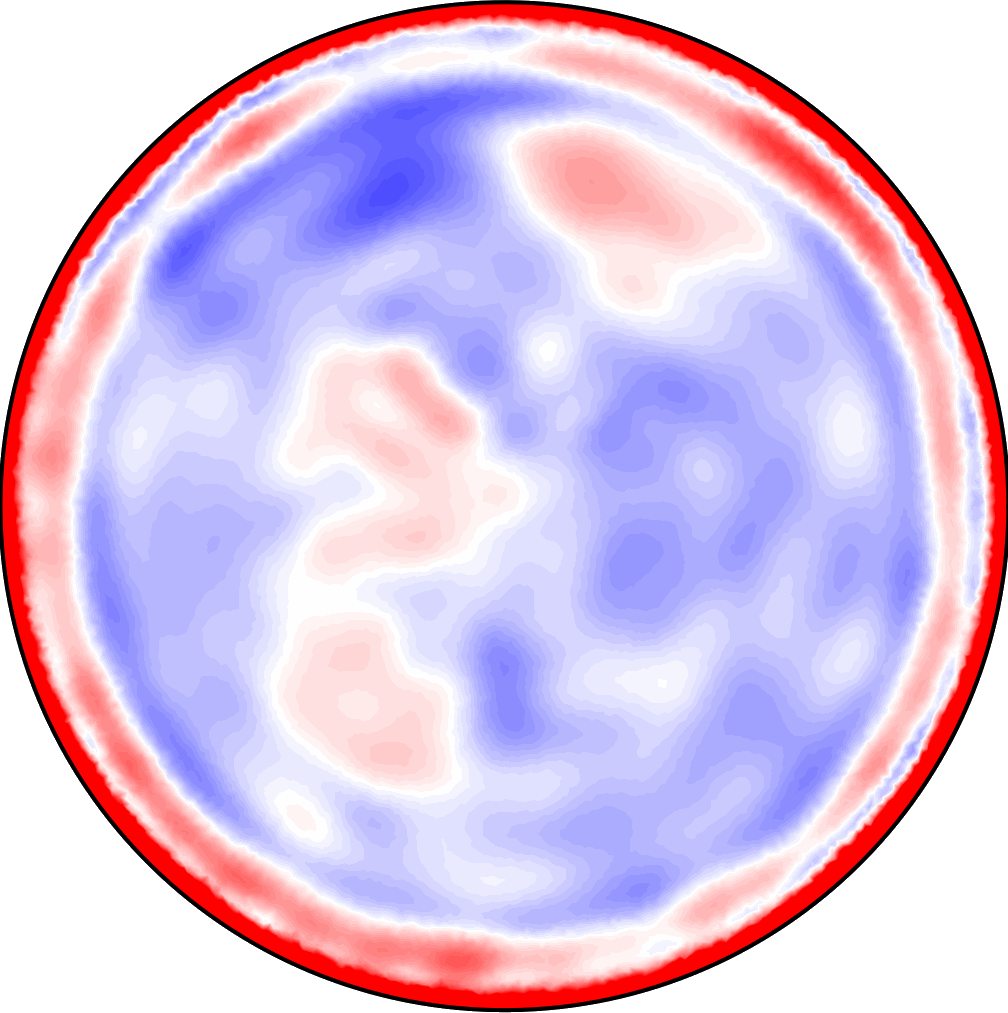}~~%
    \gfx{height=0.26\linewidth}{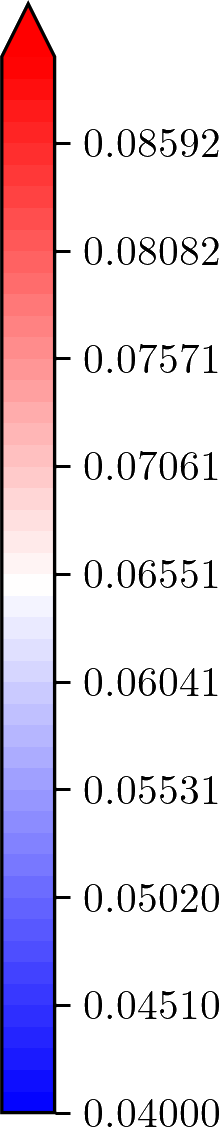}\\
    \gfx{height=0.26\linewidth}{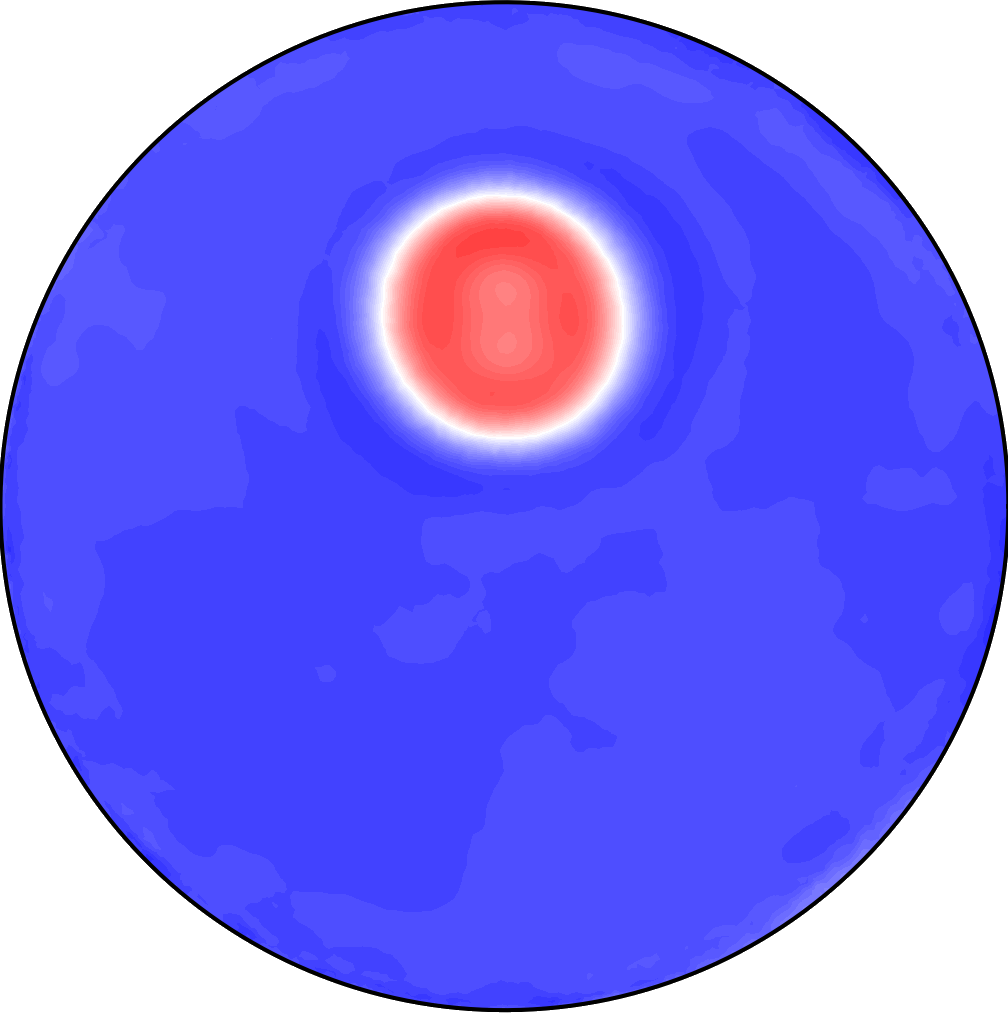}~~%
    \gfx{height=0.26\linewidth}{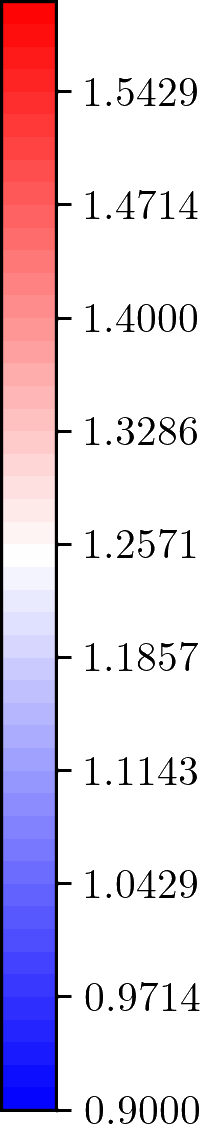}\hspace{1cm}%
    \gfx{height=0.26\linewidth}{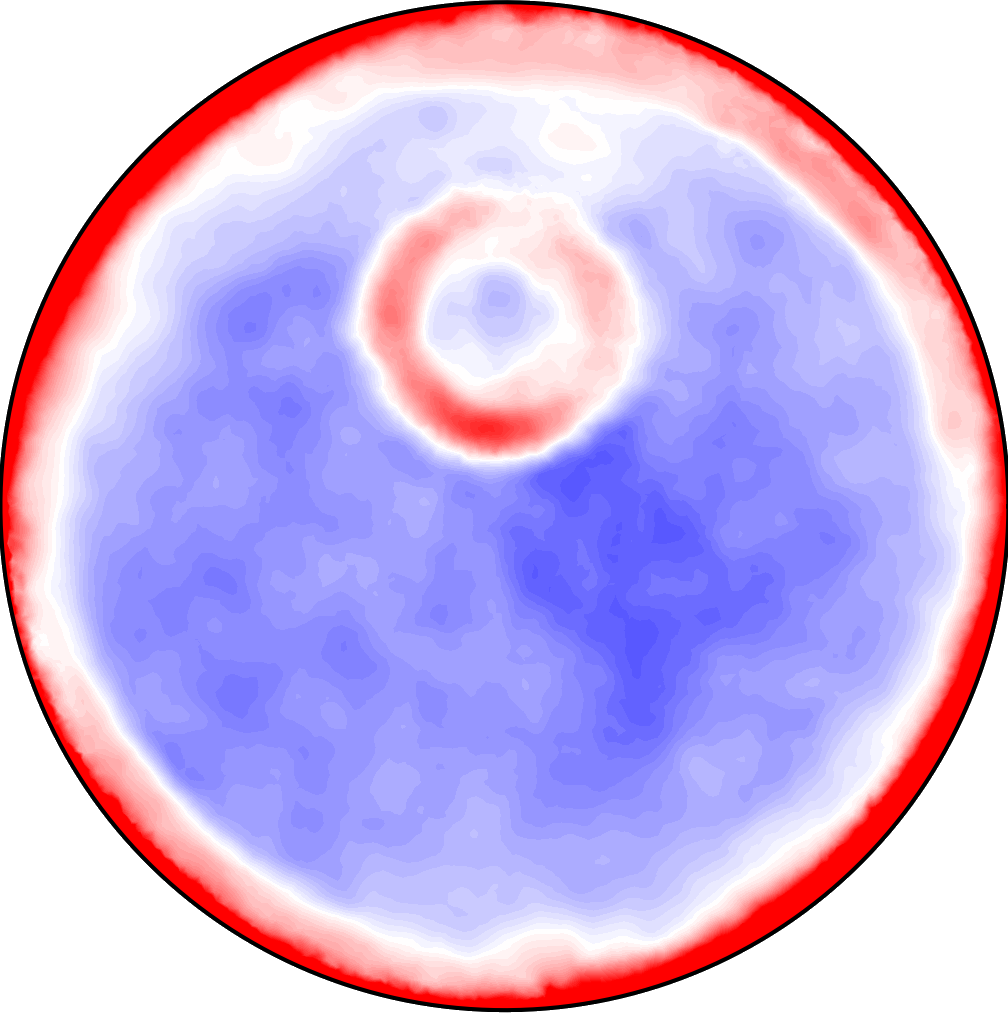}~~%
    \gfx{height=0.26\linewidth}{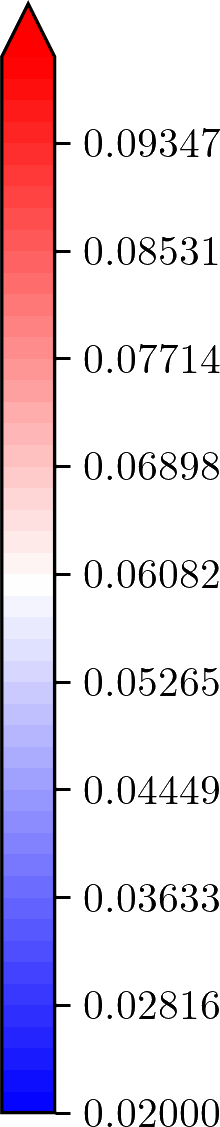}
    \caption{Node-wise means and standard deviations of reconstructions from different samples; $ \mu = 0.05 $. The upper row contains, left to right, the mean power densities for the boundary conditions $ f(x_1,x_2) = x_1$, $ x_2 $, $ (x_1+x_2)/\sqrt{2} $ respectively. The second row contains the corresponding standard deviations, and finally the last row has the mean conductivity on the left and standard deviation on the right. 150 samples were used.} \label{fig:mean-std}
\end{figure}

We present in Figure \ref{fig:mean-std} the mean and standard deviations, taken point-wise, of the power density and conductivity reconstructions from multiple different realizations of $ s(x;\theta) $. The means illustrate that our sampling does not appear to be biased in promoting non-existing features, which is to be expected but also nice to verify. The standard deviations also highlight some expected features, for example it sky-rockets near the boundary, which is in correspondence with our observations from the test cases. This is due to the increasing misalignment between the true wave $ p $ and the wave $ \widetilde{p}$  used for reconstruction as time increases; we refer again to Figure \ref{fig:artifacts}.

In a simplified scenario, if the true sound speed is slightly higher, $ p $ will have left the domain before $ \widetilde{p} $ and the reconstruction algorithm will see no signal as $ \widetilde{p} $ traverses the last stretch towards the boundary. Thus, the reconstructed power density will exhibit a low-valued zone. On the other hand, if the true sound speed is slightly lower, a reverse phenomenon happens and the reconstructed power density exhibits high values in the area. This phenomenon is very obvious when the error in $ \widetilde{p} $ is exacerbated by deliberately shifting the used sound speed by some constant.

Also, the boundary of the conductivity inclusion exhibits moderately higher standard deviation, which is not surprising when looking at the conductivity reconstructions in Figure \ref{fig:reconstructions}; here the inclusion boundary shifts a bit from one to the other. We note that 150 sampled structures $ s(x;\theta) $ were used for computing the means and standard deviations.

\section{Discussion and conclusion}
As proposed, in this paper we have explored the effect in AET of uncertainty in the sound speed and  the wave propagation. In Propositions \ref{thm:reg:q} and \ref{prop:op-cont} we have established bounds for the deviation of the wave and the operator based on variations in the sound speed coefficient. This guarantees that, for sufficiently small variations in the sound speed, the deviations scale accordingly. Considering how the error in the wave propagates to the operator, we demonstrated in Theorem \ref{thm:reg-strat} a regularization strategy for model error and used the formerly established propositions to show in Corollary \ref{cor:reg-strat} the applicability to the \aet{} problem.

In numerical simulations we found that the theory matches observations quite nicely, in particular as illustrated in Figure \ref{fig:var-lvl}, where we observe directly how the reconstruction improves as the difference between the real and estimated waves decreases. It is of course important to point out that, as mentioned in Remark \ref{rem:H-opt-beta}, in this study we perform a search for a good $ \beta $ value. 
In truth, this is a bit artificial, but it is necessary to be able to fairly compare the different end-results. Also, while we do not give the optimal values for $ \beta $ here, for the reconstructions in Figure \ref{fig:var-lvl} they start at about $ 10^{-5} $ (right-most) and decrease with $ \mu $ towards zero; the leftmost being about $ 10^{-8} $.

We saw in Figure \ref{fig:reconstructions} how the variations in realizations of the sound speed can influence the reconstructed power density quite heavily, propagating the operator error to the reconstruction. It is, however, notable how little of this error continues into the reconstructed conductivity, where the inclusion is quite convincingly reconstructed. We conjecture that this kind of error in the reconstructed power density is not too significant, probably because the features across all the power densities corresponding to the different boundary conditions are not collectively producible by the model $ H[\sigma] $  from a single choice of conductivity.

We believe that these results are quite promising for \aet{} showing that, even if the wave is not known exactly, this will not necessarily pose a huge problem for the final recovery of the conductivity, though it might produce certain irregular and artificial structures in the reconstructed power densities. We expect that the obtained results may carry over to other hybrid tomography problems such as Photo-Acoustic tomography.


\section*{Acknowledgement}
KK was supported by The Villum Foundation (grant no. 25893). BCSJ was supported by the Academy of Finland (grant no. 320022).


\bibliographystyle{siam}
\bibliography{bibliography}

\end{document}